\documentclass[12pt,reqno]{amsart}
\usepackage{amsmath,amssymb,amsfonts,amscd,latexsym,amsthm,mathrsfs,verbatim,textcomp}
\usepackage[colorlinks,linkcolor=black,citecolor=black]{hyperref}
\usepackage{cite}
\usepackage{hypbmsec}
\usepackage{enumerate}
\usepackage{bm}
\usepackage{tikz}
\usepackage{mathtools}
\usepackage{verbatim}
\usepackage{xfrac, mathabx}
\theoremstyle{plain}
\usepackage{manfnt}
\usepackage{tikz}
\usetikzlibrary{matrix,arrows,decorations.pathmorphing}
\usepackage[margin=1.1in]{geometry}
\usepackage{csquotes}
\newtheorem{lemma}{Lemma}
\newtheorem{theorem}{Theorem}
\newtheorem{corollary}{Corollary}
\newtheorem{prop}{Proposition}
\theoremstyle{remark}
\newtheorem{remark}{\bf Remark}
\newtheorem*{remark*}{\bf Remark}

\newtheorem{conjecture}{\bf Conjecture}
\newtheorem*{hypothesis*}{\bf Hypothesis}

\newtheorem{definition}{\bf Definition}

\newtheorem*{prob*}{\bf Problem}

\newtheorem*{quest*}{\bf Question}

\renewcommand{\Re}{\operatorname{Re}}

\usepackage{etoolbox}
\patchcmd{\section}{\scshape}{\bfseries}{}{}
\makeatletter
\renewcommand{\@secnumfont}{\bfseries}
\makeatother
\numberwithin{equation}{section}
\numberwithin{definition}{section}
\numberwithin{corollary}{section}
\numberwithin{lemma}{section}
\numberwithin{theorem}{section}
\numberwithin{prop}{section}
\numberwithin{remark}{section}
\renewcommand{\pmatrix}[4]{\left(\begin{smallmatrix}#1 & #2 \\ #3 & #4\end{smallmatrix}\right)}

\makeatletter
\@namedef{subjclassname@2020}{\textup{2020} Mathematics Subject Classification}
\makeatother

\newcommand\cube{\begin{tikzpicture}[scale=2.3]
    \coordinate (A1) at (0, 0);
    \coordinate (A2) at (0, 0.1);
    \coordinate (A3) at (0.1, 0.1);
    \coordinate (A4) at (0.1, 0);
    \coordinate (B1) at (0.03, 0.03);
    \coordinate (B2) at (0.03, 0.13);
    \coordinate (B3) at (0.13, 0.13);
    \coordinate (B4) at (0.13, 0.03);

    \draw (A1) -- (A2);
    \draw (A2) -- (A3);
    \draw (A3) -- (A4);
    \draw (A4) -- (A1);
    \draw[densely dotted] (A1) -- (B1);
    \draw[densely dotted] (B1) -- (B2);
    \draw (A2) -- (B2);
    \draw (B2) -- (B3);
    \draw (A3) -- (B3);
    \draw (A4) -- (B4);
    \draw (B4) -- (B3);
    \draw[densely dotted] (B1) -- (B4);
\end{tikzpicture}}

\newcommand\tinycube{\begin{tikzpicture}[scale=1.3]
    \coordinate (A1) at (0, 0);
    \coordinate (A2) at (0, 0.1);
    \coordinate (A3) at (0.1, 0.1);
    \coordinate (A4) at (0.1, 0);
    \coordinate (B1) at (0.03, 0.03);
    \coordinate (B2) at (0.03, 0.13);
    \coordinate (B3) at (0.13, 0.13);
    \coordinate (B4) at (0.13, 0.03);

    \draw (A1) -- (A2);
    \draw (A2) -- (A3);
    \draw (A3) -- (A4);
    \draw (A4) -- (A1);
    \draw[densely dotted] (A1) -- (B1);
    \draw[densely dotted] (B1) -- (B2);
    \draw (A2) -- (B2);
    \draw (B2) -- (B3);
    \draw (A3) -- (B3);
    \draw (A4) -- (B4);
    \draw (B4) -- (B3);
    \draw[densely dotted] (B1) -- (B4);
\end{tikzpicture}}

\begin{document}

\author{Alexander Dunn}
\address{School of Mathematics, Georgia Institute of Technology, 
Atlanta, USA}
\email{adunn61@gatech.edu}

\author{Maksym Radziwi\l\l}
\address{Department of Mathematics, Northwestern University  
Evanston, USA}
\email{maksym.radziwill@gmail.com}

\title[Bias in cubic Gauss sums]{Bias in cubic Gauss sums: Patterson's conjecture}
\maketitle 

\begin{abstract}
Let $W$ be a smooth test function with compact support in $(0,\infty)$.
Conditional on the Generalized Riemann Hypothesis for Hecke $L$-functions over $\mathbb{Q}(\omega)$, we prove that 
  $$
  \sum_{p \equiv 1 \pmod{3}} \frac{1}{2 \sqrt{p}} \cdot \Big ( \sum_{x \pmod{p}} e^{2\pi i x^3 / p} \Big ) W \Big ( \frac{p}{X} \Big ) \sim \frac{(2\pi)^{2/3}}{3 \Gamma(\tfrac 23)} \int_{0}^{\infty} W(x) x^{-1/6} dx \cdot \frac{X^{5/6}}{\log X},
  $$
  as $X \rightarrow \infty$ and $p$ runs over primes. This explains a well-known numerical bias in the distribution of cubic Gauss sums first observed by Kummer in 1846 and confirms (conditionally on the Generalized Riemann Hypothesis) a conjecture of Patterson \cite{Pat3} from 1978.

There are two important byproducts of our proof. The first is an explicit level aspect Voronoi summation formula for cubic Gauss sums, extending computations of Patterson and Yoshimoto. Secondly, we show that 
Heath-Brown's cubic large sieve is sharp up to factors of $X^{o(1)}$ under the Generalized Riemann Hypothesis.
This disproves the popular belief that the cubic large sieve can be improved.

An important ingredient in our proof is a dispersion estimate for cubic Gauss sums. 
It can be interpreted as a cubic large sieve with correction by a non-trivial asymptotic main term. 
This estimate relies on the Generalized Riemann Hypothesis, and
is one of the fundamental reasons why our result is conditional. 
\end{abstract}

\tableofcontents

\section{Introduction} \label{prelim}
\subsection{Exponential sums over primes}
 Kummer \cite[Paper 16,17]{Kummer} studied the distribution of the cubic exponential sums
$$
S_{p} = \sum_{n = 1}^{p} e \Big ( \frac{n^3}{p} \Big ), \quad e(x) := e^{2\pi i x},
$$
with $p \equiv 1 \pmod{3}$ prime.
The bound $|S_{p}| \leq 2 \sqrt{p}$ is well-known, and we can consequently write
\begin{equation} \label{eq:bb1}
\frac{S_p}{2 \sqrt{p}} = \cos(2\pi \theta_{p}), \quad \theta_{p} \in [0,1].
\end{equation}
This specifies the value of $\theta_{p} - \tfrac 12$ up to sign. This sign ambiguity can be resolved by noticing that \eqref{eq:bb1} is the real part of an explicit root of unity defined in \eqref{normgauss}.
To probe whether $\theta_{p}$ is equidistributed, Kummer computed the frequency with which $\cos(2 \pi \theta_p)$ lay in the intervals $I_1 = [\tfrac 12, 1]$, $I_2 = [-\tfrac 12, \tfrac 12]$ and $I_3 = [-1, - \tfrac 12]$, for $p \leq 500$. Kummer observed that $\cos(2 \pi \theta_p)$
tended to lay more frequently in $I_1$ than in $I_2$ or $I_3$ (the ratio he observed was $3:2:1$ respectively). If this bias persisted, then the angles $\theta_{p}$ \textit{would not be} uniformly distributed. 
Subsequent calculations by von Neumann-Goldstine \cite{Neumann}, Lehmer \cite{Lehmer} and Cassels \cite{Cassels} cast doubt on the persistence of this observation and suggested that $\cos(2 \pi \theta_p)$ lay equally frequently in $I_1$, $I_2$ and $I_3$, and that $\theta_{p}$ was uniformly distributed. 
In light of the new numerical evidence, Patterson \cite{Pat3} enunciated a corrected conjecture. This conjecture explained the bias observed by Kummer, and was consistent with the numerical data.

\begin{conjecture} [Patterson, 1978]
  As $X \rightarrow \infty$, 
  $$
  \sum_{\substack{p \leq X \\ p \equiv 1 \pmod{3}}} \frac{S_{p}}{2 \sqrt{p}} \sim \frac{2 (2\pi)^{2/3}}{5 \Gamma(\tfrac 23)} \cdot \frac{X^{5/6}}{\log X},
  $$
  where $p$ runs through primes.
\end{conjecture}

Patterson obtained this conjecture by developing Kubota's theory of metaplectic forms \cite{Kubota1, Kubota2}, and by appealing to a heuristic form of the circle method 
\cite{Pat3}. Unfortunately, even under the assumption of the Generalized Riemann Hypothesis,
Patterson's heuristic fell short of a proof. This was due to insufficient bounds for the minor arcs. 
Subsequently, in 1979, Heath-Brown and Patterson \cite{HBP} established that $\theta_p$ is uniformly distributed in $[0,1]$ as $p$ varies among primes congruent to $1$ modulo $3$. This decisively disproved Kummer's guess. A nice summary up to this point can be found in a standard text of Davenport \cite[Chap.~3]{Dav}.
Some 20 years later, in 2000, Heath-Brown \cite{HB} sharpened his earlier result with Patterson and obtained unconditionally the nearly tight upper bound
\begin{equation} \label{eq:hb}
\sum_{\substack{p \leq X \\ p \equiv 1 \pmod{3}}} \frac{S_{p}}{2 \sqrt{p}} \ll_{\varepsilon} X^{5/6 + \varepsilon},
\end{equation}
for any given $\varepsilon > 0$.
Heath-Brown \cite[pg.~99]{HB} also stated a refined form of Patterson's conjecture that features an error term capturing square root cancellation\footnote{The constant in Patterson's conjecture appearing in \cite{HB} is mistated due to a misprint}.

In this paper we confirm Patterson's conjecture, conditionally on the assumption of the Generalized Riemann Hypothesis. This can be also viewed as a conditional sharpening of \eqref{eq:hb}. We will explain in a later part of the introduction why the assumption of the Riemann Hypothesis (or similar unproven hypothesis) appears to be necessary at this point. 

\begin{theorem} \label{thm:main}
  Assume the Generalized Riemann Hypothesis for Hecke $L$-functions over $\mathbb{Q}(\omega)$. 
  Let $W$ be a smooth function that is compactly supported in $(0, \infty)$. Then as $X \rightarrow \infty$ we have
  $$
  \sum_{p \equiv 1 \pmod{3}} \frac{S_{p}}{2 \sqrt{p}} \cdot W \Big ( \frac{p}{X} \Big ) \sim \frac{(2\pi)^{2/3}}{3 \Gamma(\tfrac 23)} \int_{0}^{\infty} W(x) x^{-1/6} dx \cdot \frac{X^{5/6}}{\log X}, 
  $$
  where $p$ runs through primes.
\end{theorem}

Notice that the constant that we get is consistent with Patterson's \cite{Pat3} prediction: if $W(x) \rightarrow \mathbf{1}_{[0,1]}(x)$ then,
$$
\frac{(2\pi)^{2/3}}{3 \Gamma(\tfrac 23)} \int_{0}^{\infty} W(x) x^{-1/6} dx \rightarrow \frac{2 (2\pi)^{2/3}}{5\Gamma(\tfrac 23)}.
$$
Theorem \ref{thm:main} shows that the angles $\theta_{p}$ cannot be equidistributed with square-root cancellation in the error term. We make this precise in Theorem \ref{thm:cor}. 

\begin{theorem} \label{thm:cor}
  Assume the Generalized Riemann Hypothesis for Hecke $L$-functions over $\mathbb{Q}(\omega)$.
  Let $f$ be a smooth $1$-periodic function and $W$ be a smooth function that is compactly supported in $(0, \infty)$. Then we have
  \begin{align} \label{eq:bdui}
    \sum_{p \equiv 1 \pmod{3}} & f (\theta_{p}) W \Big ( \frac{p}{X} \Big ) = \int_{0}^{1} f(x) dx \sum_{p \equiv 1 \pmod{3}} W \Big ( \frac{p}{X} \Big ) \\ \nonumber & + 2\int_{0}^{1} f(x) \cos(2\pi x) dx \cdot \frac{(2\pi)^{2/3}}{3 \Gamma(\tfrac 23)} \int_{0}^{\infty} W(x) x^{-1/6} dx \cdot \frac{X^{5/6}}{\log X}  + o \Big ( \frac{X^{5/6}}{\log X} \Big ),
  \end{align}
 as $X \rightarrow \infty$. 
\end{theorem}
It is unlikely that  \eqref{eq:bdui} can be established unconditionally given the current state of knowledge. For instance, with the choice $f(x) = e(3 \ell x)$, $\ell \neq 0$,  \eqref{eq:bdui} implies a zero-free strip for certain $L$-functions associated to Gr{\"o}{\ss}encharaktern. 

Before proceeding to a high level sketch of the ideas in the paper, we remark that it is possible to slightly sharpen the rate of convergence in Theorems \ref{thm:main} and \ref{thm:cor}. One can save roughly one power of $\log X$. 

\subsection{Gauss sums over Eisenstein integers}

Let $\omega = e^{2\pi i / 3}$ and let $\mathbb{Z}[\omega]$ denote the ring of Eisenstein integers (in $\mathbb{Q}(\omega)$). It is well known that any non-zero element of $\mathbb{Z}[\omega]$ can be uniquely written as $\zeta \lambda^i c$ with $\zeta \in \{\pm 1, \pm \omega, \pm \omega^2\}$ a unit, $\lambda = \sqrt{-3} = 1 + 2 \omega$ the unique ramified prime in $\mathbb{Z}[\omega]$, $i
\in \mathbb{Z}_{\geq 0}$, and $c \in \mathbb{Z}[\omega]$ satisfying $c \equiv 1 \pmod{3}$. Furthermore, we have a cubic symbol $\big( \frac{a}{\varpi} \big)_3$ defined for $a \equiv 1 \pmod{3}$ and $\varpi \equiv 1 \pmod{3}$ prime. This is defined by 
$$
\Big ( \frac{a}{\varpi} \Big )_{3} \equiv a^{(N(\varpi) - 1) / 3} \pmod{\varpi},
$$
and the condition that it takes values in $\{1,\omega,\omega^2 \}$.
It is clearly multiplicative in $a$ and can be extended to a multiplicative function in $\varpi$ by setting
$$
\big ( \frac{a}{b} \big )_{3} = \prod_{i} \big ( \frac{a}{\varpi_i} \big )
$$
for any $b \equiv 1 \pmod{3}$ with $b = \prod_{i} \varpi_i$ and $\varpi_i$ primes. 
The cubic symbol obeys cubic reciprocity: given $a,b \equiv 1 \pmod{3}$ 
we have 
\begin{equation} \label{cuberecip}
\Big( \frac{a}{b} \Big)_3=\Big( \frac{b}{a} \Big)_3.
\end{equation}
We also have supplementary laws for units and the ramified prime.
Given $d  \equiv 1 + \alpha_2 \lambda^2 + \alpha_3 \lambda^3 \pmod{9}$ with $\alpha_2, \alpha_3 \in \{-1,0,1\}$,
then
\begin{equation} \label{supprecip}
\Big ( \frac{\omega}{d} \Big )_{3} = \omega^{\alpha_2} \quad \text{and} \quad \quad \Big ( \frac{\lambda}{d} \Big )_{3} = \omega^{-\alpha_3}.
\end{equation}

The cubic exponential sums $S_p$ are intimately connected to cubic Gauss sums over Eisenstein integers. 
For any rational prime $p \equiv 1 \pmod{3}$, we can write $p = \varpi \overline{\varpi}$ with $\varpi$ prime in $\mathbb{Z}[\omega]$. Then
\begin{equation} \label{eq:Asd}
\frac{S_p}{2 \sqrt{p}} = \Re \tilde{g}(\varpi),
\end{equation}
where the normalised Gauss sum is given by
\begin{equation} \label{normgauss}
\tilde{g}(c) = \frac{1}{|c|} \sum_{x \pmod{c}} \Big ( \frac{x}{c} \Big )_{3} \check{e} \Big ( \frac{x}{c} \Big ), \quad \check{e}(z) = e^{2\pi i (z + \overline{z})},
\end{equation}
for any $c \in \mathbb{Z}[\omega]$ with $c \equiv 1 \pmod{3}$.
Here $|\cdot|$ denotes the Euclidean distance of $c$ from the origin.
We write $g(c)$ for the unnormalized Gauss sum, namely $g(c) := |c| \tilde{g}(c)$.
We also note that 
\begin{equation} \label{eq:bb}
\tilde{g}(c)^3 = \mu(c) \frac{c}{|c|}. 
\end{equation}
Thus $\tilde{g}(c)$ is a cube root of $\mu(c) c / |c|$ (see \cite[pp. 443--445]{H}). However, given a prime $\varpi \equiv 1 \pmod{3}$, there is no known formula efficiently predicting which cube root $\tilde{g}(\varpi)$ corresponds to\footnote{We note that the work of Matthews \cite{Matthews} gives an explicit formula expressing $\tilde{g}(\varpi)$ as a product of the Weierstrass $\wp$-functions evaluated at $(N(\pi) - 1)/3$ values. Despite the beauty of Matthews' formula, it is not computationally efficient.}.

Formula \eqref{eq:Asd} shows that Patterson's conjecture is equivalent to the statement
$$
\sum_{\substack{\substack{N(\varpi) \leq X} \\ \varpi \text{ prime} \\ \varpi \equiv 1 \pmod{3}}} \tilde{g}(\varpi) \sim \frac{2 (2\pi)^{2/3}}{5 \Gamma(\tfrac 23)} \cdot \frac{X^{5/6}}{\log X}. 
$$
From the point of view of Weyl's equidistribution criterion it is also natural to ask about the behavior of $\tilde{g}(\varpi)^{k}$ with $k \in \mathbb{Z}$.
Patterson enunciated in \cite{Pat3} a complementary conjecture. It states that for all $k \not \in \{0,1,-1\}$ we have,
\begin{equation} \label{eq:ddd}
\sum_{\substack{N(\varpi) \leq X \\ \varpi \text{ prime} \\ \varpi \equiv 1 \pmod{3}}} \tilde{g}(\varpi)^k = o \Big ( \frac{X^{5/6}}{\log X} \Big )
\end{equation}
as $X \rightarrow \infty$. 
We conditionally establish a version of this conjecture with wide uniformity in $k$.
\begin{theorem} \label{thm:main3}
  Assume the Generalized Riemann Hypothesis for Hecke $L$-functions over $\mathbb{Q}(\omega)$. 
  Let $W$ be a smooth function compactly supported in $(0, \infty)$. Then as $X \rightarrow \infty$ we have 
  $$
  \sum_{\substack{\varpi \equiv 1 \pmod{3} \\ \varpi \emph{ prime}}} \tilde{g}(\varpi)^{k} \cdot W \Big ( \frac{N(\varpi)}{X} \Big )  = o \Big ( \frac{X^{5/6}}{\log X} \Big ),
  $$
   uniformly in $1 < |k| \leq X^{1/100}$
\end{theorem}
Theorem \ref{thm:cor} is a nearly immediate consequence of Theorem \ref{thm:main} and Theorem \ref{thm:main3}. 
Notice that in the case $ k \equiv 0 \pmod{3}$ ($k \neq 0$), Theorem \ref{thm:main3} unambiguously requires a zero-free strip for $L$-functions associated to
Gr{\"o}{\ss}encharaktern. 

\subsection{Cubic Gauss sums and automorphic forms}
Developing Kubota's theory \cite{Kubota1, Kubota2}, Patterson \cite{Pat1} established a functional equation for a Dirichlet series of the form
$$
\sum_{c \equiv 1 \pmod{3}} \frac{\tilde{g}(\mu, c)}{N(c)^s},
$$
where
$$
\tilde{g}(\mu, c) := \frac{1}{|c|} \sum_{x \pmod{c}} \Big ( \frac{x}{c} \Big )_{3} \check{e} \Big ( \frac{\mu x}{c} \Big ).
$$
Subsequently, Yoshimoto \cite{Yos} followed Patterson's approach to obtain a functional equation for the Dirichlet series
$$
\sum_{c \equiv 1 \pmod{3}} \frac{\tilde{g}(c) \psi(c)}{N(c)^s},
$$
where $\psi$ is a primitive Dirichlet character such that $\psi^3$ is not principal. Yoshimoto specifically excludes the case when $\psi^3$ is principal to prevent the (Kubota) multiplier from interfering with $\psi$.  We develop both of these computations further, obtaining a functional equation for the Dirichlet series
$$
\sum_{c \equiv 1 \pmod{3}} \frac{\tilde{g}(c) \phi(c)}{N(c)^s}
$$
and $\phi$ a periodic function modulo $r$ with $r \equiv 1 \pmod{3}$. We specialise our computation to the case when $\phi$ is the conjugate of a cubic character to modulus $r$. The result for this specific choice of $\phi$ could have been obtained more directly by combining \cite[Theorem 6.1]{Pat1} and \cite[Lemma 4.1]{HB}. However, we found it advantageous to develop a more general approach. First, we believe that the result will be useful in later works. Second, this more general approach improved our understanding of (and confidence in) the formula. Third, our functional equations explicate the root number. These formulas are too lengthy to be introduced here. We refer the reader to Section \ref{sec:voronoi} where they are stated in detail. 

As in earlier works \cite{HB, HBP}, this Voronoi formula is used to understand the so-called Type-I sums
\begin{equation} \label{type1}
\sum_{c, r \equiv 1 \pmod{3}} \tilde{g}(c r) \alpha_r V \Big ( \frac{N(r)}{R} \Big ) V \Big ( \frac{N(c)}{C} \Big ), \quad C R = X,
\end{equation}
with $\alpha_r$ arbitrary coefficients bounded in absolute value by $1$. 
A sharp bound for \eqref{type1} in the range $C > R^2$ was established in \cite{HB}. In the proof of Theorem \ref{thm:main} we need an asymptotic slightly past this range (with an error term $\ll X^{5/6 - \varepsilon}$). In Section \ref{type1sec} we use the Generalized Riemann Hypothesis (to cancel out the contribution of cubes) to obtain adequate pointwise (for a single value of $r$) Type-I information as long as $C > N(r)^{1 + \varepsilon}$, for any given $\varepsilon > 0$.
We also give alternative estimates in Section \ref{avgtype1sec} that use the averaging over $r$ in a non-trivial way: we obtain adequate Type-I information on average in the range $C > R^{2-\varepsilon}$, under the Generalized Riemann Hypothesis. For the interested reader we note that there are two ways of bypassing the Riemann Hypothesis in this case. One is to assume that the sequence $\alpha_r$ has a bilinear structure. The second would be to obtain ``subconvex'' bounds in the $r$ aspect for the Dirichlet series $\sum_{c \equiv 1 \pmod{3}} \tilde{g}(c r) N(c)^{-s}$. Since a more significant bottleneck appears elsewhere we have not endeavoured  to make these results unconditional. 

\subsection{Cubic Gauss sums and the cubic large sieve}
In order to obtain the bound
$$
\sum_{\substack{N(\varpi) \leq X \\ \varpi \text{ prime} \\ \varpi \equiv 1 \pmod{3}}} \tilde{g}(\varpi) \ll_{\varepsilon} X^{5/6 + \varepsilon},
$$
Heath-Brown develops in \cite[Theorem~2]{HB} the so-called ``cubic large sieve''. The cubic-large sieve states that for any sequence $\beta_b$ supported on squarefree $b \in \mathbb{Z}[\omega]$, 
\begin{equation} \label{HBinitstate}
\sum_{\substack{N(a) \leq A \\ a \equiv 1 \pmod{3}}} \mu^2(a) \Big | \sum_{\substack{N(b) \leq B \\ b \equiv 1 \pmod{3}}} \beta_b \Big( \frac{b}{a} \Big )_{3} \Big |^2 \ll_{\varepsilon} (A B)^{\varepsilon} \cdot ( A + B + (AB)^{2/3} ) \sum_{\substack{ N(b) \leq B \\ b \equiv 1 \pmod{3} }} |\beta_{b}|^2.
\end{equation}
Immediately after stating \eqref{HBinitstate} in \cite{HB}, Heath-Brown writes:
\begin{displayquote}
  ``It seems possible that the term $(AB)^{2/3}$ could be removed with further effort, and the bound would then be essentially best possible. However, the above suffices for our purposes. It should be noted that if the variables are not restricted to be squarefree, a result as sharp as Theorem 2 would be impossible. The proof of Theorem 2 is modelled on the corresponding argument for sums (over $\mathbb{Z}$) containing the quadratic residue symbol, due to the author [3] (local cit. \cite{HBquad}). The latter is distinctly unpleasant, but fortunately some of the difficulties may be reduced in our situation by the introduction of the term $(AB)^{2/3}$ in Theorem 2.
  ''\end{displayquote}
This assertion that the term $(AB)^{2/3}$ can be removed is then frequently repeated in subsequent literature. For example, in \cite{Blomer} it is asserted that
\begin{displayquote}
 ``As in [12] (local cit. \cite{HB}), the term $(AB)^{2/3}$ is not optimal and can most likely be replaced with $(AB)^{1/2}$.''
\end{displayquote}

To our great surprise we found that the term $(AB)^{2/3}$ in Heath-Brown's cubic large sieve \emph{can't} be removed.
We state our optimality result in terms of operator norms.
For $A,B \geq 10$, and $(\beta_b)_{b \in \mathbb{Z}[\omega]}$ an arbitrary sequence of complex numbers
with support contained in the set of squarefree elements of $\mathbb{Z}[\omega]$, let
\begin{equation*}
\Sigma(A,B,\boldsymbol{\beta}):=\sum_{\substack{A<N(a) \leq 2A \\ a \equiv 1 \pmod{3} }} \mu^2(a)  \Big | \sum_{ \substack{B<N(b) \leq 2B \\  b \equiv 1 \pmod{3} } } \beta_b \Big( \frac{b}{a} \Big)_3   \Big|^2. 
\end{equation*}
For $A,B \geq 1$, consider the operator norm 
\begin{equation} \label{opnorm}
\mathcal{B}(A,B):=\sup_{\boldsymbol{\beta}} \Big \{ \Sigma(A,B,\boldsymbol{\beta}) : \sum_b |\beta_b|^2=1 \Big \}.
\end{equation}
\begin{theorem} \label{hboptimal}
Assume the Generalized Riemann Hypothesis for Hecke $L$-functions over $\mathbb{Q}(\omega)$.
Let $A,B \geq 10$, then for any $\varepsilon>0$ we have 
\begin{equation*}
\begin{cases}
(A B)^{2/3}  \ll_{\varepsilon} \mathcal{B}(A,B) \ll_{\varepsilon} (A B)^{2/3 + \varepsilon} & \text{ for } A \in [\sqrt{B}, B^2] \setminus [B^{1 - \varepsilon}, B^{1 + 2 \varepsilon}], \\
(A B)^{2/3 - 3 \varepsilon}  \ll_{\varepsilon} \mathcal{B}(A, B) \ll_{\varepsilon} (A B)^{2/3 + \varepsilon} & \text{ for } A \in [B^{1 - \varepsilon}, B^{1 + 2 \varepsilon}].
\end{cases}
\end{equation*}
\end{theorem}

\begin{remark}
Only the lower bounds in Theorem \ref{hboptimal} are conditional on the Generalized Riemann Hypothesis for Hecke $L$-functions over $\mathbb{Q}(\omega)$.
The upper bounds follow unconditionally from Heath-Brown's cubic large sieve, and are included for the sake of comparison with the lower bounds.
\end{remark}

One example establishing optimality in the range $A = B^{1 - \varepsilon}$ (for any given small $\varepsilon > 0$) is
$\beta_{b} = \overline{\tilde{g}(b)}.$
This follows from applying our Voronoi summation formula in Proposition \ref{prop:typeI}, and then subsequently
using the non-trivial main term that arises when summing cubic Gauss sums over all elements of $\mathbb{Z}[\omega]$ (see Section \ref{sec:dispersion} for details). This is far from the only obstruction. Any sequence of the form $\beta_b \overline{\tilde{g}(b)}$ with $\beta_b$ non-negative and not correlated with cubic symbols would provide a counterexample. 

To address this limitation of the cubic large sieve we introduce a correction term. This allows us to beat the exponent $(AB)^{2/3}$, albeit only for sequences that have substantial cancellations against all non-principal cubic characters. We show that there exists a small fixed $\delta>0$ such that
for any sequence $\boldsymbol{\beta}$ on $\mathbb{Z}[\omega]$ satisfying 
\begin{itemize}
\item $|\beta_b| \leq 1$ for all $b \in \mathbb{Z}[\omega]$;
\item $\beta_b$ supported on square-free $w$-rough integers (i.e all prime factors of $b \in \mathbb{Z}[\omega]$ have norm $>w$);
\item$\beta_b$ supported on $b \equiv 1 \pmod{3}$ with $N(b) \asymp B$;
\end{itemize}
then
\begin{align}  \nonumber
\sum_{\substack{a \in \mathbb{Z}[\omega] \\ a \equiv 1 \pmod{3} }} & \mu^2(a) V \Big ( \frac{N(a)}{A} \Big ) \Big | \sum_{b \in \mathbb{Z}[\omega]} \beta_b \tilde{g}(b) \overline{\Big ( \frac{b}{a} \Big )_{3}} - \frac{(2\pi)^{2/3}}{3 \Gamma(\tfrac 23)} \frac{\overline{\tilde{g}(a)}}{N(a)^{1/6}} \sum_{\substack{ b \in \mathbb{Z}[\omega] \\ (b,a)=1 }} \frac{\beta_{b}}{N(b)^{1/6}} \Big |^2 \\ \label{eq:down} & \approx \frac{A}{B} \sum_{\substack{h \in \mathbb{Z}[\omega] \\ 0 < |h| \leq B^2 / A \\ h \neq \tinycube}} \Big | \sum_{b} \beta_b \Big ( \frac{b}{h} \Big )_{3} \Big |^2 + O \Big ( \frac{(A B)^{2/3} \cdot B}{w} + (AB)^{2/3 - \delta} \cdot B \Big ),   
\end{align}
where $\cube$ denotes an integer of the form $k^3$ with $k \in \mathbb{Z}[\omega]$. In particular, if the sequence $\beta_b$ exhibits square root cancellations 
against all non-trivial cubic characters and $w > (AB)^{\varepsilon}$, then \eqref{eq:down} is $\ll (AB)^{o(1)}(AB+B^2+(AB)^{2/3 - \varepsilon} \cdot B)$. 
This suggests that in order to beat the cubic large sieve, the correction term alone is not enough; we really need to know additional information about the sequence $\beta_b$. It is tempting to try to use Dirichlet polynomial techniques to bound \eqref{eq:down}. However, the \textit{optimal} term $(AB)^{2/3}$ in the cubic large sieve adds substantial technical challenges preventing us from being able to use these techniques. Precise versions of \eqref{eq:down}
are given in Propositions \ref{prop:main} and \ref{prop:adj}.

Using the above estimates we are able to show in Section \ref{broad2sec} that for a broad class of sequences we have
\begin{equation} \label{eq:disp2}
\sum_{a,b \equiv 1 \pmod{3}} \alpha_a \beta_b \tilde{g}(a b) \sim \frac{(2\pi)^{2/3}}{3 \Gamma(\tfrac 23)} \sum_{a,b \equiv 1 \pmod{3}} \frac{\alpha_a \beta_b \mu^2(a b)}{N(a b)^{1/6}}.
\end{equation}
Note that $\mu^2(a b)$ can be inserted at will since $b$ is supported on $w$-rough integers and $w > (\log X)^{10}$ is reasonably large.
It is perhaps appropriate to call \eqref{eq:disp2} a \textit{dispersion estimate}. Compared to the usual dispersion estimates we use the assumption of the Generalized Riemann Hypothesis instead of the usual Siegel-Walfisz assumption, and the condition $a b\equiv 1 \pmod{q}$ is replaced by the term $\tilde{g}(a b)$. 

The estimate \eqref{eq:disp2} will be indispensible in estimating so-called Type-II sums, which we discuss in the next section. 
Our example suggests that the $\operatorname{GL}_3$-spectral large sieve recently established by Young \cite{Young} might also be optimal. 
In the same vein, it is also interesting to note that Iwaniec and Li found unexpected main terms appearing in the spectral large sieve for $\Gamma_1(q)$
\cite{IwLi}. Other versions of the cubic large sieve have been established by Baier and Young \cite{BY} in their work on the first moment of 
Dirichlet $L$-functions (over $\mathbb{Q}$) twisted by cubic characters.

\subsection{The overall strategy of the proof}
Having explained above the main ingredients in our proof we will now explain how they are combined in Sections 
\ref{mainsec1} and \ref{mainsec2}.
It will be useful to compare the argument with \cite{HB}. In order to establish the bound
\begin{align} \label{eq:hbound}
\sum_{N (\varpi) \leq X} \tilde{g}(\varpi) \ll_{\varepsilon} X^{5/6 + \varepsilon}
\end{align}
in \cite{HB}, Heath-Brown needs to address two types of sums,
\begin{align}
  \sum_{a , b\equiv 1 \pmod{3}} \alpha_a \tilde{g}(a b) V \Big ( \frac{N(a)}{A} \Big ) V \Big ( \frac{N(b)}{B} \Big ), & \qquad A B = X; \\
  \sum_{a,b \equiv 1 \pmod{3}} \alpha_a \beta_b \tilde{g}(a b) V \Big ( \frac{N(a)}{A} \Big ) V \Big ( \frac{N(b)}{B} \Big ), & \qquad A B = X. 
\end{align}
The first sum is known as a Type-I sum and the second sum as a Type-II sum. If we are aiming for a bound of the form $X^{5/6 + \varepsilon}$ for the sum over primes then we need to bound the Type-II sums by $\ll X^{5/6 + \varepsilon}$ in the range $A > X^{1/3}$ (since we will apply Cauchy-Schwarz on the $a$-sum and thus we can hope for at most a saving of $\sqrt{A}$ in the most favourable scenario). We also need to bound the Type-I sums by $\ll X^{5/6+\varepsilon}$
in the range $A \leq X^{1/3}$. The main point is that if asymptotic estimates are not sought, then proving bounds with an $X^{\varepsilon}$-loss is sufficient.

If we aim to refine Heath-Brown's bound \eqref{eq:hbound} to an asymptotic then first we need to refine the Type-I estimate to an asymptotic. This can be done simply by a careful derivation of Voronoi summation. We also need to push the range slightly past $A > X^{1/3}$, but this does not present us with any significant difficulties under the 
Generalized Riemann Hypothesis (other than the tedium of the computations). 

Second, we need to refine Type-II estimates to an asymptotic; this is significantly more tricky. For this it is necessary to use our version of the cubic large sieve with the correction term. Since the error term needs to be smaller than $X^{5/6}$ we now need to take $A> X^{1/3 + \varepsilon}$ in the Type-II sums. This however creates a problem since the ranges in which we can handle Type-I and Type-II sums are not enough to obtain primes. In fact we now need to also bound the contribution of so-called Type-III sums of the form
$$
\sum_{a,b,c \equiv 1 \pmod{3}} \tilde{g}(a b c) V \Big ( \frac{N(a)}{A} \Big ) V \Big ( \frac{N(b)}{B} \Big ) V \Big ( \frac{N(c)}{C} \Big ), \\ \qquad A B C = X.
$$ 
with $A, B, C = X^{1/3 + O(\varepsilon)}$ and $a,b,c$ supported on primes. Bounding these sums with a power-saving presents a real challenge that we do not know how to solve. The main problem arises when $A = B = C = X^{1/3}$. In that regime, executing Voronoi summation on any single variable produces an essentially self-dual situation. Furthermore, the only admissible way of applying Cauchy-Schwarz is by grouping two variables together, and this then leads to a very long off-diagonal that appears even more difficult to handle. 

Instead, we use the Generalized Riemann Hypothesis in Section \ref{narrowtype2sec} to refine the large sieve bound to a bound that is tight up to constant factors and holds with wider uniformity than the Type-II bound stated above. As a consequence we can show that the overall contribution of these Type-III sums is only $\varepsilon$ times the expected main term. This strategy was previously used by Heath-Brown in his work on primes in short intervals \cite{HBprimeshort}.  Our work is the first to execute this idea in the context of an oscillating sequence with a main term of density $X^{-\delta}$.

Finally we note that our Type-II bound (as stated) is not able to handle the narrow range $X^{1/2 - \varepsilon} \leq A, B \leq X^{1/2 + \varepsilon}$. So instead in this range we use the same kind of ideas that we used to handle Type-III sums. This is not necessary and we could have obtained a power-saving in this range with a little bit more work. However, this wouldn't have made a significant difference so we refrained from this additional work. 

\subsection{Acknowledgements}
We warmly thank Samuel Patterson for his helpful correspondence and encouragement, as well as
Matthew Young and Scott Ahlgren for their helpful feedback on the manuscript. 
We also thank the referee for their meticulous comments on the manuscript.
M.R. was supported on 
NSF grant DMS-1902063 and a Sloan Fellowship.

\section{Basic facts about $\mathbb{Q}(\omega)$}
Let $\mathbb{Q}(\omega)$ be the Eisenstein quadratic number field,
where $\omega$ is identified with $e^{2\pi i/3} \in \mathbb{C}$.
It has ring of integers $\mathbb{Z}[\omega]$, discriminant $-3$,
and class number $1$. It also has six units $\{\pm 1, \pm \omega, \pm \omega^2\}$ and 
one ramified prime $\lambda:=1+2 \omega=\sqrt{-3}$ dividing $3$. Let 
$N(x):=N_{\mathbb{Q}(\omega)/\mathbb{Q}}(x)=|x|^2$
denote the norm form on $\mathbb{Q}(\omega)/\mathbb{Q}$. 
Each ideal 
$0 \neq \mathfrak{c} \unlhd \mathbb{Z}[\omega]$ is principal. 
If $(\mathfrak{c},3)=1$, then $\mathfrak{c}$ 
has a unique generator $\mathfrak{c}=(c)$ that satisfies $c \equiv 1 \pmod{3}$.

Whenever we write $d | c$ with $c \equiv 1 \pmod{3}$,
it is our convention that $d \equiv 1 \pmod{3}$. 
If $p \equiv 1 \pmod{3}$ is a rational prime, then 
$p=\varpi \overline{\varpi}$ in $\mathbb{Z}[\omega]$
with $N(\varpi)=p$ and $\varpi$ a prime in $\mathbb{Z}[\omega]$.
If $p \equiv 2 \pmod{3}$ is a rational prime, then $p=\varpi$ is inert in
$\mathbb{Z}[\omega]$ and $N(\varpi)=p^2$.
Define
\begin{equation*}
\check{e}(z):=e^{2 \pi i \text{Tr}_{\mathbb{C}/\mathbb{R}}(z)}
=e^{2 \pi i(z+\overline{z})}, \quad z \in \mathbb{C}.
\end{equation*}
For
$c \in \mathbb{Z}[\omega]$ with
$c \equiv 1 \pmod{3}$, the cubic Gauss sum is defined by 
\begin{equation} \label{unnormgauss}
g(c):=\sum_{d \pmod{c}}
\left( \frac{d}{c} \right)_3 \check{e} \bigg( \frac{d}{c} \bigg).
\end{equation}
We have the formula \cite[pp. 443--445]{H}
\begin{equation} \label{cuberel}
g(c)^3=\mu(c) c^2 \overline{c},
\end{equation}
where $\mu$ denotes the M\"{o}bius function on $\mathbb{Z}[\omega]$.
Observe that \eqref{cuberel} implies that $g(c)$ is supported on squarefree moduli.
We write
$$
\tilde{g}(c) := \frac{g(c)}{|c|}, 
$$
for the normalised cubic Gauss sum. Note that $|\tilde{g}(c)|=\mu^2(c)$ for all $c \in \mathbb{Z}[\omega]$.

An important property
of $g(c)$ is twisted multiplicativity \cite[pp. 443--445]{H}. It states
that 
\begin{equation} \label{twistmult}
\tilde{g}(ab)=\overline{\left( \frac{a}{b}  \right)_3}
\tilde{g}(a) \tilde{g}(b) \quad \text{for} \quad a,b \in \mathbb{Z}[\omega] \quad \text{satisfying} \quad (a,b)=1.
\end{equation}
Both sides of \eqref{twistmult} are zero when $(a,b) \neq 1$, and so \eqref{twistmult} can be trivially extended to all
$a,b \in \mathbb{Z}[\omega]$. Recall that cubic reciprocity and a supplement to cubic reciprocity for the units/ramified prime are is given in
\eqref{cuberecip} and \eqref{supprecip} respectively.

\section{Notational conventions and definitions}
Throughout the paper, $\varpi$ will denote a general prime in $\mathbb{Z}[\omega]$ satisfying $(\varpi,3)=1$,
and $\lambda := \sqrt{-3}$ is the unique ramified prime.
We also denote by $\cube$ an element of the form $b^3$ with $b \in \mathbb{Z}[\omega]$.
For $z \geq 3$, let
\begin{equation} \label{Pzdef}
\mathscr{P}(z):=\prod_{\substack{N(\varpi) \leq z \\ \varpi \equiv 1 \pmod{3} \\ \varpi \text{ prime} }} \varpi.
\end{equation}
For a given $w \geq 1$,
we say that $a \in \mathbb{Z}[\omega]$
with $a \equiv 1 \pmod{3}$ is $w$-rough if and only if $(a, \mathscr{P}(w))=1$.

Unless otherwise specified, it should be clear from context whether $\overline{q}$ means modular inverse (with respect to an appropriate modulus) 
or complex conjugation.

Many estimates in this paper hold for a large class of sequences given in
Definition \ref{seqdef} below.

\begin{definition} \label{seqdef}
  Given $\eta > 0$, $A \geq 1$, and $w \geq 1$, let $\mathcal{C}_{\eta}(A, w)$ denote the
  set of sequences $\boldsymbol{\alpha} := (\alpha_a)$ such that
  \begin{enumerate}
    \item $|\alpha_a| \leq 1$ for all $a \in \mathbb{Z}[\omega]$;
    \item $\alpha_a$ is supported on squarefree $w$-rough $a \in \mathbb{Z}[\omega]$ with $a \equiv 1 \pmod{3}$;
     \item $\alpha_a = 0$ unless $N(a) \asymp A$;
    \item For any $\varepsilon > 0$, $t \in \mathbb{R}$, $\ell \in \mathbb{Z}$ and $k, u \in \mathbb{Z}[\omega]$ with $k,u \equiv 1 \pmod{3}$,
    we have  
   \begin{equation}\label{eq:can}
      \sum_{\substack{a \in \mathbb{Z}[\omega] \\ a \equiv 1 \pmod{3} \\ u | a}} \alpha_a \Big ( \frac{a}{|a|} \Big )^{\ell} N(a)^{it} \Big ( \frac{k}{a} \Big )_{3} \ll_{\varepsilon} (1+|\ell|)^{\varepsilon} N(k)^{\varepsilon} (1 + |t|)^{\varepsilon} \Big ( \frac{A}{N(u)} \Big )^{1/2 + \eta + \varepsilon},
      \end{equation}
  \end{enumerate}
  provided that $\ell \neq 0$, or if $\ell = 0$, then provided that $k \neq \cube$. 
\end{definition}

The Generalized Riemann Hypothesis is used to show that
axiom \eqref{eq:can} above holds for sequences of interest to us 
(i.e. smoothed indicator functions on the set of $w$-rough integers in $\mathbb{Z}[\omega]$.)
See Section \ref{primesec} for details.

Where important, the dependence of implied constants on auxiliary parameters
will be indicated in subscripts i.e. $O_{\varepsilon,\xi,A,\ldots}$, $\ll_{\varepsilon,\xi,A,\ldots}$ and $\gg_{\varepsilon,\xi,A,\ldots}$. It will be crucial to give the implied constants 
of certain error terms in the proofs of Theorem \ref{thm:main}
and Theorem \ref{thm:main3}
explicitly in terms of some of the auxiliary parameters. Such terms are clearly indicated.

\section{Poisson summation formula}
We will need a minor variant of the number field Poisson summation formula.
\begin{lemma} \label{pois1}
  Let $V : \mathbb{R}^2 \rightarrow \mathbb{C}$ be a smooth Schwartz function. By an abuse of notation, set $V(x + i y) := V(x,y)$. Then
  we have
  \begin{equation} \label{poissonstate}
  \sum_{m \in \mathbb{Z}[\omega]} V(m) = \frac{2}{\sqrt{3}} \sum_{k \in \mathbb{Z}[\omega]} \int_{\mathbb{R}^2} V(x,y) \check{e} \Big ( \frac{k (x + i y)}{\lambda} \Big ) dx dy.
 \end{equation}
  \end{lemma}
\begin{proof}
  Let $\Lambda:=\mathbb{Z}[\omega]$, viewed as a discrete lattice embedded in $\mathbb{C}$.
  We identify $x + iy$ with $(x,y) \in \mathbb{R}^2$. Poisson summation gives
  $$
  \sum_{\mathbf{x} \in \Lambda} V(\mathbf{x}) = \frac{1}{\text{covol}(\Lambda)} \sum_{\mathbf{x} \in \Lambda^{\star}} \widehat{V}(\mathbf{x}),
  $$
  where $\Lambda^{\star}  = \lambda^{-1} \mathbb{Z}[\omega]$ is the dual lattice to $\Lambda$,
  $$
  \widehat{V}(\mathbf{x}) := \int_{\mathbb{R}^2} V(u,v) e(2(x u + v y)) du dv, 
  $$
  and
  $$
  \text{covol}(\Lambda) = \frac{\sqrt{3}}{2}.
  $$
Observe that
  $$
  \text{Re}[ (x + i y) (u - i v)] = x u + y v.
  $$
 Thus Poisson summation for $\mathbb{Z}[\omega]$ is given by
  \begin{equation} \label{poissonpenult}
  \sum_{m \in \mathbb{Z}[\omega]} V(m) = \frac{2}{\sqrt{3}} \sum_{k \in \lambda^{-1} \mathbb{Z}[\omega]} \int_{\mathbb{R}^2} V(x,y) e( 2 \text{Re} (\overline{k} (x + i y))) dx dy.
  \end{equation}
  We can replace $\overline{k}$ by $k$ in \eqref{poissonpenult}, since $\lambda^{-1} \mathbb{Z}[\omega]$ is closed under conjugation.
  Thus \eqref{poissonstate} holds, as required.
\end{proof}

\begin{lemma} \label{pois2}
  Let $0 \neq q \in \mathbb{Z}[\omega]$,
  $\psi : \mathbb{Z}[\omega] \rightarrow \mathbb{C}$ be a $q$-periodic function, and
  $V : \mathbb{R}^2 \rightarrow \mathbb{C}$ be a smooth Schwartz function.
Then
  $$
  \sum_{m \in \mathbb{Z}[\omega]} \psi(m) V(m) = \frac{2}{\sqrt{3} N(q)} \sum_{k \in \mathbb{Z}[\omega]} \dot{\psi}(k) \dot{V} \Big ( \frac{k}{q} \Big ), 
  $$
  where
  $$
  \dot{\psi}(k) := \sum_{t \pmod{q}} \psi(t) \check{e} \Big ({- \frac{t k}{q \lambda}} \Big ), 
  $$
  and
  $$
  \dot{V}(u) := \int_{\mathbb{R}^2} V(x,y) \check{e} \Big ( \frac{u (x + iy)}{\lambda} \Big ) dx dy .
  $$
\end{lemma}
\begin{remark} \label{dualrem}
For each $t \in \mathbb{Z}[\omega]$, note that the additive character
$$
k \mapsto \check{e} \Big ({-\frac{t k}{3 \lambda q }} \Big )
$$
has minimal period $3q$ (not $3 \lambda q$). 
\end{remark}

\begin{proof}
  We have
  $$
  \sum_{m \in \mathbb{Z}[\omega]} \psi(m) V(m)  = \sum_{t \pmod{q}} \psi(t) \sum_{m \in \mathbb{Z}[\omega]} V(m q + t)
  $$
Application of Lemma \ref{pois1} to the summation over $m$ gives
  $$
  \sum_{m \in \mathbb{Z}[\omega]} V(m q + t) = \frac{2}{\sqrt{3}} \sum_{k \in \mathbb{Z}[\omega]} \int_{\mathbb{R}^2} V((x + i y) q + t) \check{e} \Big ( \frac{k (x + i y)}{\lambda} \Big ) dx dy.
  $$
A linear change of variable then shows that 
 \begin{equation} \label{lastpois}
  \sum_{m \in \mathbb{Z}[\omega]} V(m q + t)=\check{e} \Big ( {- \frac{k t}{q \lambda}} \Big ) \frac{2}{\sqrt{3} N(q)} \sum_{k \in \mathbb{Z}[\omega]} \int_{\mathbb{R}^2} V (x,y) \check{e} \Big ( \frac{k (x + i y)}{q \lambda} \Big ) dx dy.
 \end{equation}
 The result follows upon summing both sides of \eqref{lastpois} 
 over $t \pmod{q}$ with the $q$-periodic weights $\psi(t)$. 
\end{proof}

We will specialise to the case where the test function is radially symmetric. 
\begin{lemma} \label{radialpois}
  Let $q \in \mathbb{Z}[\omega]$ with $q \equiv 1 \pmod{3}$,
  $\psi: \mathbb{Z}[\omega] \rightarrow \mathbb{C}$ be a $q$-periodic function,
  and $V : \mathbb{R} \rightarrow \mathbb{C}$ be a smooth Schwartz function.
  Then for any $M > 0$ we have
  $$
  \sum_{\substack{m \in \mathbb{Z}[\omega] \\ m \equiv 1 \pmod{3}}} \psi(m) V \Big (\frac{N(m)}{M} \Big )
  = \frac{4 \pi M}{9 \sqrt{3} N(q)} \sum_{k \in \mathbb{Z}[\omega]} \ddot{\psi}(k) \ddot{V} \Big ( \frac{k \sqrt{M}}{q} \Big ),
  $$
  where
  $$
\ddot{\psi}(k):=\check{e} \Big({-\frac{k}{3 \lambda}} \Big)  \sum_{x \pmod{q}} \psi(3 \lambda x) \check{e} \Big ({- \frac{k x}{q}} \Big ),
$$
 and
$\ddot{V} : \mathbb{C} \rightarrow \mathbb{C}$ is defined by
\begin{equation} \label{Vddot}
\ddot{V}(u):=\int_{0}^{\infty} r V(r^2) J_0 \Big ( \frac{4\pi r|u|}{3 \sqrt{3}} \Big ) dr.
\end{equation}
\end{lemma}
\begin{proof}
Application of Lemma \ref{pois2} gives
  $$
  \sum_{\substack{m \in \mathbb{Z}[\omega] \\ m \equiv 1 \pmod{3}}} \psi(m) V \Big ( \frac{N(m)}{M} \Big ) = \frac{1}{N(3 q)} \frac{2}{\sqrt{3}}
  \sum_{k \in \mathbb{Z}[\omega]} \dot{\psi}(k) \int_{\mathbb{R}^2} V \Big ( \frac{x^2 + y^2}{M} \Big ) \check{e} \Big ( \frac{k (x + i y)}{3 \lambda q} \Big ) dx dy, 
  $$
  where
  \begin{equation} \label{phidotk}
  \dot{\psi}(k) = \sum_{\substack{t \pmod{3 q} \\ t \equiv 1 \pmod{3}}} \psi(t) \check{e} \Big ({ - \frac{t k}{3 \lambda q }} \Big ). 
 \end{equation}
 We first simplify the integral. A change of variable gives
  \begin{align} \label{eq:a}
  \int_{\mathbb{R}^2} V \Big ( \frac{x^2 + y^2}{M} \Big ) \check{e} \Big ( \frac{k(x + i y)}{3 \lambda q} \Big ) dx dy = M \int_{\mathbb{R}^2} V(x^2 + y^2) \check{e} \Big ( \frac{k(x + iy) \sqrt{M}}{3 \lambda q} \Big ) dx dy. 
  \end{align}
  We change $x + iy$ to polar coordinates via the substitution $x + i y = r e^{i\vartheta }$. 
  Let $\kappa \in [-\pi,\pi)$ be a fixed angle (depending on $k$ and $q$) such that
  $$
  e^{-i \kappa} =
  \begin{cases}
   \frac{k}{|k|} \frac{|\lambda q|}{\lambda q} & \text{if} \quad k \neq 0 \\
   1 & \text{if} \quad k=0.
   \end{cases}
  $$
Then \eqref{eq:a} becomes
  \begin{align} \label{eq:b}
  M \int_{0}^{\infty} r V(r^2) \int_{0}^{2\pi} \check{e} \Big ( \frac{r e^{i \vartheta - i \kappa} |k| \sqrt{M}}{3 \sqrt{3} |q|} \Big ) d \vartheta d r.
  \end{align}
  We eliminate the translation by $\kappa$ by a linear change of variable in $\vartheta$ and the fact that the integrand is periodic with period $2 \pi$.
  Therefore \eqref{eq:b} becomes
  $$
  M \int_{0}^{\infty} r V(r^2) \int_{0}^{2\pi} \exp \Big ( \frac{4 \pi i r \cos(\vartheta) |k| \sqrt{M}}{3 \sqrt{3} |q|} \Big ) d \vartheta dr.
  $$
Observe that by \cite[(10.9.2)]{NIST:DLMF} and the fact that $J_0(x)$ is real-valued we obtain 
 \begin{equation*}
  J_0(x) =\frac{1}{2\pi} \int_{0}^{2 \pi} e^{i x \cos(\vartheta)} d \vartheta.
  \end{equation*}
Thus \eqref{eq:b} is equal to
  $$
  2\pi M \int_{0}^{\infty} r V(r^2) J_0 \Big ( \frac{4\pi r |k| \sqrt{M}}{3 \sqrt{3} |q|} \Big ) dr.
  $$
  
  It remains to compute $\dot{\psi}(k)$. 
  The Chinese Remainder theorem guarantees that we can write any $t \pmod{3q}$ as $t = a q + 3 b$ with $a$ a representative of a residue class $\pmod{3}$ and $b$ a representative of a residue class $\pmod{q}$. Necessarily $a \equiv \overline{q} \pmod{3}$. 
  Thus \eqref{phidotk} becomes
  \begin{align*}
\dot{\psi}(k)&:= \Big( \sum_{\substack{a \pmod{3} \\ a \equiv \overline{q} \pmod{3}}} \check{e} \Big ({- \frac{a k}{3 \lambda}} \Big ) \Big) \Big( \sum_{b \pmod{q}} \psi(3 b) 
\check{e} \Big ({- \frac{b k}{\lambda q}} \Big ) \Big)  \\ 
& = \check{e} \Big ( {- \frac{k}{3 \lambda}} \Big ) \sum_{b \pmod{q}} \psi(3 \lambda b) \check{e} \Big ({-\frac{b k}{q}} \Big ),  
  \end{align*}
  where the displays followed from the fact that $q \equiv 1 \pmod{3}$, $\psi$ is periodic modulo $q$ with $(q,3)=1$,
  and Remark \ref{dualrem}.
 \end{proof}

We now state the final version of the Poisson summation formula needed for this paper.
\begin{corollary} \label{corpois}
  Let $n_1, n_2 \in \mathbb{Z}[\omega]$ be squarefree and satisfy $n_1 \equiv n_2 \equiv 1 \pmod{3}$.
  Let $d:= (n_1, n_2)$ and $V : \mathbb{R} \rightarrow \mathbb{C}$ be a smooth Schwartz function.
Then for any $M  > 0$ we have
  \begin{align*}
  \sum_{\substack{m \in \mathbb{Z}[\omega] \\ m \equiv 1 \pmod{3}}} \Big ( \frac{m}{n_1} \Big )_{3} \overline{\Big ( \frac{m}{n_2} \Big )_{3}} V \Big ( \frac{N(m)}{M} \Big )
  &=\frac{4 \pi \big ( \frac{d}{n_1 / d} \big )_{3} \overline{\big ( \frac{d}{n_2 / d} \big )_{3}} M g(n_1/d) \overline{g (n_2/d)}}{9 \sqrt{3} N \big (n_1 n_2/d \big )} \\
  & \times \sum_{k \in \mathbb{Z}[\omega]} \widetilde{c}_{d}(k) \overline{\Big ( \frac{k}{n_1 / d} \Big )_{3}} \Big ( \frac{k}{n_2 / d} \Big )_3
  \ddot{V} \Big ( \frac{k d \sqrt{M}}{n_1 n_2} \Big ), 
 \end{align*}
  where
\begin{equation} \label{tildec}
\widetilde{c}_{d}(k) := \check{e} \Big({-\frac{k}{3 \lambda}} \Big) \sum_{\substack {x \pmod{d} \\ (x,d)=1 }} \check{e} \Big ({- \frac{k x}{d} }\Big ),
\end{equation}
and $\ddot{V} : [0, \infty) \rightarrow \mathbb{C}$ is given by
$$
\ddot{V}(u) := \int_{0}^{\infty} r V(r^2) J_0 \Big ( \frac{4\pi r|u|}{3 \sqrt{3}} \Big ) dr.
$$
\end{corollary} 
\begin{proof}
  We apply Lemma \ref{radialpois} to 
  \begin{equation*}
  \psi_{n_1,n_2}(m):=\left(\frac{m}{n_1} \right)_3 \overline{\left( \frac{m}{n_2} \right)_3}
  =\left( \frac{m}{n_1/d} \right)_3 \overline{ \left( \frac{m}{n_2/d} \right)_3   } \mathbf{1}_d(m), \quad m \in \mathbb{Z}[\omega],
  \end{equation*}
  where $\textbf{1}_d$ denotes the principal character modulo $d$. Observe that 
  $\psi_{n_1,n_2}$ is $n_1 n_2/d$ periodic. All that remains to do is to
 compute $\ddot{\psi}_{n_1,n_2}$. We have
  \begin{equation} \label{ddotinit}
  \ddot{\psi}_{n_1, n_2}(k) = \check{e} \Big ({- \frac{k}{3 \lambda}} \Big ) \sum_{\substack{x \pmod{n_1 n_2/d} \\ (x,d) = 1}}
   \Big ( \frac{3 \lambda x}{n_1/d} \Big )_{3} \overline{\Big ( \frac{3 \lambda x}{n_2/d} \Big )_{3}} \check{e} \Big ({- \frac{k x}{(n_1 n_2/d)}} \Big ). 
  \end{equation}
  Observe that $3 \lambda=(-\lambda)^3$ and $(n_1 n_2/d,3)=1$, 
  so \eqref{ddotinit} becomes 
  \begin{equation*}
  \ddot{\psi}_{n_1, n_2}(k) = \check{e} \Big ({- \frac{k}{3 \lambda}} \Big ) \sum_{\substack{x \pmod{n_1 n_2/d} \\ (x,d) = 1}}
   \Big ( \frac{x}{n_1/d} \Big )_{3} \overline{\Big ( \frac{x}{n_2/d} \Big )_{3}} \check{e} \Big ({- \frac{k x}{(n_1 n_2/d)}} \Big ).
  \end{equation*}
Since $n_1$ and $n_2$ are squarefree we have $(n_1 n_2/d^2,d)=1$.
We use the Chinese remainder theorem to write $x = a (n_1 n_2/d^2) + b d$. We find that
  \begin{align*}
\ddot{\psi}_{n_1,n_2}(k)&=\check{e} \Big ({- \frac{k}{3 \lambda}} \Big )
 \sum_{\substack{x \pmod{n_1 n_2/d} \\ (x,d) = 1}} \Big ( \frac{x}{n_1/d} \Big )_{3} \overline{\Big ( \frac{x}{n_2/d} \Big )_{3}}
\check{e} \Big ({- \frac{k x}{(n_1 n_2/d)} }\Big )  \\ 
& =\Big(  \check{e} \Big ({- \frac{k}{3 \lambda}} \Big ) \sum_{\substack{a \pmod{d} \\ (a,d) = 1}} \check{e} \Big ({- \frac{a k}{d}} \Big ) \Big) 
\Big( \sum_{b \pmod{n_1 n_2/d^2}} \Big ( \frac{b d}{n_1/d} \Big )_{3} \overline{\Big ( \frac{b d}{n_2/d} \Big )_{3}} \check{e} \Big ({ - \frac{k b}{(n_1 n_2/d^2)}} \Big ) \Big). 
  \end{align*}
Observe that $(n_1/d,n_2/d)=1$.
 To evaluate the sum over $b$ we use the 
 Chinese remainder theorem again. Writing $b = t (n_1/d) + u (n_2/d)$
 gives
  \begin{align*}
  \ddot{\psi}_{n_1,n_2}(k)&=\widetilde{c}_d(k) \Big( \sum_{t \pmod{n_2/d}} \overline{\Big ( \frac{t n_1}{n_2/d} \Big )_{3}} 
  \check{e} \Big ({ - \frac{t k}{n_2/d}} \Big ) \Big)
 \Big( \sum_{u \pmod{n_1/d}} \Big ( \frac{u n_2}{n_1/d} \Big )_3 \check{e} \Big ( {- \frac{u k}{n_1/d}} \Big ) \Big) \\ 
& = \widetilde{c}_d(k) \Big ( \frac{k}{n_2/d} \Big )_{3} \overline{\Big ( \frac{k}{n_1/d} \Big )_{3}} \Big ( \frac{n_2}{n_1/d} \Big )_{3}
 \overline{\Big ( \frac{n_1}{n_2/d} \Big )_{3}} \cdot g (n_1/d) \overline{g(n_2/d)},
  \end{align*}
where the last display follows from the primitivity of characters $\left(\frac{\cdot}{n_1/d} \right)_3$
and $\overline{\left(\frac{\cdot}{n_2/d} \right)_3}$.
Finally,
  $$
  \Big ( \frac{n_2}{n_1/d} \Big )_{3} \overline{\Big ( \frac{n_1}{n_2/d} \Big )_{3}}=\left( \frac{n_2/d}{n_1/d} \right)_3 \left(\frac{d}{n_1/d}  \right)_3
   \overline{\left(\frac{n_1/d}{n_2/d}  \right)_3 \left( \frac{d}{n_2/d} \right)_3}   
  =\Big ( \frac{d}{n_1/d} \Big )_{3} \overline{\Big ( \frac{d}{n_2/d} \Big )_{3}},
  $$
  where the last equality follows from cubic reciprocity. This completes the proof.
  \end{proof}

  We close this section with standard estimate for $\ddot{V}$.
  
  \begin{lemma} \label{ddotdecay}
  Let $V: \mathbb{R} \rightarrow \mathbb{C}$ be a smooth compactly supported function. Then for any integer $k \geq 0$,
  \begin{equation}
  |\ddot{V}(u)| \ll_{k,V} (1+|u|)^{-k}, \quad u \in \mathbb{C}.
  \end{equation}
  \end{lemma}
  \begin{proof} 
  Integrating \eqref{Vddot} by parts $k \in \mathbb{Z}_{\geq 0}$ times using \cite[(10.22.1)]{NIST:DLMF} 
  gives
  \begin{equation} \label{intparts}
  \ddot{V}(u)=(-1)^k \Big( \frac{3 \sqrt{3}}{2 \pi} \Big)^{k} \frac{1}{|u|^k} \int_0^{\infty} V^{(k)}(r^2) \cdot r^{k+1} 
  J_k \Big( \frac{4 \pi r |u|}{3 \sqrt{3}} \Big) dr.
  \end{equation}
  The claim immediately follows.
 \end{proof}

\section{Voronoi summation in the level aspect} \label{sec:voronoi}
The Fourier coefficients of the cubic theta function essentially sample cubic Gauss sums.
Naturally, automorphy of the theta function is a key input in the proof of our 
level aspect Voronoi summation formula. 

\subsection{Geometry, groups, and the cubic theta function at cusps}
Let $\mathbb{H}^{3}$ denote the hyperbolic 3-space $\mathbb{C} \times \mathbb{R}^{+}$.
We embed $\mathbb{H}^3$ in the Hamilton quaternions  by identifying 
$i=\sqrt{-1}$ with $\hat{i}$ and 
$w=(z,v)=(x+iy,v) \in \mathbb{H}^3$ with $x+y \hat{i}+v \hat{k}$, where 
$1,\hat{i},\hat{j},\hat{k}$ denote the unit quaternions. In terms of 
quaternion arithmetic, the group action 
of $\operatorname{SL}_2(\mathbb{C})$ on $\mathbb{H}^3$
is given by
\begin{equation*}
\gamma w=(aw+b)(cw+d)^{-1}, \quad \quad \gamma=\begin{pMatrix}
a b
c d 
\end{pMatrix} \in \operatorname{SL}_2(\mathbb{C}) \quad \text{and} \quad
w \in \mathbb{H}^3.
\end{equation*}
In terms of coordinates, 
\begin{equation} \label{grpaction}
\gamma w= \bigg( \frac{(az+b) \overline{(cz+d)}+a \overline{c} v^2}{|cz+d|^2 +|c|^2 v^2}, 
\frac{v}{|cz+d|^2+|c|^2 v^2} \bigg),
\quad w=(z,v) \in \mathbb{H}^3.
\end{equation}

Let $\Gamma:=\operatorname{SL}_2(\mathbb{Z}[\omega])$.
It is a standard fact that $\Gamma$ is generated by the elements
\begin{equation*}
P:=\begin{pMatrix}
{\omega}  {0}  
{0}  {\omega^2} 
\end{pMatrix}, \quad T:=\begin{pMatrix}
{1}  {1}  
{0}  {1} 
\end{pMatrix}, \quad \text{and} \quad 
E:=\begin{pMatrix}
{0}  {-1}  
{1}  {0} 
\end{pMatrix}.
\end{equation*} 
Let $A \in \mathbb{Z}[\omega]$
satisfy $A \equiv 0 \pmod{3}$, and
\begin{equation*}
\Gamma_1(A):=\{ \gamma \in \Gamma : \gamma \equiv I \pmod{A}  \}.
\end{equation*}
Observe that $\Gamma_1(A)$ is a normal subgroup of $\Gamma$ since it is the kernel of the 
reduction modulo $A$ map. Let 
\begin{equation} \label{Gamma2}
\Gamma_2:=\langle \operatorname{SL}_2(\mathbb{Z}), \Gamma_1(3) \rangle=\operatorname{SL}_2(\mathbb{Z}) \Gamma_1(3)=\Gamma_1(3)
\operatorname{SL}_2(\mathbb{Z}),
\end{equation}
where the last two equalities follow because 
$\Gamma_1(3)$ is normal in $\Gamma$. 
We also have $[\Gamma:\Gamma_2]=27$ (see \cite[\S 2]{Pat1} for the calculation).
We also observe the trivial (but useful) fact that for any $g,g^{\prime} \in \Gamma$ 
we have that
\begin{equation*}
\Gamma_1(3) g = \Gamma_1(3) g^{\prime} \iff g \equiv g^{\prime}  \pmod{3}.
\end{equation*}
Furthermore, for any $g,g^{\prime} \in \Gamma$ we have that
\begin{equation} \label{simpleimpl}
g \equiv g^{\prime}  \pmod{3} \implies \Gamma_2 g = \Gamma_2 g^{\prime}.
\end{equation}
More properties of these groups can be found in \cite[\S 2]{Pat1}.

Let
\begin{equation*}
 \chi: \Gamma_1(3) \rightarrow ~ \{1,\omega, \omega^2 \}
\end{equation*}
be the famous cubic Kubota character \cite{Kub1,Kubota1}, given by
\begin{equation} \label{chidef}
\chi(\gamma):=\begin{cases}
\left(\frac{c}{a} \right)_3 & \text{if} \quad \gamma=\pmatrix
abcd \in \Gamma_1(3)
\quad \text{and} \quad c \neq 0, \\
1 & \text{otherwise}.
\end{cases}
\end{equation}
It was shown by Patterson \cite[\S 2]{Pat1} that $\chi$ extends
to a well-defined homomorphism 
\begin{equation*} 
\chi: \Gamma_2 \rightarrow \{1,\omega,\omega^2\},
\end{equation*}
when one defines
$\chi \lvert_{\operatorname{SL}_2(\mathbb{Z})} \equiv 1$. 

\begin{remark}
There is a useful alternative expression for $\chi(\gamma)$ to the one given in \eqref{chidef}. For $\gamma=\pmatrix
abcd \in \Gamma_1(3)$, we have the determinant equation
$ad-bc=1$ with $1+bc \equiv 1 \pmod{9}$. Given $0 \neq c \in \mathbb{Z}[\omega]$ 
we write $c=\pm \lambda^{k} \omega^{j} c^{\prime}$ where $k \geq 2$, $0 \leq j \leq 2$ are integers, and
$c^{\prime} \in \mathbb{Z}[\omega]$ is such that $c^{\prime} \equiv 1 \pmod{3}$.
Thus
\begin{equation*}
\Big( \frac{c}{ad} \Big)_3=\Big(\frac{c}{1+bc} \Big)_3=\Big( \frac{\lambda} {1+bc} \Big)^{k}_3 \Big( \frac{\omega}{1+bc} \Big)^{j}_3 \Big( \frac{c^{\prime}}{1+bc} \Big)_3=1,
\end{equation*}
where the last equality follows from $1+bc \equiv 1 \pmod{9}$, \eqref{cuberecip}, and \eqref{supprecip}.
Hence 
\begin{equation*}
\Big( \frac{c}{a} \Big)_3= \overline{\Big( \frac{c}{d} \Big)_3}=\Big( \frac{b}{d} \Big)_3,
\end{equation*}
and we obtain
\begin{equation} \label{altchi}
\chi(\gamma)=
\begin{cases}
\left(\frac{b}{d} \right)_3 & \text{if} \quad \gamma=\pmatrix
abcd \in \Gamma_1(3)
\quad \text{and} \quad c \neq 0, \\
1 & \text{otherwise}.
\end{cases}
\end{equation}
\end{remark}

Let $\theta(w)$ denote the cubic metaplectic theta function of Kubota on $\mathbb{H}^3$. 
It is automorphic on $\Gamma_2$
with multiplier $\chi$. It has Fourier expansion (at $\infty$) given by
\begin{equation*}
\theta(w)=\sigma v^{2/3}+\sum_{\mu \in \lambda^{-3} \mathbb{Z}[\omega]}
\tau(\mu) v K_{\frac{1}{3}}(4 \pi |\mu| v) \check{e}(\mu z), \quad w \in \mathbb{H}^3,
\end{equation*}
where 
\begin{equation} \label{sigma}
\sigma:=3^{5/2}/2,
\end{equation}
and the other Fourier coefficients were  
computed by Patterson \cite[Theorem~8.1]{Pat1}. They are
\begin{equation} \label{cubiccoeff}
\tau(\mu)=\begin{cases}
\overline{g(\lambda^2,c)} \big |\frac{d}{c} \big | 3^{n/2+2} & \quad \text{if} \quad \mu=\pm \lambda^{3n-4} c d^3, \quad n \geq 1\\ 
e^{-\frac{2 \pi i}{9}} \overline{g(\omega \lambda^2,c)} \big |\frac{d}{c} \big | 3^{n/2+2} & \quad \text{if} \quad \mu=\pm \omega \lambda^{3n-4} c d^3, \quad   n \geq 1\\
e^{\frac{2 \pi i}{9}} \overline{g(\omega^2 \lambda^2,c)} \big |\frac{d}{c} \big | 3^{n/2+2} & \quad \text{if} \quad \mu=\pm \omega^2 \lambda^{3n-4} c d^3, \quad n \geq 1\\
\overline{g(1,c)} \big |\frac{d}{c} \big | 3^{n/2+5/2} & \quad \text{if} \quad \mu=\pm  \lambda^{3n-3} c d^3, \quad n \geq 0 \\ 
0 & \quad \text{otherwise},
\end{cases}
\end{equation}
where 
\begin{equation} \label{cdcond}
c,d \in \mathbb{Z}[\omega], \quad c,d \equiv 1 \pmod{3}, \quad \text{and} \quad \mu^2(c)=1. 
\end{equation}
It follows from \eqref{cubiccoeff} that $\tau(\cdot)$ is an even function.

Implicit in \cite[\S 7 and \S8]{Pat1} is a careful study of $\theta(w)$ at
various cusps of $\Gamma_2$. We extract the information that will be of use to us.
Let $\{\gamma_j: j=1,\ldots 27\}$ be the complete set of inequivalent representatives
for $\Gamma_2 \backslash \Gamma$ given in \cite[Table II pg.~129]{Pat1}.
Particular coset representatives $\gamma_j$ of $\Gamma_2 \backslash \Gamma$
of importance to us are
\begin{equation*}
\gamma_1=I, \gamma_2=\begin{pMatrix}
1  \omega  
0  1 
\end{pMatrix}, \gamma_3=
\begin{pMatrix}
1  {-\omega}  
0  1 
\end{pMatrix},
\gamma_{10}=
\begin{pMatrix}
1  {0}  
{\omega}  1 
\end{pMatrix},
\text{ and }
\gamma_{19}=
\begin{pMatrix}
1  {0}  
{\omega^2}  {1} 
\end{pMatrix}.
\end{equation*}
For each $j=1,2,\ldots,27$, let
\begin{equation}
F_j(w):=\theta(\gamma_j(w)), \quad w \in \mathbb{H}^3.
\end{equation}
If $g \in \Gamma$, then
\begin{equation} \label{aut1}
\gamma_j g=g_j(g) \gamma_{k_j(g)}, \quad \text{for some} \quad g_j(g) \in \Gamma_2 
\quad \text{and} \quad 1 \leq k_j(g) \leq 27.
\end{equation}
Thus 
\begin{equation} \label{aut2}
F_j(g(w))=\chi(g_j(g)) F_{k_j(g)}(w) \quad \text{for all} \quad w \in \mathbb{H}^3.
\end{equation}
Each $F_j$ is automorphic on $\Gamma_1(9)$ with multiplier $\chi$
by \cite[Lemma~2.1]{Pat4}. Following Patterson, we define
\begin{equation} \label{Fjstarfourier}
F_j^{\star}(w):=\sum_{\mu}
d_j(\mu) v K_{1/3}(4 \pi |\mu| v) \check{e}(\mu z), \quad w \in \mathbb{H}^3,
\end{equation}
where the 
$d_j(\mu)$ have support contained in $\lambda^{-4} \mathbb{Z}[\omega] \setminus \{0\}$,
and have expressions in terms of $\tau(\mu)$, $\tau_1(\mu)$ \cite[(8.8)]{Pat1} 
and $\tau_2(\mu)$ \cite[(8.9)]{Pat1}, see Appendix \ref{appendix}. For the reader's convenience we state formulae for
$\tau_1(\mu)$ and $\tau_2(\mu)$ here. They are given by
\begin{equation} \label{cubiccoeff1}
\tau_1(\mu)=\begin{cases}
9 \omega \overline{g(\lambda^2,c)} \big | \frac{d}{c} \big | & \quad \text{if} \quad \mu =\lambda^{-4} cd^3  \\
9 e^{-\frac{2 \pi i}{9}} \omega^2 \overline{g(\omega \lambda^2,c)} \big | \frac{d}{c} \big | & \quad \text{if} \quad \mu=\omega \lambda^{-4} c d^3 \\
9 e^{\frac{2 \pi i}{9}} \overline{g(\omega^2 \lambda^2,c)} \big | \frac{d}{c} \big | & \quad \text{if} \quad \mu=\omega^2 \lambda^{-4} c d^3 \\
0 & \quad \text{otherwise},
 \end{cases}
\end{equation}
and
\begin{equation} \label{cubiccoeff2}
\tau_2(\mu)= 
\begin{cases}
9 \omega^2 \overline{g(\lambda^2,c)} \big | \frac{d}{c} \big | & \quad \text{if} \quad \mu=-\lambda^{-4} cd^3  \\
9 e^{-\frac{2 \pi i}{9}} \overline{g(\omega \lambda^2,c)} \big | \frac{d}{c} \big | &\quad \text{if} \quad \mu=-\lambda^{-4} \omega c d^3 \\
9 \omega e^{\frac{2 \pi i}{9}} \overline{g(\omega^2 \lambda^2,c)} \big | \frac{d}{c} \big | & \quad \text{if} \quad \mu=-\lambda^{-4} \omega^2 c d^3 \\
0 & \quad \text{otherwise},
\end{cases}
\end{equation}
where $c$ and $d$ are as in \eqref{cdcond}. The formulas for the $d_j(\mu)$ 
are given in \cite[Table III pg.~151]{Pat1}. We have also included
them in Appendix \ref{appendix}.
We have the Fourier expansions (at $\infty$)
\cite[pg.~148]{Pat1},
\begin{equation} \label{Fjstar}
F_j(w)=\begin{cases}
\sigma v(w)^{2/3}+F_j^{\star}(w) & \text{if} \quad 1 \leq j \leq 9, \\
F_j^{\star}(w) & \text{if} \quad 10 \leq j \leq 27
\end{cases}, \quad w \in \mathbb{H}^3.
\end{equation}
To understand the maps
$j \mapsto g_j(\cdot)$, $j \mapsto k_j(\cdot)$
and $j \mapsto \chi(g_j(\cdot))$
occurring in \eqref{aut1}
and \eqref{aut2}, 
it suffices to compute them on
the generators of $\Gamma=\operatorname{SL}_2(\mathbb{Z}[\omega])$: $P,T,$ and $E$.
The values of $k_j(E)$ appear in
\cite[Table~III]{Pat1}. We have included
the $k_j$ values on all three generators in 
Appendix \ref{appendix}.

\subsection{Conjugation and coefficient sieving} 
It is more convenient for us to work with
the $\overline{F_j}(w)$.
It follows from \eqref{Fjstar} that they each have Fourier expansion (at $\infty$) given by
\begin{equation} \label{fjfourier}
\overline{F_j}(w)=
\begin{cases}
\sigma v^{2/3}+\sum_{ \mu}
\overline{d_j(-\mu)} v K_{\frac{1}{3}}(4 \pi |\mu| v) \check{e}(\mu z) & \text{if} \quad 1 \leq j \leq 9 \\
\sum_{\mu}
\overline{d_j(-\mu)} v K_{\frac{1}{3}}(4 \pi |\mu| v) \check{e}(\mu z)  & \text{if} \quad 10 \leq j \leq 27
\end{cases}, \quad w \in \mathbb{H}^3,
\end{equation}
since $K_{1/3}(x) \in \mathbb{R}$ for $x>0$.

The Fourier coefficients of $\overline{F_1}(w)$ are given by
\begin{equation*}
\overline{d_1(-\mu)}=\overline{\tau(-\mu)}=\overline{\tau(\mu)},
\end{equation*}
where the last equality follows from the fact that $\tau(\cdot)$ is even. Let
\begin{equation} \label{Sdefn}
S:=\{\lambda^{-3} c d^3 \in \mathbb{Q}(\omega) : c,d \in \mathbb{Z}[\omega], 
\quad c,d \equiv 1 \hspace{-0.2cm} \pmod{3} \quad \text{and} \quad \mu^2(c)=1 \},
\end{equation}
and 
\begin{equation*}
\overline{F_1}(w)_S:=\sum_{\mu \in S}
\overline{\tau(\mu)} v K_{\frac{1}{3}}(4 \pi |\mu| v) \check{e}(\mu z), \quad w \in \mathbb{H}^3.
\end{equation*}
\begin{remark} \label{FSremark}
We study the function $\overline{F_1}(w)_S$ because its Fourier coefficients are supported on $S$ 
and the $\mu=\lambda^{-3}cd^3$th coefficient is
$\widetilde{g}(c)|d|$ up to an absolute constant (cf. \eqref{cubiccoeff}).
These are the coefficients (up to twisting)
we want to appear in our Voronoi formula (cf. Proposition \ref{voronoiprop}).
\end{remark}

\begin{lemma} \label{Ssieve}
Let $\overline{F_1}(w)_S$ be as above. Then
\begin{equation*} 
\overline{F_1}(w)_S=\frac{1}{3} \Big(\overline{F_1}(w)+\omega \overline{ F_2}(w)+\omega^2 \overline{F_3}(w) \Big),
\end{equation*}
and $\overline{F_1}(w)_S$ is automorphic under $\Gamma_1(9)$ with multiplier $\overline{\chi}$.
\end{lemma}

\begin{proof}
Following \cite[Theorem~5.2]{Pat4}, we detect $\mu \in S$ additively.
From \eqref{cubiccoeff}, we have 
\begin{equation*}
\{\mu \in \lambda^{-3} \mathbb{Z}[\omega] : \check{e}(\omega \mu)=\omega^2 \quad \text{and} \quad 
\tau(\mu) \neq 0 \}=S. 
\end{equation*}
Thus 
\begin{align}
\overline{F_1}(w)_S&=\frac{1}{3} \sigma v^{2/3}(1+\omega+\omega^2) \nonumber \\
&+\frac{1}{3} \sum_{\substack{\mu} \in \lambda^{-3} \mathbb{Z}[\omega]}
\overline{\tau(\mu)} v K_{\frac{1}{3}}(4 \pi |\mu| v) \check{e}(\mu z) \big(1+\omega \check{e}(\omega \mu)+\omega^2 \check{e}(2 \omega \mu)\big) \nonumber \\
&=\frac{1}{3} \Big(\overline{F_1}(w)+\omega \overline{ F_2}(w)+\omega^2 \overline{F_3}(w) \Big), \label{handyeq}
\end{align} 
where the last term in \eqref{handyeq} was obtained by writing 
$\pmatrix
{1} {2 \omega}
{0} {1}=\pmatrix
{1} {3 \omega}
{0} {1} \pmatrix
{1} {-\omega}
{0} {1}$
and using automorphy of $\overline{F_1}(w)=\overline{\theta}(w)$ 
on $\Gamma_2$ with multiplier $\overline{\chi}$. This proves the first claim.
Each $\overline{F_j}$ is automorphic on $\Gamma_1(9)$ 
with multiplier $\overline{\chi}$, and so the second claim follows.
\end{proof}

\subsection{Twists}
Let $r \in \mathbb{Z}[\omega]$ with $r \equiv 1 \pmod{3}$, and $\psi$ 
be a function on $\mathbb{Z}[\omega]$
that is periodic modulo $r$.
In view of \eqref{Fjstarfourier}, the $\psi$-twist of $\overline{F^{\star}_j}(w)$ is
defined by 
\begin{equation} \label{zerotwist}
\overline{F^{\star}_j}(w;\psi):=\sum_{\mu}
\overline{d_j(-\mu)} \psi(\lambda^4 \mu) v K_{\frac{1}{3}}(4 \pi |\mu| v) \check{e}(\mu z), \quad w \in \mathbb{H}^3.
\end{equation}
In view of \eqref{fjfourier}, the $\psi$-twist of $\overline{F_j}(w)$
is
\begin{equation} \label{psitwist1}
\overline{F_j}(w;\psi):= \begin{cases}
\sigma \psi(0) v^{2/3} +\overline{F^{\star}_j}(w;\psi) & 1 \leq j \leq 9 \\ 
\overline{F^{\star}_j}(w;\psi) & 10 \leq j \leq 27
\end{cases},  \quad w \in \mathbb{H}^3.
\end{equation}
\begin{remark}
The Fourier coefficients of all the $\overline{F_j}$ 
have support contained in $\lambda^{-4} \mathbb{Z}[\omega]$. This explains why we define a general twist by
$\psi(\lambda^4 (\cdot))$ in \eqref{psitwist1}. In the special case $\overline{F_1}=\overline{\theta}$, \eqref{cubiccoeff}
tells us that the Fourier coefficients have support contained in $\lambda^{-3} \mathbb{Z}[\omega]$. Thus
our twisting definition produces an extraneous $\psi(\lambda)$ factor in this case. This will be immaterial
in our final results.
\end{remark}

Define the Fourier transform
\begin{equation} \label{fouriertrans}
\widehat{\psi}(u) := \sum_{x \pmod{r}} \psi(x) \check{e} \Big ( \frac{u x}{r} \Big ), \quad u \in \mathbb{Z}[\omega].
\end{equation}
Fourier inversion tells us that
$$
\psi(x) = \frac{1}{N(r)} \sum_{u \pmod{r}} \widehat{\psi}(u) \check{e} \Big ({-\frac{u x}{r}} \Big ),
\quad x \in \mathbb{Z}[\omega].
$$
We also define the following non-Archimedean
analogue of a Bessel $K_{1/3}$--transform,
\begin{equation} \label{tildepsi}
\widetilde{\psi}(u) := \sum_{\substack{x \pmod{r}}} \psi(x) S_{1/3}(x,u; r), \quad u \in \mathbb{Z}[\omega],
\end{equation}
where
\begin{equation} \label{S13}
S_{1/3}(x, u; r) := \left( \frac{\lambda}{r} \right)_3 \sum_{\substack{d \pmod {r} \\ (d,r) = 1 \\ (\lambda^4 d) (\lambda^4 a) \equiv 1 \pmod{r}}} 
\Big ( \frac{a}{r} \Big )_{3} \check{e} \Big ( \frac{xd+ua}{r} \Big ),
\end{equation}
is the cubic Kloosterman sum. Note that it is convenient for us to have the $\lambda^4=9$ factors built into the congruence
in the cubic Kloosterman sum. These factors naturally appear when we use the automorphy of the $\overline{F}_j$ (with multiplier $\overline{\chi}$)
on the group $\Gamma_1(\lambda^4)$ in the proof of the following result.

To isolate twists of the cubic Gauss sums, we need to analyse 
\begin{equation} \label{psitwistS}
\overline{F_1}(w;\psi)_S:=
\sum_{\mu \in S}
\overline{\tau(\mu)} \psi(\lambda^4 \mu) v K_{\frac{1}{3}}(4 \pi |\mu| v) \check{e}(\mu z), \quad w \in \mathbb{H}^3.
\end{equation}

\begin{lemma} \label{functwisteq}
Suppose $r \in \mathbb{Z}[\omega]$ with $r \equiv 1 \pmod{3}$,
and $\psi$ is a sequence on $\mathbb{Z}[\omega]$ that is periodic modulo $r$.
Suppose that $\widehat{\psi}$ is supported only on residue classes coprime to $r$.
For $w=(z,v) \in \mathbb{H}^3$,
we have 
\begin{align} \label{funceqstatement}
\overline{F_1}(w;\psi)_S
  = \frac{1}{3 N(r)} \Big \{  \overline{F_1}+\omega \overline{F^{\star}_{19}}+ \omega^2 \overline{F^{\star}_{10}}  \Big \}
\Big ( \Big(-\frac{\overline{z}}{r^2(|z|^2+v^2)}, 
\frac{v}{|r|^2(|z|^2+v^2)} \Big); \widetilde{\psi} \Big ). 
\end{align}
\end{lemma}
\begin{proof}
Fourier inversion and Lemma \ref{Ssieve} imply that 
\begin{align} \label{conjthetatwist}
 \overline{F_1}(w;\psi)_S
 &=\frac{1}{N(r)}
 \sum_{\substack{d \pmod{r}  \\ (d,r)=1 }} \widehat{\psi}(d)
\overline{F_1} \Big(z-\frac{\lambda^4 d}{r},v \Big)_S \nonumber \\
&=\frac{1}{3 N(r)}
\sum_{\substack{d \pmod{r}  \\ (d,r)=1 }} \widehat{\psi}(d)
 \Big(\overline{F_1}\Big(z-\frac{\lambda^4 d}{r},v \Big)+\omega \overline{F_2} \Big(z-\frac{\lambda^4 d}{r},v \Big)
+\omega^2 \overline{F_3} \Big(z-\frac{\lambda^4 d}{r},v \Big) \Big).
\end{align}
Given our $r \equiv 1 \pmod{3}$, and each $d$ in \eqref{conjthetatwist}, we have
$(r, \lambda^4 d)=1$. Thus
there exists
\begin{equation*}
\begin{pMatrix}
{r} {-\lambda^4 a}
{\lambda^4 d} {b}
\end{pMatrix} \in \Gamma_1(3),
\end{equation*}
and hence there exists
\begin{equation} \label{gammadefn}
\gamma:=\begin{pMatrix}
{0} {1} 
{-1} {0}
\end{pMatrix}
\begin{pMatrix}
{r} {-\lambda^4 a}
{\lambda^4 d} {b}
\end{pMatrix}
=
\begin{pMatrix}
{\lambda^4 d}  {b} 
{-r}  {\lambda^4 a}
\end{pMatrix} \in \Gamma_2.
\end{equation} 
Note that we also used \eqref{Gamma2} implicitly in the above display.

A direct computation shows that
\begin{equation} \label{gammarel1}
\gamma \Big(\frac{\lambda^4 a}{r}-\frac{\overline{z}}{r^2(|z|^2+v^2)},
\frac{v}{|r|^2(|z|^2+v^2)} \Big)=\Big(z-\frac{\lambda^4 d}{r},v \Big).
\end{equation}
We now carefully factorise the $\gamma$ in \eqref{gammadefn} as a word in
$P$, $E$ and $T$ so that \eqref{gammarel1} and automorphy can be used in 
\eqref{conjthetatwist}. 
For each $\kappa=m+n \omega \in \mathbb{Z}[\omega]$, $m,n \in \mathbb{Z}$,
let
\begin{equation*}
A(\kappa):=PT^{-m} P T^{-m+n}P
=\begin{pMatrix}
1  \kappa 
0  1 
\end{pMatrix}.
\end{equation*}
For each $r,b \in \mathbb{Z}[\omega]$ occurring in \eqref{gammadefn}, let
\begin{equation*}
W(r,b):=E^3 A(r) E A(b) E A(r) = 
\begin{pMatrix}
  {b}  {-1 + br} 
  {1 - br}  {2 r - b r^2}
\end{pMatrix}.
\end{equation*}
Then
\begin{equation} \label{gammatildedef}
W(r,b) E \gamma = 
\begin{pMatrix}
  {-9d + br + 9bdr}  {- b - 9 a b+ b^2 r} 
  {r + 18 dr - br^2 - 9 b d r^2}  {- 9 a + 2 br +9 a b r - b^2 r^2}
\end{pMatrix} \\
=: \widetilde{\gamma}
\in \Gamma_1(9).
\end{equation}
Equivalently,
\begin{equation} \label{gammarel3}
\gamma=E^3 W(r,b)^{-1} \widetilde{\gamma}. 
\end{equation}
To obtain \eqref{Fjtransform}
immediately below we use \eqref{gammarel1}, \eqref{gammarel3}, \eqref{aut1}, \eqref{aut2}, 
and the fact that each $F_j$ is automorphic on $\Gamma_1(9)$
with multiplier $\chi$. For each $j=1,2,3$,
\begin{align} \label{Fjtransform}
\overline{F_j} \Big(z-\frac{\lambda^4 d}{r},v \Big)&=
\overline{\chi} (g_j(E^3 W(r,b)^{-1})) \cdot \overline{\chi}(\widetilde{\gamma}) \nonumber  \\
& \times \overline{F_{k_j(E^3 W(r,b)^{-1})}}
\Big(\frac{\lambda^4 a}{r}-\frac{\overline{z}}{r^2(|z|^2+v^2)}, \frac{v}{|r|^2(|z|^2+v^2)}  \Big). 
\end{align}
We claim that
\begin{align}
  k_1\big(E^3 W(r,b)^{-1} \big) & = 1; \label{k1} \\
  k_2 \big(E^3 W(r,b)^{-1}  \big) & = 19; \label{k2}  \\
  k_3\big( E^3 W(r,b)^{-1} \big) & = 10. \label{k3}
\end{align}
We will use \eqref{simpleimpl} and the computations below to establish \eqref{k1}--\eqref{k3}.
Observe that  
\begin{align}  \label{comp1}
\gamma_1 E^3 W(r,b)^{-1} \gamma_{1}^{-1} &=
\begin{pMatrix}
  {- 1 + br}  {b} 
  {- 2 r + br^2}  {-1 + br}
\end{pMatrix} \nonumber \\
& \equiv E^3 \pmod{3}.
\end{align}
Thus $\gamma_1 E^3 W(r,b)^{-1} \gamma_1^{-1} \in \Gamma_2$, and \eqref{k1} follows.
Note that \eqref{k2} (resp. \eqref{k3}) follow similarly from 
\begin{align}  \label{comp2}
\gamma_2 E^3 W(r,b)^{-1} \gamma_{19}^{-1} 
 & = \begin{pMatrix}
  {(\omega + 1)b  - 2 \omega r + \omega b r^2}  {- \omega + b + \omega b r}
  {-(\omega + 1) - 2 r + (\omega + 1) br  + b r^2}  {- 1 + b r}
  \end{pMatrix} \nonumber \\
  & \equiv T^{-1} E^3 \pmod{3}, 
\end{align}
(resp.)
\begin{align} \label{comp3}
 \gamma_{3} E^3  W(r,b)^{-1} \gamma_{10}^{-1} & = 
 \begin{pMatrix}
  {\omega  - \omega b  + 2 \omega r  - \omega br - \omega b r^2}  {\omega + b - \omega br} 
  {\omega - 2 r - \omega b r +  b r^2}  {-1 +  b r} 
  \end{pMatrix}  \nonumber \\
  &\equiv E^3  \pmod{3}. 
\end{align}

We now compute the automorphy factor $\overline{\chi} (g_j(E^3 W(r,b)^{-1})) \cdot \overline{\chi}(\widetilde{\gamma})$ 
for each $j=1,2,3$ in \eqref{Fjtransform}. 
The displays \eqref{comp1}--\eqref{comp3} prove that
\begin{align*}
g_1 (E^3 W(r,b)^{-1})&=E^3 W(r,b)^{-1};  \\
g_2 (E^3 W(r,b)^{-1})&=\gamma_2 E^3 W(r,b)^{-1} \gamma_{19}^{-1};  \\
g_3 (E^3 W(r,b)^{-1})&=\gamma_{3} E^3  W(r,b)^{-1} \gamma_{10}^{-1},
\end{align*}
respectively. In order to use either \eqref{chidef} or \eqref{altchi} to compute
$\overline{\chi}(g_j(E^3 W(r,b)^{-1}))$ 
we note that $\chi |_{\operatorname{SL}_2(\mathbb{Z})}=1$
and $Eg_1 (E^3 W(r,b)^{-1}), ET g_2 (E^3 W(r,b)^{-1}),
Eg_3 (E^3 W(r,b)^{-1}) \in \Gamma_1(3)$ (cf. \eqref{comp1}--\eqref{comp3}).
We also repeatedly use the determinant equation $81ad+br=1$ from 
\eqref{gammadefn}. Now,
\begin{align} \label{chig1}
\overline{\chi} (g_1(E^3 W(r,b)^{-1}) )
&=\overline{\chi} (E g_1(E^3 W(r,b)^{-1}) ) \nonumber \\
&=\begin{cases}
\overline{\left(\frac{1-br}{b}  \right)_3} & \text{if} \quad -1+br \neq 0  \\
 1 & \text{otherwise}
\end{cases} \nonumber \\
&=1.
\end{align}
Similar computations yield 
\begin{align} \label{chig3}
\overline{\chi} (g_3(E^3 W(r,b)^{-1}) )
&=\overline{\chi} (E g_3(E^3 W(r,b)^{-1}) ) \nonumber \\
&=\begin{cases}
\overline{\left(\frac{1-br}{\omega+b-\omega br}  \right)_3} & \text{if} \quad \omega(1-b+2r-br-br^2) \neq 0 \\
1 & \text{otherwise}
\end{cases} \nonumber \\
&=1, 
\end{align}
and
\begin{align} \label{chig2}
\overline{\chi} & (g_2(E^3 W(r,b)^{-1}) ) \nonumber \\
&=\overline{\chi} (ET g_2(E^3 W(r,b)^{-1}) ) \nonumber \\
&=\begin{cases}
\overline{\left( \frac{1-br}{-\omega+b+\omega br+br-1} \right)_3} & 
\text{if} \quad (\omega+1)b-2 \omega r+\omega b r^2 \\
&\qquad \qquad \qquad -(\omega+1)-2r+(\omega+1)br +br^2 \neq 0 \\
1 & \text{otherwise}
\end{cases} \nonumber \\
&=1.
\end{align}
We also have
\begin{align} \label{tildefac}
\overline{\chi}(\widetilde{\gamma})&=\overline{\left( \frac{r+18dr-br^2-9bdr^2}{-9d+br+9bdr} \right)_3}
=\overline{\left( \frac{81adr+18dr-9bdr^2}{-9d+br+9bdr} \right)_3} \nonumber \\
&=\overline{\left( \frac{9d}{-9d+br+9bdr} \right)_3} \cdot \overline{ \left(\frac{9ar+2r-br^2}{-9d+br+9bdr}  \right)_3  } \nonumber \\
&=\overline{\left( \frac{9d}{1} \right)_3} \cdot \overline{ \left(\frac{-9d+br+9bdr}{9ar+2r-br^2}  \right)_3  } \quad \quad \text{(by cubic reciprocity)}  \nonumber \\
&=\overline{\left(\frac{-9d+br+9bdr}{r} \right)} \cdot \overline{\left(\frac{-9d+br+9bdr}{9a+2-br} \right)_3} \nonumber \\
&=\overline{\left(\frac{9d}{r} \right)_3} \cdot \overline{\left( \frac{1-81ad-729ad^2}{9a+1+81ad} \right)_3 }
=\left(\frac{9a}{r} \right)_3 \cdot \overline{\left( \frac{9d+1}{9a+1+81ad} \right)_3 }  \nonumber \\
&=\left(\frac{9a}{r} \right)_3 \cdot \overline{\left(\frac{9d+1}{1} \right)_3} 
=\left( \frac{\lambda^4 a}{r} \right)_3. 
\end{align}

We combine \eqref{Fjtransform}--\eqref{tildefac} 
in \eqref{conjthetatwist}. Note that $\overline{F_{19}}=\overline{F^{\star}_{19}}$
and $\overline{F_{10}}=\overline{F^{\star}_{10}}$ by \eqref{Fjstar}.
We then use the Fourier expansions
\eqref{fjfourier} to open $\overline{F_1}$, $\overline{F^{\star}_{19}}$ and $\overline{F^{\star}_{10}}$,
and assembling the sum over $d$ (equivalently $a$) shows that
\begin{align*}
\overline{F}_1(w;\psi)_S=\frac{1}{3 N(r)} \Big ( \overline{F_1} + \omega \overline{F^{\star}_{19}} + \omega^2 \overline{F^{\star}_{10}} \Big ) \Big ( -\frac{\overline{z}}{r^2(|z|^2+v^2)}, \frac{v}{|r|^2(|z|^2+v^2)} ; \Psi \Big),
\end{align*}
where
\begin{equation} \label{Psidefn}
\Psi(u):= \sum_{\substack{d \pmod{r} \\ (d,r) = 1 \\ (\lambda^4 d) (\lambda^4 a) \equiv 1 \pmod{r}}} \widehat{\psi}(d) \Big ( \frac{\lambda^4 a}{r} \Big )_{3} \check{e} \Big ({\frac{a u}{r}} \Big ), \quad u \in \mathbb{Z}[\omega].
\end{equation}
After opening $\widehat{\psi}(d)$ using the definition \eqref{fouriertrans}, 
and interchanging the order of summation,
we readily see that $\Psi(u) = \widetilde{\psi}(u)$ for all $u \in \mathbb{Z}[\omega]$. 
\end{proof}
For the coming lemma it will be instructive to open the definition of $\check{e}(\cdot)$,
\begin{equation*}
\check{e}(\mu z)=e (\mu z+ \overline{\mu z}), \quad z \in \mathbb{C}.
\end{equation*}
For $\ell \in \mathbb{Z} \setminus \{0\}$ and $1 \leq j \leq 27$, let
$$
\overline{F_j}(w; \psi, \ell) = \overline{F_j^{\star}}(w; \psi, \ell) := \begin{cases}
  \sum_{\mu} \overline{d_j(-\mu)} \psi(\lambda^4 \mu) \mu^{\ell} v K_{\tfrac 13} (4\pi |\mu| v) e(\mu z +\overline{\mu z} ),  & \text{ if } \ell > 0 \\
  \sum_{\mu} \overline{d_j(-\mu)} \psi(\lambda^4 \mu) \overline{\mu}^{|\ell|} v K_{\tfrac 13}(4\pi |\mu| v) e (\mu z+\overline{\mu z}), & \text{ if } \ell < 0
\end{cases}.
$$
For $\ell=0$, $\overline{F_j}(w; \psi,0):=\overline{F_j}(w; \psi)$ from before.
We have
\begin{equation} \label{ellexp}
\overline{F_j}(w; \psi,\ell) =\frac{1}{(2\pi i)^{|\ell|}}
\begin{cases}
 \big ( \frac{\partial}{\partial z} \big )^{\ell} \overline{F_j}(w; \psi), & \text{if} \quad  \ell>0 \\
 \big ( \frac{\partial}{\partial \overline{z}} \big )^{|\ell|} \overline{F_j}(w; \psi), & \text{if} \quad \ell<0,
\end{cases}
\quad w=(z,v) \in \mathbb{H}^3.
\end{equation}

We apply differential operators in the proof of the next lemma. Thus we remind the reader that 
$\overline{F_j}(w; \psi, \ell)=\overline{G_{j}}(z,\overline{z},v;\psi,\ell)$ is a function of $z,\overline{z}$ and $v$,
although the $\overline{F_j}$ notation suppresses this.

With these observations in mind we deduce the following Corollary. 

\begin{corollary} \label{autcor}
 Suppose $r \in \mathbb{Z}[\omega]$ with $r \equiv 1 \pmod{3}$,
and $\psi$ is a sequence on $\mathbb{Z}[\omega]$ that is periodic modulo $r$.
Suppose that $\widehat{\psi}$ is supported only on residue classes coprime to $r$.
 By abuse of notation, write
  $F_j((0,v); \psi,\ell)$ as $F_j(v; \psi,\ell)$ for all $v>0$, $1 \leq j \leq 27$ and $\ell \in \mathbb{Z}$.
 Then we have 
  $$ \overline{F_1}(v; \psi , \ell)_{S} = \frac{(-1)^{\ell}}{3 N(r)^{1 + |\ell|} v^{2 |\ell|}} \Big( \frac{\overline{r}}{r} \Big)^{-\ell}  \Big ( \delta_{\ell \neq 0 } \overline{F_1^{\star}} + \delta_{\ell = 0} \overline{F_1} + \omega \overline{F^{\star}_{19}} + \omega^2 \overline{F^{\star}_{10}} \Big ) \Big ( \frac{1}{|r|^2 v} ; \widetilde{\psi} , - \ell \Big ),
  $$
 or equivalently,
  $$
  \Big ( \overline{F_1^{\star}} \delta_{\ell \neq 0} + \overline{F_1} \delta_{\ell = 0} + \omega \overline{F^{\star}_{19}} + \omega^2 \overline{F^{\star}_{10}} \Big ) (v; \widetilde{\psi}, -\ell) = \frac{3 \cdot (-1)^{\ell}}{N(r)^{|\ell| - 1} v^{2|\ell|}} \Big(\frac{\overline{r}}{r} \Big)^{\ell} \overline{F_1} \Big ( \frac{1}{|r|^2 v}; \psi, \ell \Big ).
  $$
\end{corollary}
\begin{proof}
 Setting $z = 0$ in \eqref{funceqstatement} gives the claim when $\ell = 0$. 
  If $\ell > 0$, we write $|z|^2 = z \overline{z}$ and
  apply the operator $\frac{1}{(2\pi i)^{\ell}}  \big ( \frac{\partial}{\partial z} \big )^{\ell} \big |_{z = 0}$
to  both sides of \eqref{funceqstatement}. 
A computation with the chain rule yields
  \begin{align*}
    \overline{F_1}((0,v); \psi, \ell)_{S} = \frac{(-1)^{\ell}}{3 N(r)^{1 + \ell} v^{2 \ell}} \Big( \frac{\overline{r}}{r} \Big)^{-\ell}
    \cdot \Big ( \overline{F_1^{\star}} + \omega \overline{F^{\star}_{19}} + \omega^2 \overline{F^{\star}_{10}} \Big ) \Big ( \Big ( 0, \frac{1}{|r|^2 v} \Big ) ; \widetilde{\psi}, -\ell \Big ).
  \end{align*}
 
 When $\ell < 0$ a similar argument with the operator $\frac{1}{(2\pi i)^{|\ell|}}  \big ( \frac{\partial}{\partial \overline{z}} \big )^{|\ell|} \Big |_{z = 0}$
yields the analogous result.
\end{proof}

\subsection{Poles and Dirichlet series}
Let $\psi$ be as in Corollary \ref{autcor}.
For $\Re(s)>1$, $\ell \in \mathbb{Z}$ and $1 \leq j \leq 27$, consider the family of Dirichlet series
\begin{align}  \label{DF1bar}
\mathcal{D}(s,\overline{F^{\star}_j};\psi, \ell)&:=\sum_{\mu} 
\frac{\overline{d_j(-\mu)} \psi(\lambda^4 \mu) \big ( \frac{\mu}{|\mu|} \big )^{\ell}}{N(\mu)^s}; \nonumber \\
\mathcal{D}(s,{\overline{F_1}};\psi, \ell)_S&:=
\sum_{\mu \in S} \frac{\overline{\tau(\mu)} \psi(\lambda^4 \mu) \big ( \frac{\mu}{|\mu|} \big )^{\ell}}{N(\mu)^s}.
\end{align}

For $\Re(s)>1$, we introduce the integral transforms
\begin{align*}
\Lambda (s,\overline{F^{\star}_j};\psi, \ell)&:= \int_{0}^{\infty} \overline{F^{\star}_j} (v;\psi, \ell) v^{2s+|\ell|-2} dv; \\
\Lambda (s,{\overline{F_1}};\psi, \ell)_S&:= \int_{0}^{\infty} {\overline{F}_1} (v;\psi, \ell)_S v^{2s+|\ell|-2} dv,
\end{align*}
where by abuse of notation we wrote $F_{i}(v; \psi, \ell) = F_{i}((0, v); \psi, \ell)$. 
In the case $\ell = 0$ we will omit the index $\ell$ from the notation. 

\begin{lemma} \label{mellin}
For $\Re(s)>1$ we have
\begin{align*}
  \Lambda(s,\overline{F^{\star}_j};\psi, \ell)&=\frac{1}{4} (2 \pi)^{-2s-|\ell|} \Gamma \Big (s + \frac{|\ell|}{2} - \frac{1}{6} \Big )
  \Gamma \Big (s + \frac{|\ell|}{2} + \frac{1}{6} \Big )
\mathcal{D}(s,\overline{F^{\star}_j};\psi, \ell); \\
\Lambda(s,\overline{F_1};\psi, \ell)_S&=\frac{1}{4} (2 \pi)^{-2s-|\ell|} \Gamma \Big ( s + \frac{|\ell|}{2} - \frac{1}{6} \Big )
\Gamma \Big (s + \frac{|\ell|}{2} + \frac{1}{6} \Big )
\mathcal{D}(s,\overline{F_1};\psi, \ell)_S.
\end{align*}
\end{lemma}
\begin{proof}
The proofs of both identities are virtually identical, so we prove the latter, with $\ell < 0$. 
For $\Re(s)>1$ we have
\begin{align}
 \Lambda(s,\overline{F}_1;\psi, \ell)_S 
&=\int_{0}^{\infty} \sum_{\mu \in S}
\overline{\tau(\mu)} \psi(\lambda^4 \mu) \overline{\mu}^{|\ell|} K_{\frac{1}{3}}(4 \pi |\mu| v) v^{2s+|\ell|-1} dv \nonumber \\
&=\frac{1}{(4 \pi)^{2s + |\ell|}} \sum_{\mu \in S} 
\frac{\overline{\tau(\mu)} \psi(\lambda^4 \mu) \big ( \frac{\overline{\mu}}{|\mu|} \big )^{|\ell|}}{N(\mu)^s}
\int_{0}^{\infty} K_{\frac{1}{3}}(T) T^{2s+|\ell|-1} dT \nonumber \\
&=\frac{1}{4} (2 \pi)^{-2s-|\ell|} \Gamma \Big (s + \frac{|\ell|}{2} - \frac{1}{6} \Big ) \Gamma \Big (s + \frac{|\ell|}{2} + \frac{1}{6} \Big )
\sum_{\mu \in S}
\frac{\overline{\tau(\mu)} \psi(\lambda^4 \mu) \big ( \frac{\mu}{|\mu|} \big )^{\ell}}{N(\mu)^s}. 
\label{mellineval}
\end{align}
The interchange of summation and integration above for $\Re(s)>1$ is justified by absolute convergence
(cf. \cite[(10.25.3),(10.30.2)]{NIST:DLMF}). Furthermore,
\eqref{mellineval} follows from \cite[(10.43.19)]{NIST:DLMF}.
\end{proof}

\begin{prop} \label{polefunc}
The completed Dirichlet series $\Lambda(\overline{F_1}, s; \psi)_{S}$ admits a meromorphic continuation to the whole complex plane $\mathbb{C}$. It has a unique pole 
 (that is simple) at $s = 5/6$, with residue
 \begin{equation} \label{res56}
 \underset{s = \tfrac 56}{\text{Res }} \Lambda(s,\overline{F_1}; \psi)_{S} = \frac{\sigma \widetilde{\psi}(0)}{6 N(r)^{5/3}},
 \end{equation}
 where $\sigma=3^{5/2}/2$ is as in \eqref{sigma}.
 In particular,
 \begin{equation} \label{resD56}
 \underset{s = \tfrac 56}{\text{Res }} \mathcal{D}(s, \overline{F_1}; \psi)_{S} = \frac{2 (2\pi)^{5/3} \sigma \widetilde{\psi}(0)}{3 \Gamma(\tfrac 23) N(r)^{5/3}}. 
\end{equation}
 For $\ell \neq 0$ the Dirichlet series $\Lambda(\overline{F_1}, s; \psi, \ell)_{S}$ is entire. 
 Moreover, for all $\ell \in \mathbb{Z}$ we have the functional equation
 \begin{equation} \label{levelfunc}
 3 (-1)^{\ell} N(r)^{2s} \Big( \frac{\overline{r}}{r} \Big)^{\ell} 
 \Lambda(s,\overline{F_1}; \psi, \ell)_{S} = \Lambda(1 - s, \overline{F^{\star}_1} + \omega \overline{F^{\star}_{19}} + \omega^2 \overline{F^{\star}_{10}}; \widetilde{\psi}, - \ell).
 \end{equation}
 This functional equation also determines the poles of $\Lambda(s, \overline{F^{\star}_1} + \omega \overline{F^{\star}_{19}} + \omega^2 \overline{F^{\star}_{10}}; \widetilde{\psi})$. 
\end{prop}

\begin{proof}
For $\Re s>1$ we have
  \begin{equation} \label{splitint1}
    \Lambda(s,\overline{F_1}; \psi, \ell)_{S}= \int_{0}^{N(r)^{-1}} \overline{F_1}(v; \psi, \ell)_{S} v^{2s + |\ell| - 2} dv + \int_{N(r)^{-1}}^{\infty} \overline{F_1}(v; \psi, \ell)_{S} v^{2s + |\ell| - 2} dv.
  \end{equation}
When $\ell=0$ observe that $\overline{F_1}(v; \psi, \ell)_{S}$ has exponential decay at $\infty$ by \eqref{Sdefn}
and \eqref{psitwistS}. The same claim holds when $\ell \neq 0$
 using \eqref{ellexp} (termwise differentiation of the Fourier series) and the same reasoning as before.
Thus the second integral in \eqref{splitint1} has analytic continuation
to an entire function.

 Let 
  \begin{equation} \label{Gdefuse}
    \overline{G}(w; \widetilde{\psi}, \ell) := \begin{cases} (\overline{F_1} +\omega \overline{F^{\star}_{19}} + \omega^2 \overline{F^{\star}_{10}})(w; \tilde{\psi}, \ell) & \text{ if } \ell = 0 \\
      (\overline{F_1^{\star}} +\omega \overline{F^{\star}_{19}} + \omega^2 \overline{F^{\star}_{10}})(w; \tilde{\psi}, \ell) & \text{ if } \ell \neq 0
      \end{cases},
      \quad w \in \mathbb{H}^3.
  \end{equation}
Using \eqref{Gdefuse} and Corollary \ref{autcor} gives
  \begin{align}
    \int_{0}^{N(r)^{-1}} & \overline{F_1}(v; \psi, \ell)_{S} v^{2s + |\ell| - 2} dv \nonumber \\
    &  = \frac{(-1)^{\ell}}{3 N(r)^{1 + |\ell|}} \Big( \frac{\overline{r}}{r} \Big)^{-\ell}
    \int_{0}^{N(r)^{-1}} \overline{G} \Big ( \frac{1}{v |r|^2} ; \tilde{\psi} , - \ell \Big )  v^{2s - |\ell| - 2} dv \label{splitintstar} \\ 
    & = \frac{\sigma \widetilde{\psi}(0) N(r)^{-2s}}{6} \cdot \frac{\delta_{\ell = 0}}{s - \tfrac 56} \nonumber  \\
   & + \frac{(-1)^{\ell} N(r)^{-2s}}{3} \Big( \frac{\overline{r}}{r} \Big)^{-\ell} \int_{1}^{\infty} \Big ( \overline{F_1^{\star}} + \omega \overline{F_{19}^{\star}} + \omega^2 \overline{F_{10}^{\star}} \Big ) (v ; \widetilde{\psi}, -\ell) v^{|\ell| - 2s} dv.  \label{splitint2}
  \end{align}
  When $\ell=0$ observe that $(\overline{F_1^{\star}} + \omega \overline{F^{\star}_{19}} + \omega^2 \overline{F^{\star}_{10}})(\cdot,\widetilde{\psi}, -\ell)$ has exponential decay at $\infty$ by \eqref{zerotwist}, Appendix \ref{appendix} (the expressions for $d_j(\mu)$), \eqref{cubiccoeff}, \eqref{cubiccoeff1}, and \eqref{cubiccoeff2}. The same claim holds when $\ell \neq 0$
  using \eqref{ellexp} (termwise differentiation of the Fourier series)
  and the same reasoning as before. Thus the integral in \eqref{splitint2} has analytic continuation to an entire function. This gives the meromorphicity and entirety claims in the Proposition, as well as \eqref{res56}. Observe that \eqref{resD56} follows from \eqref{res56} and Lemma \ref{mellin}.
   
  We now prove the functional equation \eqref{levelfunc}. 
 From \eqref{splitint1} and \eqref{splitintstar} we found that  
  \begin{align} \label{firstfunceq}
  \Lambda(s,\overline{F_1}; \psi,\ell)_{S} = \frac{(-1)^{\ell} N(r)^{- 2s}}{3} \Big( \frac{\overline{r}}{r} \Big)^{-\ell}
  \int_{1}^{\infty} \overline{G} (v ; \widetilde{\psi}, -\ell) v^{|\ell| - 2s} dv + \int_{N(r)^{-1}}^{\infty} \overline{F_1}(v; \psi, \ell)_S v^{2s + |\ell| - 2} dv.
  \end{align}
  We now repeat a similar argument, but instead start with 
  \begin{equation} \label{Gstarcombo}
   \overline{G^{\star}}(w; \widetilde{\psi}, \ell)=(\overline{F_1^{\star}} + \omega \overline{F_{19}^{\star}} + \omega^2 \overline{F_{10}^{\star}})(w; \widetilde{\psi}, \ell), 
   \quad \text{for all} \quad w \in \mathbb{H}^3 \quad \text{and} \quad \ell \in \mathbb{Z}.
  \end{equation}
  For $\Re s>1$ we have 
  \begin{equation} \label{LambdaGstar}
    \Lambda(s,\overline{G^{\star}}; \widetilde{\psi}, -\ell)= \int_{0}^{1} \overline{G^{\star}}(v; \widetilde{\psi},-\ell) v^{2s + |\ell| - 2} dv 
    +\int_{1}^{\infty} \overline{G^{\star}}(v; \widetilde{\psi},-\ell) v^{2s + |\ell| - 2} dv. 
  \end{equation}  
For $\Re s>1$ we have 
 \begin{equation} \label{G1}
  \int_{0}^{1} \overline{G^{\star}}(v; \widetilde{\psi}, -\ell) v^{2s + |\ell| - 2} dv = \int_{0}^{1} \overline{G}(v; \widetilde{\psi}, -\ell) v^{2s + |\ell| - 2} dv - \delta_{\ell = 0}
   \cdot \frac{3\sigma \widetilde{\psi}(0)}{6s - 1}.
 \end{equation}
 Then \eqref{G1} holds for all $s \in \mathbb{C}$ by meromorphic continuation. Similarly,
 for $\Re s<-1$, we have
 \begin{equation} \label{G2}
  \int_{1}^{\infty} \overline{G^{\star}}(v; \widetilde{\psi}, -\ell) v^{2s + |\ell| - 2} dv= \int_{1}^{\infty} \overline{G}(v; \widetilde{\psi}, -\ell) v^{2s + |\ell| - 2} dv + \delta_{\ell = 0}  \frac{3 \sigma \widetilde{\psi}(0)}{6 s - 1}.
 \end{equation}
 Then \eqref{G2} holds for all $s \in \mathbb{C}$ by meromorphic continuation.
 Insertion of \eqref{G1} and \eqref{G2} into \eqref{LambdaGstar} gives 
 \begin{equation} \label{LambdaGstar2}
   \Lambda(s,\overline{G^{\star}}; \widetilde{\psi}, - \ell)= \int_{0}^{1} \overline{G}(v; \widetilde{\psi}, - \ell) v^{2s + |\ell| - 2} dv
    + \int_{1}^{\infty} \overline{G}(v; \widetilde{\psi}, -\ell) v^{2s + |\ell| - 2} dv,
 \end{equation}
 where both integrals are to be interpreted as the meromorphic continuations of the original integrals. 
 Using Corollary \ref{autcor} we obtain
  \begin{align} \label{aut2cor}
    \int_{0}^{1} \overline{G}(v; \widetilde{\psi}, -\ell) v^{2s + |\ell| - 2} dv & = 3 (-1)^{\ell} N(r)^{1 - |\ell|} \Big( \frac{\overline{r}}{r} \Big)^{\ell} \int_{0}^{1} \overline{F_1} \Big ( \frac{1}{|r|^2 v} ; \psi , \ell \Big )_S v^{2s - |\ell| - 2} dv \nonumber \\
    & = 3 (-1)^{\ell} N(r)^{2 - 2s} \Big( \frac{\overline{r}}{r} \Big)^{\ell} \int_{N(r)^{-1}}^{\infty} \overline{F_1} (v; \psi,\ell)_S v^{|\ell| -2s} dv.
  \end{align}
 Substitution of \eqref{aut2cor} into \eqref{LambdaGstar2} gives
  $$
  \Lambda(s,\overline{G^{\star}}; \widetilde{\psi}, -\ell) = 3 (-1)^{\ell} N(r)^{2 - 2s} \Big( \frac{\overline{r}}{r} \Big)^{\ell} \int_{N(r)^{-1}}^{\infty} \overline{F_1} (v; \psi, \ell)_S v^{|\ell|-2s} dv + \int_{1}^{\infty} \overline{G}(v; \widetilde{\psi}, - \ell)v^{2s + |\ell| - 2} dv.
  $$
Equivalently,
  \begin{equation} \label{secondfunceq}
  \Lambda(1-s,\overline{G^{\star}}; \widetilde{\psi}, - \ell) = 3 (-1)^{\ell} N(r)^{2 s} \Big( \frac{\overline{r}}{r} \Big)^{\ell}
  \int_{N(r)^{-1}}^{\infty} \overline{F_1} (v; \psi, \ell)_S v^{2s + |\ell| - 2} dv 
  + \int_{1}^{\infty} \overline{G}(v; \widetilde{\psi},-\ell)v^{|\ell| - 2s} dv.  
  \end{equation}
  
  After combining \eqref{firstfunceq} and \eqref{secondfunceq} we obtain
  \begin{equation*}
3 (-1)^{\ell} N(r)^{2s} \Big(\frac{\overline{r}}{r} \Big)^{\ell} \Lambda(s,\overline{F_1}; \psi, \ell)_{S}= \Lambda(1-s,\overline{G^{\star}}; \widetilde{\psi}, - \ell),
  \end{equation*}
  as required.
\end{proof}

\subsection{Sieving for $\tilde{g}(c)$}
Let $r \in \mathbb{Z}[\omega]$ satisfy $r \equiv 1 \pmod{3}$, $\psi$ be a primitive character to modulus $r$, 
and $\ell \in \mathbb{Z}$. Then let
$$
\zeta_{\mathbb{Q}(\omega)}(s; \psi, \ell):= \sum_{\substack{d \equiv 1 \pmod{3}}}
\frac{\psi(d) \big ( \frac{d}{|d|} \big )^{\ell}}{N(d)^s}, \quad \Re(s)>1.
$$
In the case $\ell = 0$ we omit $\ell$ from the notation and write $\zeta_{\mathbb{Q}(\omega)}(s; \psi)$. 
We denote the principal character modulo $r$ by $\boldsymbol{1}_r$.
Any $J \unlhd \mathbb{Z}[\omega]$ with $(J,3)=1$ 
has a unique generator $d \equiv 1 \pmod{3}$. 
Thus when $\ell=0$ and $\psi=\mathbf{1}_r$, the above $L$-function coincides with the Dedekind $\zeta$-function of $\mathbb{Q}(\omega)$, 
except at the local factors of primes dividing $(\lambda r )$.  Note that 
$\zeta_{\mathbb{Q}(\omega)}(s; \psi, \ell)$ has standard meromorphic continuation
to all of $\mathbb{C}$; the only case when this function is not holomorphic is when $\psi = \mathbf{1}_{r}$ is the principal character and $\ell = 0$, in that case there is a unique simple pole at $s=1$. Standard functional equations for these $L$-functions can be found in \cite[\S 3.3]{Miy}.

\begin{lemma} \label{sievedir}
  Let $r \in \mathbb{Z}[\omega]$ with $r \equiv 1 \pmod{3}$, and
  $\psi$ be a primitive cubic Dirichlet character on $\mathbb{Z}[\omega]$ to modulus $r$.
 For $\ell \in \mathbb{Z}$ and 
  $\Re s>1$ we have
  \begin{equation} \label{transition}
 (-1)^{\ell} i^{-\ell} \sum_{\substack{c \equiv 1 \pmod{3}}} \frac{\tilde{g}(c) \psi(\lambda c) \big ( \frac{c}{|c|} \big )^{\ell}}{N(c)^s} = 
  \frac{3^{-5/2 - 3s} \cdot \mathcal{D}(s, \overline{F_1}; \psi, \ell)_{S}}{\zeta_{\mathbb{Q}(\omega)}(3s - \tfrac 12;\mathbf{1}_{r}, 3 \ell)}, 
 \end{equation}
  where $\mathcal{D}(s, \overline{F_1}; \psi, \ell)_{S}$ is as in \eqref{DF1bar}.
 \end{lemma}
 
 \begin{remark}
 Note that we have abused notation in the results and proofs that follow: $\mu \in \mathbb{Q}(\omega)$ 
 is used to index Fourier coefficients of various automorphic forms, and $\mu(\cdot)$
 denotes the M\"{o}bius function on $\mathbb{Z}[\omega]$. Meanings should be clear from context.
 \end{remark}

\begin{proof}
The Dirichlet coefficients on
the right side on \eqref{transition} have support contained
in $ \mathbb{Z}[\omega] \setminus \{0\}$ since $3^{-3s}=N(\lambda^{3})^{-s}$.
The $\nu$th Dirichlet coefficient on the right side of \eqref{transition} is given by
  \begin{equation} \label{rightside} 
  \frac{1}{3^{5/2}} \Big ( \frac{\nu/\lambda^3}{|\nu/\lambda^3 |} \Big )^{\ell} \sum_{\substack{\mu \in S \\ d \equiv 1 \pmod{3} \\ (d,r) = 1 \\ \nu = \lambda^3 \mu d^3 }}
  \overline{\tau(\mu)} \psi(\lambda^4 \mu) \mu(d) |d|. 
  \end{equation}
  Recall the definition of $S$ in \eqref{Sdefn}.
  If $\nu \notin \lambda^3 S$, then \eqref{rightside} is zero.
  Therefore we can assume that $\nu \in \lambda^3 S$.
  If $\mu \in S$ and $\tau(\mu) \neq 0$, then by \eqref{cubiccoeff}
  we must have 
  $$
  \overline{\tau(\mu)} = 3^{5/2} \tilde{g}(e) |f|,
  $$
  where 
  \begin{equation*}
  \mu=\lambda^{-3} e f^3 \quad  \text{for some} \quad e, f \equiv 1 \pmod{3}  \quad \text{and} \quad \mu^2(e)=1.
  \end{equation*}
  Thus \eqref{rightside} is equal to
  \begin{align} \label{rightside2}
(-1)^{\ell} i^{-\ell} \Big ( \frac{\nu}{|\nu |} \Big )^{\ell} & \sum_{\substack{\substack{e,f,d \equiv 1 \pmod{3} \\ (d,r)=1 \\ 
  \nu=e (df)^3}}} \tilde{g}(e) \psi(\lambda ef^3) \mu(d) |d| |f| \nonumber \\
  &= (-1)^{\ell} i^{-\ell} \Big ( \frac{\nu}{|\nu|} \Big )^{\ell} \sum_{\substack{\substack{e,f,d \equiv 1 \pmod{3} \\ (df,r)=1 \\ 
  \nu=e (df)^3  }}} \tilde{g}(e) \psi(\lambda e) \mu(d) |d| |f|,
  \end{align}
  where the last display follows from the assumption that $\psi$ is a primitive cubic
  character to modulus $r$. Note that it is redundant to have $\mu^2(e)=1$
  in \eqref{rightside2} because this condition is automatically captured
  by \eqref{cuberel}. M\"{o}bius inversion then tells us that
  the right side of \eqref{rightside2} is equal to
  \begin{align*}
 (-1)^{\ell} i^{-\ell} \Big ( \frac{\nu}{|\nu|} \Big )^{\ell}  \sum_{\substack{e,u \equiv 1 \pmod{3} \\ (u,r)=1 \\ \nu=e u^3 }} &
   \tilde{g}(e) \psi(\lambda e) |u| \Big( \sum_{\substack{d \mid u \\ d \equiv 1 \pmod{3} }} \mu(d) \Big) \\
  & =(-1)^{\ell} i^{-\ell} \Big ( \frac{e}{|e|} \Big )^{\ell} \tilde{g}(e) \psi(\lambda e),
  \end{align*}
  as required.
\end{proof}

The following lemma records the standard evaluation of Ramanujan sums over $\mathbb{Z}[\omega]$.
The lemma follows from the evaluation
\begin{align*}
 c_{\varpi}(u)&=\begin{cases}
 \varphi(\varpi) & \quad \text{if} \quad u \equiv 0 \pmod{\varpi} \\
 -1 &  \quad \text{otherwise}
\end{cases},  \\
\quad & \text{for} \quad \varpi \equiv 1 \pmod{3} \quad \text{prime} \quad \text{and} \quad u \in \mathbb{Z}[\omega], 
\end{align*}
and the multiplicativity of Ramanujan sums i.e. 
\begin{align*}
c_{mn}(u)=c_{m}(u) c_{n}(u) 
\quad & \text{for} \quad u,m,n \in \mathbb{Z}[\omega]  
\quad \text{with} \quad 
m,n \equiv 1 \pmod{3} \\
\quad & \text{and} \quad (m,n)=1.
\end{align*}

\begin{lemma} \label{ramstand}
Let $m \in \mathbb{Z}[\omega]$ be squarefree and satisfy $m \equiv 1 \pmod{3}$. Then for $u \in \mathbb{Z}[\omega]$,
  $$
  c_{m}(u) := \sum_{\substack{x \pmod{m} \\ (x,m) = 1}} \check{e} \Big ( \frac{ux}{m} \Big ) =\mu \Big( \frac{m}{(m,u)}  \Big) 
  \frac{\varphi(m)}{\varphi \big( \frac{m}{(m,u)} \big)},
  $$
 where $\varphi(\cdot)$ is the Euler phi function on $\mathbb{Z}[\omega]$.
\end{lemma}

Let $r \in \mathbb{Z}[\omega]$ be squarefree and satisfy $r \equiv 1 \pmod{3}$.
We are now able to study the 
the analytic properties of the Dirichlet series
\begin{equation} \label{keyseries}
\sum_{c \equiv 1 \pmod{3}} \frac{\tilde{g}(c) \overline{\big ( \frac{c}{r} \big )_{3}} \big ( \frac{c}{|c|} \big )^{\ell}}{N(c)^s}, \quad \Re(s)>1.
\end{equation}
The following result records a level aspect (i.e. in the conductor of the cubic twist in \eqref{keyseries}) functional equation for 
\eqref{keyseries} that generalises
\cite[Theorem~6.1]{Pat1}. It explicates the root number and level, and is crucial
to our paper. 
Yoshimoto \cite{Yos} established level aspect analogues 
of \cite[Theorem~6.1]{Pat1} for twists of Gauss sums 
by arbitrary non-cubic Dirichlet characters. 
Clearly, Yoshimoto's results do not cover the case
we need.

\begin{remark}
The functional equation in \cite[Theorem~6.1]{Pat1} could be potentially used to obtain a functional equation for \eqref{keyseries} similar to the one presented below
in Proposition \ref{impfunc} (at least when $(r,c)=1$). 
Patterson exploited the fact that \eqref{keyseries} (and its variants) are the Fourier coefficients
of cubic metaplectic Eisenstein series attached to the essential cusps of $\Gamma_1(3)$ with respect to $\chi$.
The drawback of the functional equation in \cite[Theorem~6.1]{Pat1} is that the dual side of  is a linear combination of variants
of Dirichlet series of twisted cubic Gauss sums that involve the ramified prime, and this obscures the root number.
We found it advantageous to work directly with the automorphy of the cubic theta function (this is a more advanced starting point since
the purpose of \cite{Pat1} was to compute the Fourier coefficients of the cubic theta function).
\end{remark}

\begin{prop} \label{impfunc}
 Let $r \in \mathbb{Z}[\omega]$ be squarefree and satisfy $r \equiv 1 \pmod{3}$, and
 $\psi:= \overline{\big ( \frac{\cdot}{r} \big )_{3}}$. Let $\ell \in \mathbb{Z}$. 
  Then the Dirichlet series
   \begin{equation} \label{Rspsi}
  \mathcal{R} (s;\psi, \ell):=\sum_{c \equiv 1 \pmod{3}} \frac{\tilde{g}(c) \overline{\big ( \frac{c}{r} \big )_{3}} \big ( \frac{c}{|c|} \big )^{\ell}}{N(c)^s}, \quad \Re(s)>1,
   \end{equation}
   admits meromorphic continuation to all of $\mathbb{C}$. If $\ell = 0$, the Dirichlet series 
   \begin{equation} \label{eq:444}
   \zeta_{\mathbb{Q}(\omega)}(3s - \tfrac 12;\mathbf{1}_r,  3 \ell) \mathcal{R}(s; \psi, \ell)
   \end{equation}
   has a unique pole located at $s = \tfrac 56$, and it is simple. If $\ell \neq 0$ the Dirichlet series
   \eqref{eq:444} defines an entire function. We have  
 $$
 \text{Res}_{s=\frac{5}{6}} \big( \zeta_{\mathbb{Q}(\omega)}(3s - \tfrac 12;\mathbf{1}_r) \mathcal{R}(s; \psi)
  \big)=\frac{(2\pi)^{5/3} \overline{g(r)} \varphi(r)}{\Gamma ( \tfrac 23 ) 3^{7/2} N(r)^{5/3}}, 
$$
and for $\ell \in \mathbb{Z}$ we have the functional equation
\begin{align*}
\zeta_{\mathbb{Q}(\omega)}&(3s - \tfrac 12;\mathbf{1}_r, 3 \ell) \mathcal{R}(s;\psi, \ell) \\ & =\overline{g(r)} \Big( \frac{\overline{r}}{r} \Big)^{-\ell}
\cdot \frac{i^{\ell} (2\pi)^{4s - 2}}{3^{7/2} N(r)^{2s}} \frac{\Gamma(\tfrac 56 + \tfrac{|\ell|}{2} -s)\Gamma(\tfrac 76 + \tfrac{|\ell|}{2} - s)}{\Gamma(s + \tfrac {|\ell|}{2} - \tfrac 16) \Gamma( s + \tfrac{|\ell|}{2} + \tfrac 16)} \nonumber \\
& \times \zeta_{\mathbb{Q}(\omega)} (\tfrac{5}{2} - 3s;\mathbf{1}_r, -3 \ell) \mathcal{R}_r^{\dagger}(1 - s; -\ell),
\end{align*}
where
$$
\mathcal{R}_r^{\dagger}(s; \ell):= (-1)^{\ell} i^{-\ell} \sum_{\nu \in \lambda^{-1} \mathbb{Z}[\omega]} \frac{a^{\dagger}(\nu) b_r^{\dagger}(\nu) \big ( \frac{\nu}{|\nu|} \big )^{\ell}}{N(\nu)^s}, \quad \Re(s)>1,
\quad \ell \in \mathbb{Z},
$$
for some Dirichlet coefficients $a^{\dagger}(\nu)$. The coefficients $a^{\dagger}(\nu)$ have support contained 
in the set 
\begin{align*}
Q^{\dagger}:=\big \{\nu=\lambda^{L} \zeta & hw {h^{\prime }}^{3}: L \in \mathbb{Z}_{\geq -1}, \quad \zeta \in \{\pm 1, \pm \omega, \pm \omega^2 \}, \\
& h,h^{\prime},w \equiv 1 \pmod{3}, \quad h,h^{\prime} \mid r^{\infty}, \quad (w,r)=1 \quad \text{and} \quad \mu^2(hw)=1 \big \},
\end{align*}
and for $\nu \in Q^{\dagger}$,
\begin{equation} \label{adag}
a^{\dagger}(\nu)= 
     \begin{cases}
     \overline{\tau(-\lambda^{L-3} \nu)} & \text{if} \quad L \geq 0 \\
     \omega^2 \overline{\tau_1(-\lambda^{L-3} \omega^2 \nu)} \check{e}(\lambda^{L-3} \nu) & \text{if} \quad L=-1, \quad \zeta \in \{-1,-\omega,-\omega^2 \}  \\
     \omega \overline{\tau_2(-\lambda^{L-3} \omega \nu)} \check{e}(\lambda^{L-3} \nu) & \text{if} \quad L=-1, \quad \zeta \in \{1,\omega,\omega^2 \}
     \end{cases},
\end{equation}
and
\begin{equation} \label{bdag}
b_r^{\dagger}(\nu)=\mu \Big( \frac{r}{(\lambda \nu,r)} \Big) \frac{\varphi(r)}{\varphi \big( \frac{r}{(\lambda \nu,r)} \big) }.
\end{equation}
\end{prop}

\begin{remark}
Recall that $\tau(\cdot), \tau_1(\cdot)$ and $\tau_2(\cdot)$ are given in \eqref{cubiccoeff}, \eqref{cubiccoeff1} and
\eqref{cubiccoeff2} respectively.
\end{remark}
 
\begin{proof}
   Meromorphic continuation of
   $\zeta_{\mathbb{Q}(\omega)}(3s - \tfrac 12;\mathbf{1}_r, 3 \ell) \mathcal{R}(s; \psi, \ell)$
   to all of $\mathbb{C}$ follows from Lemma \ref{mellin},
   Proposition \ref{polefunc} and Lemma \ref{sievedir}. 
   If $\ell \neq 0$, then it is entire. If $\ell=0$ ,
   then it has a unique simple pole at $s=5/6$ with residue
   $$
  \text{Res}_{s=\frac{5}{6}} \big( \zeta_{\mathbb{Q}(\omega)}(3s - \tfrac 12;\mathbf{1}_r) \mathcal{R}(s; \psi)
  \big)=\left( \frac{\lambda}{r} \right)_3  \frac{(2\pi)^{5/3} \widetilde{\psi}(0)}{3^{7/2} \Gamma(\tfrac 23) N(r)^{5/3}}.
   $$

   We now evaluate $\widetilde{\psi}(u)$ (defined in \eqref{tildepsi}). Recall from \eqref{Psidefn} and the argument following it that we have
     \begin{equation} \label{psitildecalc}
     \widetilde{\psi}(u) = \left(\frac{\lambda}{r} \right)_3 \sum_{\substack{d \pmod{r} \\ (d,r)=1 \\ (\lambda^4 d)(\lambda^4 a) \equiv 1 \pmod{r}}} 
     \widehat{\psi}(d)  \Big ( \frac{a}{r} \Big )_3 \check{e} \Big (\frac{a u}{r} \Big ), \quad u \in \mathbb{Z}[\omega].
    \end{equation}
     Moreover, using the definition \eqref{fouriertrans} and the fact that $\psi$ is primitive gives us
     $$
     \widehat{\psi}(d) = \Big ( \frac{d}{r} \Big )_{3} \overline{g(r)}.
     $$
     We have $ad \equiv \overline{\lambda^8} \pmod{r}$ in \eqref{psitildecalc}.
     Therefore 
     \begin{equation} \label{tildeeval}
     \widetilde{\psi}(u)=\Big (\frac{\lambda^2}{r} \Big )_{3} \overline{g(r)} c_r(u),
     \end{equation}
     where $c_r(\cdot)$ denotes the usual Ramanujan sum.
     In particular,
     \begin{equation*}
     \widetilde{\psi}(0) = \left ( \frac{\lambda^2}{r} \right )_{3} \overline{g(r)} \varphi(r).
     \end{equation*}
     Lemma \ref{sievedir} tells us that 
     \begin{equation*}
     \zeta_{\mathbb{Q}(\omega)}(3s - \tfrac 12; \mathbf{1}_r, 3 \ell) \mathcal{R}(s; \psi, \ell)
     =(-1)^{\ell} i^{\ell}  \left( \frac{\lambda}{r} \right)_3 3^{-5/2-3s} \mathcal{D}(s, \overline{F_1}; \psi, \ell)_{S}.
     \end{equation*}
    Thus Lemma \ref{mellin} and Proposition \ref{polefunc} imply that
 \begin{align} \label{levelfunc2}
\zeta_{\mathbb{Q}(\omega)} & (3s - \tfrac 12; \mathbf{1}_r, 3 \ell) \mathcal{R}(s; \psi, \ell) \nonumber \\
& = \Big( \frac{\overline{r}}{r} \Big)^{-\ell} \frac{i^{\ell} (2\pi)^{4s - 2}}{3^{7/2} N(r)^{2s}} \frac{\Gamma(\tfrac 56 + \tfrac{|\ell|}{2} -s)\Gamma(\tfrac 76 + \tfrac{|\ell|}{2} - s)}{\Gamma(s + \tfrac{|\ell|}{2} - \tfrac 16) \Gamma( s + \tfrac{|\ell|}{2} + \tfrac 16)} \zeta_{\mathbb{Q}(\omega)} \big (\tfrac{5}{2} - 3s;\mathbf{1}_r, - 3 \ell \big )
\nonumber \\
& \times  \left( \frac{\lambda}{r} \right)_3
3^{-3s}  \frac{\mathcal{D}(1 - s, \overline{G^{\star}}; \widetilde{\psi}, - \ell)}{\zeta_{\mathbb{Q}(\omega)}(\frac{5}{2} - 3s;\mathbf{1}_r, -3 \ell)},
\end{align}
 where $\overline{G^{\star}}(w;\widetilde{\psi},\ell)$ is as in \eqref{Gstarcombo}. 
 Observe that \eqref{tildeeval} gives 
 \begin{equation*}
  \left( \frac{\lambda}{r} \right)_3 3^{-3s}  \frac{\mathcal{D}(s, \overline{G^{\star}}; \widetilde{\psi}, - \ell)}{\zeta_{\mathbb{Q}(\omega)}(3s - \tfrac 12;\mathbf{1}_r, - 3 \ell)}
  =\overline{g(r)} \mathcal{R}_r^{\dagger}(s, - \ell),
   \end{equation*}
   where   
    \begin{equation} \label{Rdefn}
     \mathcal{R}_r^{\dagger}(s;\ell):=3^{-3s} \frac{\mathcal{D}(s, \overline{G^{\star}}; c_r(\cdot), \ell)}{\zeta_{\mathbb{Q}(\omega)}(3s - \tfrac 12;\mathbf{1}_r, 3 \ell)},
     \quad \ell \in \mathbb{Z}.
     \end{equation}
   
     We now analyse the Dirichlet coefficients of $\mathcal{R}_r^{\dagger}(s;-\ell)$.
     Let $a_{-\ell}^{\dagger}(\nu)$ and $a_{-\ell}^{\star}(\nu)$ be the Fourier coefficients of
     $\mathcal{R}_r^{\dagger}(s;-\ell)$
     and $\mathcal{D}(s, \overline{G^{\star}};-\ell)$ respectively. 
      Using the definition of $G^{\star}(w;\widetilde{\psi},\ell)$ in \eqref{Gstarcombo} and Appendix \ref{appendix} for the closed form expressions for $d_1(\mu), d_{19}(\mu)$, and $d_{10}(\mu)$ gives
     \begin{align} \label{astardef}
     a_{-\ell}^{\star}(\mu)&=\Big( \frac{\mu}{|\mu|} \Big)^{-\ell} \big( \overline{d_1(-\mu)}+\omega \overline{d_{19}(-\mu)}+\omega^2 \overline{d_{10}(-\mu)} \big) \nonumber \\
     &= \Big( \frac{\mu}{|\mu|} \Big)^{-\ell} \Big(\overline{\tau(-\mu)}+\omega^2 \overline{\tau_1(-\omega^2 \mu)} \check{e}(\mu) +\omega \overline{\tau_2(-\omega \mu)} \check{e}(\mu) \Big).
     \end{align}
     Consultation with \eqref{cubiccoeff}, \eqref{cubiccoeff1} and \eqref{cubiccoeff2}
     shows that the $a^{\star}(\mu)$ have support contained in the set  
     \begin{equation*}
     U:=\big \{ \mu=\lambda^{k} \zeta c j^3: k \in \mathbb{Z}_{\geq -4}, \quad \zeta \in \{\pm 1, \pm \omega, \pm \omega^2 \},
     \quad c,j \equiv 1 \pmod{3} \quad \text{and} \quad \mu^2(c)=1 \big \}.
     \end{equation*}
      Each of the three terms in \eqref{astardef} have disjoint support.
     In particular,
     \begin{align} \label{astarcases}
     a_{-\ell}^{\star}(\mu)&= \Big( \frac{\mu}{|\mu|} \Big)^{-\ell}
     \cdot 
     \begin{cases}
     \overline{\tau(-\mu)} & \text{if} \quad k \geq -3 \\
     \omega^2 \overline{\tau_1(-\omega^2 \mu)} \check{e}(\mu) & \text{if} \quad k=-4, \quad \zeta \in \{-1,-\omega,-\omega^2 \}  \\
     \omega \overline{\tau_2(-\omega \mu)} \check{e}(\mu) & \text{if} \quad k=-4, \quad \zeta \in \{1,\omega,\omega^2 \}
     \end{cases} \nonumber \\
     &=: \Big(\frac{\mu}{|\mu|} \Big)^{-\ell} a^{\star}(\mu).
     \end{align}
     Observe that \eqref{Rdefn} and \eqref{astardef} 
     imply that the coefficients $a^{\dagger}(\cdot)$ have support contained in $\lambda^{3} U \subset \lambda^{-1} \mathbb{Z}[\omega] \setminus \{0\}$.
    Then
    \begin{align*}
     a_{-\ell}^{\dagger}(\nu) &=\Big( \frac{\nu/\lambda^3}{|\nu/\lambda^3|} \Big)^{-\ell}
     \sum_{\substack{\nu = \lambda^3 \mu d^3 \\ \mu \in U \\ (d,r)=1 \\ d \equiv 1 \pmod{3}  }} a^{\star}(\mu) c_r(\lambda^{4} \mu) \mu(d) |d| \\
     &=(-1)^{\ell} i^{\ell} \Big( \frac{\nu}{|\nu|} \Big)^{-\ell}
     \sum_{\substack{\nu = \lambda^3 \mu d^3 \\ \mu \in U \\ (d,r)=1 \\ d \equiv 1 \pmod{3}  }} a^{\star}(\mu) c_r(\lambda^{4} \mu) \mu(d) |d|.
     \end{align*}
     Evaluation of the Ramanujan sum using 
     Lemma \ref{ramstand} gives
     \begin{equation} \label{adagdiv}
    a_{-\ell}^{\dagger}(\nu)=(-1)^{\ell} i^{\ell} \Big( \frac{\nu}{|\nu|} \Big)^{-\ell} \sum_{\substack{\nu = \lambda^3 \mu d^3 \\ \mu \in U \\ (d,r)=1 \\ d \equiv 1 \pmod{3}  }}  
    a^{\star}(\mu) \mu \Big( \frac{r}{(\lambda^4 \mu,r)} \Big) 
    \frac{\varphi(r)}{\varphi \big( \frac{r}{(\lambda^4 \mu,r)} \big) } \mu(d) |d|.
     \end{equation}
     To continue the evaluation of $a^{\dagger}(\nu)$ in \eqref{adagdiv}, 
     we write each $\mu \in U$ occurring on the right side 
     uniquely as
     \begin{align*}
     \lambda^4 \mu& =\lambda^{k+4} \zeta h w (h^{\prime} w^{\prime})^3  
     \quad \text{with} \quad \zeta \in \{\pm 1,\pm \omega, \pm \omega^2 \}, \\ 
     & \quad h,h^{\prime} \mid r^{\infty}, \quad (ww^{\prime},r)=1,  \quad h,h^{\prime},w,w^{\prime} \equiv 1 \pmod{3}
     \quad \text{and} \quad \mu^2(h w)=1.
     \end{align*}
     Then 
     \begin{equation} \label{astardevelop}
     a_{-\ell}^{\dagger}(\nu)=(-1)^{\ell} i^{\ell} \Big( \frac{\nu}{|\nu|} \Big)^{-\ell} \sum_{\substack{\nu = \lambda^{k+3} \zeta h w (h^{\prime} w^{\prime} d)^3 \\
     h,h^{\prime},w,w^{\prime},d \equiv 1 \pmod{3} \\ \mu^2(hw)=1 \\ h,h^{\prime} \mid r^{\infty} \\ (dww^{\prime},r)=1 }}  
    a^{\star} \big(\lambda^k \zeta h w (h^{\prime} w^{\prime})^3 \big) \mu \Big( \frac{r}{(h{h^{\prime}}^3,r)} \Big) 
    \frac{\varphi(r)}{\varphi \big( \frac{r}{(h{h^{\prime}}^3,r)} \big) } \mu(d) |d|.
     \end{equation}
    Furthermore, \eqref{astarcases} tells us that
     \begin{align*}
     a^{\star} \big(&\lambda^k \zeta h w (h^{\prime} w^{\prime})^3 \big) \\
     &= 
     \begin{cases}
     \overline{\tau \big({-\lambda^k \zeta h w (h^{\prime} w^{\prime})^3 }\big)} & \text{if} \quad k \geq -3 \\
     \omega^2 \overline{\tau_1 \big({-\lambda^k \omega^2 \zeta h w (h^{\prime} w^{\prime}})^3 \big)} \check{e} \big(\lambda^k \zeta h w (h^{\prime} w^{\prime})^3 \big) 
     & \text{if} \quad k=-4,        \quad \zeta \in \{-1,-\omega,-\omega^2 \}  \\
     \omega \overline{\tau_2 \big({-\lambda^k \omega \zeta h w (h^{\prime} w^{\prime})^3 \big)}} \check{e} \big(\lambda^k \zeta h w (h^{\prime} w^{\prime})^3 \big) & \text{if} \quad k=-4, \quad \zeta \in \{1,\omega,\omega^2 \}
     \end{cases}.
     \end{align*}
    Further consultation with \eqref{cubiccoeff}, \eqref{cubiccoeff1} and \eqref{cubiccoeff2} shows that
    \begin{equation} \label{bstar}
    a^{\star} \big(\lambda^k \zeta h w (h^{\prime} w^{\prime})^3 \big)=b_{k,\zeta}^{\star}(hw)  \frac{|h^{\prime} w^{\prime} |}{|hw|},
    \end{equation}
     for some sequence of coefficients $b^{\star}_{k,\zeta}(\cdot)$ on squarefree elements of
     $\mathbb{Z}[\omega]$ that are congruent to $1$ (the sequence depends only on $k$ and $\zeta$).
     Using \eqref{bstar} in \eqref{astardevelop}, we obtain 
     \begin{align} \label{adagmob}
     a_{-\ell}^{\dagger}(\nu)= (-1)^{\ell} i^{\ell} \Big( \frac{\nu}{|\nu|} \Big)^{-\ell} & \sum_{\substack{\nu = \lambda^{k+3} \zeta h w (h^{\prime} u)^3 \\
     h,h^{\prime},w,u \equiv 1 \pmod{3} \\ \mu^2(hw)=1 \\ h,h^{\prime} \mid r^{\infty} \\ (uw,r)=1 }} \frac{b^{\star}_{k,\zeta}(hw)}{|hw|} 
     |h^{\prime} u|
     \mu \Big( \frac{r}{(h{h^{\prime}}^3,r)} \Big) 
    \frac{\varphi(r)}{\varphi \big( \frac{r}{(h{h^{\prime}}^3,r)} \big) } \\ 
    & \times \Big( \sum_{\substack{ d \mid u \\ d \equiv 1 \pmod{3}}} \mu(d) \Big). \nonumber
   \end{align}
    M\"{o}bius inversion tells us that $u=1$ in \eqref{adagmob}. 
    Subsequent use of \eqref{bstar} (in reverse) 
    gives
    \begin{equation} \label{adagfinal}
   a_{-\ell}^{\dagger}(\nu)=(-1)^{\ell} i^{\ell} \Big( \frac{\nu}{|\nu|} \Big)^{-\ell} 
    \sum_{\substack{\nu = \lambda^{k+3} \zeta h w {h^{\prime}}^3 \\
     h,h^{\prime},w \equiv 1 \pmod{3} \\ \mu^2(hw)=1 \\ h,h^{\prime} \mid r^{\infty} \\ (w,r)=1 }}  
     a^{\star}(\lambda^k \zeta hw {h^{\prime}}^3)
     \mu \Big( \frac{r}{(h{h^{\prime}}^3,r)} \Big) 
    \frac{\varphi(r)}{\varphi \big( \frac{r}{(h{h^{\prime}}^3,r)} \big) }.
    \end{equation}
   For a given $\nu$, there is at most one summand on the right side on \eqref{adagfinal}.
   This completes the proof.
\end{proof}

\subsection{Voronoi formula}
 We are finally able to prove a variant of the Voronoi summation formula.
 \begin{prop} \label{voronoiprop}
   Let $W$ be a smooth Schwartz function, compactly supported in $(0, \infty)$. Let $\ell \in \mathbb{Z}$.
   Then for $X>0$ we have
   \begin{align} \label{voronoiform}
   \sum_{\substack{c,d \equiv 1 \pmod{3} \\ (d,r)=1}} & |d| \tilde{g}(c) \Big ( \frac{c d^3}{|c d^3|} \Big )^{\ell} \cdot \overline{\Big ( \frac{c}{r} \Big )_{3}} W \Big ( \frac{N(c d^3)}{X} \Big ) \nonumber \\
    & = \delta_{\ell = 0} \cdot X^{5/6} \widetilde{W} \Big ( \frac{5}{6} \Big ) \frac{(2\pi)^{5/3} \varphi(r) \overline{g(r)}}{3^{7/2} \Gamma ( \tfrac 23 ) N(r)^{5/3}} \nonumber \\ 
    & +   \frac{\overline{g(r)}}{3^{7/2} (2\pi)^2} \Big( \frac{\overline{r}}{r} \Big)^{-\ell} \sum_{\substack{\nu \in \lambda^{-1} \mathbb{Z}[\omega] \\ d \equiv 1 \pmod{3} \\ (d,r)=1 }} \frac{a^{\dagger}(\nu) b_r^{\dagger}(\nu)}{N(\nu) N(d)^{5/2}} \Big ( \frac{d^3 \nu}{|d^3 \nu|} \Big )^{- \ell} \cdot \widecheck{W_{\ell}} \Big ( \frac{(2\pi)^4 N(d^3 \nu) X}{N(r)^2} \Big ), 
   \end{align}
   where the $a^{\dagger}(\cdot)$ and $b_r^{\dagger}(\cdot)$ are as in \eqref{adag} and \eqref{bdag} respectively, and 
   $\widecheck{W_{\ell}}: \mathbb{R}_{>0} \rightarrow \mathbb{C}$ is defined by
   \begin{equation} \label{Wcheckdef}
   \widecheck{W_{\ell}}(u) := \frac{1}{2\pi i} \int_{-\varepsilon- i \infty}^{-\varepsilon+ i \infty} u^s \frac{\Gamma(\tfrac 56 + \tfrac{|\ell|}{2} -s)\Gamma(\tfrac 76 + \tfrac{|\ell|}{2} -s)}{\Gamma(s + \tfrac{|\ell|}{2}- \tfrac 16)\Gamma(s +\tfrac{|\ell|}{2} + \tfrac 16)} \widetilde{W}(s) ds,
  \end{equation}
  for $\varepsilon \in (0,\frac{1}{10000})$.
   For any $A > 0$ we have
   \begin{equation} \label{Wcheckdecay}
   \widecheck{W_{\ell}}(u) \ll_{W,A} (1 + |\ell|)^2 \cdot \begin{cases} 
   (u / (1+\ell^4))^{5/6} & \text{ if } |u| \leq (1+\ell^4) \\
    (u / (1+\ell^4))^{-A} & \text{ if } |u| > (1+\ell^4).
    \end{cases}
   \end{equation}
   \end{prop}
  \begin{proof}
   We have
  \begin{align}
      \sum_{\substack{d,c \equiv 1 \pmod{3} \\ (d,r)=1 }} |d| \tilde{g}(c) & \overline{\Big ( \frac{c}{r} \Big )_{3}} \Big ( \frac{c d^3}{|c d^3|} \Big )^{\ell} W \Big ( \frac{N(c d^3)}{X} \Big ) \nonumber \\
      & = \frac{1}{2\pi i} \int_{2 - i \infty}^{2 + i \infty} \zeta_{\mathbb{Q}(\omega)}(3s - \tfrac 12;\mathbf{1}_r, 3 \ell) \mathcal{R}(s; \psi, \ell) X^s \widetilde{W}(s) ds, \label{openmov}
   \end{align}
      where $R(s;\psi,\ell)$ is given in \eqref{Rspsi}.
      We shift the contour to $\Re s = - \varepsilon$. Proposition \ref{impfunc} tells us that we collect a pole at $s = \tfrac 56$ when $\ell=0$. 
      Thus \eqref{openmov} is equal to
      \begin{equation} \label{resline}
     \delta_{\ell = 0} \cdot X^{5/6} \widetilde{W}\Big ( \frac{5}{6} \Big ) \frac{(2\pi)^{5/3} \overline{g(r)} \varphi(r)}{3^{7/2} \Gamma(\tfrac 23) N(r)^{5/3}} + \frac{1}{2\pi i} \int_{- \varepsilon - i\infty}^{- \varepsilon + i \infty} \zeta_{\mathbb{Q}(\omega)}(3s - \tfrac 12;\mathbf{1}_r, 3 \ell) \mathcal{R}(s; \psi, \ell) X^s \widetilde{W}(s) ds. 
      \end{equation}
      We evaluate the integral in \eqref{resline} by applying the functional equation in Proposition \ref{impfunc}.
      We obtain
      \begin{align} \label{rootintegral}
& \frac{1}{2\pi i} \int_{- \varepsilon - i\infty}^{- \varepsilon + i \infty} \zeta_{\mathbb{Q}(\omega)}(3s - \tfrac 12;\mathbf{1}_r,3 \ell) \mathcal{R}(s; \psi,\ell) X^s \widetilde{W}(s) ds \nonumber \\ & = \frac{1}{2\pi i} \cdot \frac{i^{\ell} \overline{g(r)}}{3^{7/2}} \Big( \frac{\overline{r}}{r} \Big)^{-\ell} \int_{- \varepsilon - i \infty}^{- \varepsilon + i \infty} \frac{(2\pi)^{4s - 2}}{N(r)^{2s}} \frac{\Gamma(\tfrac 56+\tfrac{|\ell|}{2} -s)\Gamma(\tfrac 76+\tfrac{|\ell|}{2} - s)}{\Gamma(s + \tfrac{|\ell|}{2} - \tfrac 16) \Gamma( s + \tfrac{|\ell|}{2} + \tfrac 16)} \nonumber \\ 
  &  \qquad \qquad \qquad \times \zeta_{\mathbb{Q}(\omega)} (\tfrac{5}{2} - 3s ;\mathbf{1}_r , - 3\ell) \mathcal{R}^{\dagger}_r(1 - s, - \ell) X^{s} \widetilde{W}(s) ds.  
      \end{align}
    Expanding the absolutely convergent series we see that \eqref{rootintegral} is equal to
      \begin{align*}
      \frac{\overline{g(r)}}{3^{7/2} (2\pi)^2} \Big( \frac{\overline{r}}{r} \Big)^{-\ell} & \sum_{\substack{\nu \in \lambda^{-1} \mathbb{Z}[\omega] \\ d \equiv 1 \pmod{3} \\ (d,r)=1}}
      \frac{a^{\dagger}(\nu) b_{r}^{\dagger}(\nu) \big ( \frac{d^3 \nu}{|d^3 \nu|} \big )^{-\ell}}{N(\nu) N(d)^{5/2}} \\ & \times \frac{1}{2\pi i} \int_{- \varepsilon - i \infty}^{- \varepsilon + i \infty} \Big ( \frac{(2\pi)^4 N(d^3 \nu) X}{N(r)^2} \Big )^{s}   \frac{\Gamma(\tfrac 56+ \tfrac{|\ell|}{2} -s)\Gamma(\tfrac 76 +\tfrac{|\ell|}{2} - s)}{\Gamma(s + \tfrac{|\ell|}{2} - \tfrac 16) \Gamma( s +\tfrac{|\ell|}{2} + \tfrac 16)} \cdot \widetilde{W}(s) ds.
      \end{align*}
     The above display can be expressed as 
      $$
      \frac{\overline{g(r)}}{3^{7/2} (2 \pi)^2}  \Big( \frac{\overline{r}}{r} \Big)^{-\ell} \sum_{\substack{\nu \in \lambda^{-1} \mathbb{Z}[\omega] \\ d \equiv 1 \pmod{3} \\ (d,r)=1}} \frac{a^{\dagger}(\nu) b_{r}^{\dagger}(\nu) \big ( \frac{d^3 \nu}{|d^3 \nu|} \big )^{-\ell}}{N(\nu) N(d)^{5/2}} \cdot \widecheck{W}_{\ell} \Big ( \frac{(2\pi)^4 N(d^3 \nu) X}{N(r)^2} \Big ). 
      $$
      This establishes \eqref{voronoiform}. The decay bound \eqref{Wcheckdecay} follows from a direct computation with Stirling's asymptotic 
      \cite[(5.11.1)]{NIST:DLMF}.
      
 \end{proof}

\section{Cancellations in sequences over primes} \label{primesec}

We begin with a remark about ordering integral ideals in number fields.
\begin{remark} \label{idealorder}
Let $K$ be an algebraic number field of degree $n$, and let $\mathfrak{a}$ and $\mathfrak{b}$ denote integral ideals of $K$. 
The order on integral ideals of $K$, $<_{K}$, will essentially be given by the norm $N_{K/\mathbb{Q}}$. That is, $N_{K/\mathbb{Q}}(\mathfrak{a}) < N_{K/\mathbb{Q}}(\mathfrak{b})$ implies $\mathfrak{a} <_K \mathfrak{b}$. For ideals with equal norm the ordering can be arbitrary. For instance, if
$N_{K/\mathbb{Q}}(\mathfrak{a})=N_{K/\mathbb{Q}}(\mathfrak{b})$ and $\tilde{a}$ and $\tilde{b}$ are some fixed choice of
ideal numbers of 
$\mathfrak{a}$ and $\mathfrak{b}$ respectively, then one could declare that
$\mathfrak{a} <_K  \mathfrak{b}$ if and only if  $\arg \tilde{a} < \arg \tilde{b}$.
We will abuse notation and denote $<_{\mathbb{Q}(\omega)}$ by $<$ when the meaning is clear from context 
\end{remark}

\begin{lemma}\label{le:primes}
Assume the Generalized Riemann Hypothesis for the Dedekind zeta function attached to $\mathbb{Q}(\omega)$ twisted by
  Gr{\"o}{\ss}encharaktern. 
  Let $W$ be a smooth test function with compact support in $(0, \infty)$ and $m_W:=\max \{1,\max_{x \in (0,\infty)}|W(x)|\}$.
  Let $B \geq 10$, $10 \leq w \leq M \leq N \ll B$,
  and $\pi \in \mathbb{Z}[\omega]$ satisfy $\pi \equiv 1 \pmod{3}$
  be a prime or $1$.  
  If $R \leq \frac{\log B}{K \log\log B}$ with $K > 1000$, then the sequence
  $$
  \beta_b = \frac{1}{m_W} W \Big ( \frac{N(b)}{B} \Big ) \sum_{\substack{b = \prod_{j = 1}^{R} \varpi_j  \\ \varpi_j \equiv 1 \pmod{3} \\ \varpi_1< \varpi_2 <\ldots<\varpi_R     \\ M \leq N(\varpi_j) \leq N \\ (b,\pi)=1 }} 1  $$
   belongs to $\mathcal{C}_{\eta}(B, w)$ (given in Definition \ref{seqdef}) for all $\eta>\frac{100}{K}$.  
\end{lemma}

\begin{remark}
Note that scaling $m_W$ is a minor technicality introduced so that the given $\boldsymbol{\beta}$ satisfies 
the first axiom of Definition \ref{seqdef}.
\end{remark}

\begin{proof}
It is clear the first three properties in Definition \ref{seqdef} follow from definition
    of $\boldsymbol{\beta}$. 
    After Mellin inversion of $W$, 
    $$
\beta_b:=\frac{1}{2 \pi i} \frac{1}{m_W}
\int_{2-i \infty}^{2+i \infty} \widetilde{W}(v) B^{v} \sum_{\substack{b = \prod_{j = 1}^{R} \varpi_j \\ M \leq \varpi_1< \varpi_2 <\ldots<\varpi_R \leq N  \\ \varpi_j \equiv 1 \pmod{3} \\ \varpi_j \neq \pi }} \prod_{j=1}^{R} \frac{1}{N(\varpi_j)^v}   dv,
    $$
    it suffices to check the fourth property of Definition \ref{seqdef} for all $\eta \gg 1/K$.
    For $t \in \mathbb{R}$, $\ell \in \mathbb{Z}$ and $k, u \in \mathbb{Z}[\omega]$ with $k,u \equiv 1 \pmod{3}$,
   it suffices to estimate
    \begin{equation} \label{eq:poi}
    \sum_{\substack{b \in \mathbb{Z}[\omega] \\ u | b}} \beta_b \Big ( \frac{b}{|b|} \Big )^{\ell} \Big ( \frac{k}{b} \Big )_{3} N(b)^{it},
    \end{equation}
    provided that $\ell \neq 0$, or if $\ell = 0$, then provided that $k \neq \cube$.
    Without loss of generality we can take $u = 1$ since the case $u \neq 1$ reduces to this case after combinatorial manipulations. 
    Thus \eqref{eq:poi} (with $u=1$) is equal to 
      \begin{equation} \label{poisetup}
 \frac{1}{2 \pi i} \frac{1}{m_W} \int_{2-i \infty}^{2+i \infty} \widetilde{W}(v) B^{v} \sum_{\substack{M \leq \varpi_1< \varpi_2 <\ldots<\varpi_R \leq N  \\ \varpi_j \equiv 1 \pmod{3}  \\  \varpi_j \neq \pi }} \prod_{j=1}^{R} \frac{\big(\frac{k}{\varpi_j} \big)_3 \big( \frac{\varpi_j}{|\varpi_j|} \big)^{\ell}}{N(\varpi_j)^{v-it}}   dv.
   \end{equation}
  The Newton-Girard identity \cite[$(2.14^{\prime})$]{Mac} implies that 
    \begin{align} \label{newgir}
    \sum_{\substack{M \leq \varpi_1< \varpi_2 <\ldots<\varpi_R \leq N  \\ \varpi_j \equiv 1 \pmod{3} \\ \varpi_j \neq \pi }}&  \prod_{j=1}^{R} \frac{\big(\frac{k}{\varpi_j} \big)_3 \big( \frac{\varpi_j}{|\varpi_j|} \big)^{\ell}}{N(\varpi_j)^{v-it} } \nonumber\\
    &=(-1)^R
    \sum_{\substack{m_1,\ldots,m_R \geq 0 \\ m_1+2 m_2+\cdots+R m_R=R  }}  \prod_{j=1}^R \frac{(-1)^{m_j}}{m_j ! j^{m_j}} \Big(\sum_{\substack{ M \leq N(\varpi) \leq N \\ \varpi \equiv 1 \pmod{3} \\ \varpi \neq \pi }}  \frac{\big(\frac{k}{\varpi} \big)^{j}_3 \big( \frac{\varpi}{|\varpi|} \big)^{j \ell}}{N(\varpi)^{j(v-it)}}   \Big)^{m_j}.
    \end{align}
We can assume without loss of generality that $M$ and $N$ are half-integers. Using \cite[pg.~105]{Dav} each sharp cut-off can be written as 
 \begin{align} \label{sharpy}
 \sum_{\substack{ M \leq N(\varpi) \leq N \\ \varpi \equiv 1 \pmod{3} \\ \varpi \neq \pi }}  \frac{\big(\frac{k}{\varpi} \big)^{j}_3 \big( \frac{\varpi}{|\varpi|} \big)^{j \ell}}{N(\varpi)^{j(v-it)}} 
 &=\frac{1}{(2 \pi i)^2} \iint_{1/\log B-iB^{100}}^{1/\log B+iB^{100}} D \big (j(v-it)+s-w; \Big ( \frac{k}{\cdot} \Big )^{j}_{3} \boldsymbol{1}_{\pi}, j \ell \big ) N^{s} M^{-w} \frac{ds dw}{s w} 
 \nonumber \\
 &+O (B^{-50}) \quad \text{say}, \quad \Re(v)=2,
 \end{align}
 where (for $\Re(s)>1/2$) we have
  \begin{equation*}
  D \Big(s;\Big(\frac{k}{\cdot} \Big )^{j}_{3} \boldsymbol{1}_{\pi}, j \ell \Big):=
 \log \zeta_{\mathbb{Q}(\omega)}(s; \Big(\frac{k}{\cdot} \Big)_3^{j} \boldsymbol{1}_{\pi}, j\ell)-\sum_{m \geq 2} \sum_{\substack{ \varpi \equiv 1 \pmod{3} \\ \varpi \neq \pi }} \frac{\big(\frac{k}{\varpi} \big)^{mj}_3 \big( \frac{\varpi}{|\varpi|}     \big)^{m j \ell}}{m N(\varpi)^{ms}}.
 \end{equation*}   

 We shift the $v$-contour in \eqref{poisetup} to $\Re(v)=1/2+1/\log B$. 
 From \eqref{sharpy}, the Riemann hypothesis, and \cite[Theorem~5.19]{IwK}, 
 we deduce that (uniformly in $j \geq 1$)
 \begin{equation*}
\sum_{\substack{ M \leq N(\varpi) \leq N \\ \varpi \equiv 1 \pmod{3} \\ \varpi \neq \pi }}  \frac{\big(\frac{k}{\varpi} \big)^{j}_3 \big( \frac{\varpi}{|\varpi|} \big)^{j \ell}}{N(\varpi)^{j(v-it)}}
\ll  (\log^2 B) \cdot  \log^2 \Big (2 + (1+|v|+|t|)(1+| \ell|) N(k) B \Big), \Re v \geq \frac 12 + \frac{1}{\log B}.
 \end{equation*}
 Substitution of this bound into \eqref{newgir} shows that for $\Re(v) \geq 1/2+1/\log B$ we have
 \begin{align*}
&  \sum_{\substack{M \leq \varpi_1< \varpi_2 <\ldots<\varpi_R \leq N  \\ \varpi_j \equiv 1 \pmod{3} \\ \varpi_j \neq \pi }}   \prod_{j=1}^{R} \frac{\big(\frac{k}{\varpi_j} \big)_3 \big( \frac{\varpi_j}{|\varpi_j|} \big)^{\ell}}{N(\varpi_j)^{v-it}} \\
 & \ll C^{R} (\log^{2R} B) \cdot \log^{2R} \Big (2 + (1+|v|+|t|)(1+| \ell|) N(k) B \Big) \\
 & \ll (C/\varepsilon^4)^{R} (\log^{2R} B^{\varepsilon}) \cdot \log^{2R} \Big (2 + \big((1+|v|+|t|)(1+| \ell|) N(k) B \big)^{\varepsilon} \Big),
 \end{align*}
for some absolute constant $C>1$ and any fixed $\varepsilon>0$.
Returning to \eqref{poisetup} (and recalling that we shifted the contour to $\Re(v)=1/2+1/\log B$) 
we see that \eqref{eq:poi} is
  \begin{equation} \label{eq:g}
 \ll (C/\varepsilon^4)^{R} B^{1/2+1/\log B} (\log^{2R} B^{\varepsilon} ) \cdot \log^{2R} \big( 2 + \big((1 + |t|)(1+|\ell|) N(k) B \big)^{\varepsilon} \big).
  \end{equation}
 We use the hypothesis $1 \leq R \leq \log B/(K \log\log B)$, the inequality
  $$
  (\log x)^{L} \leq L! x, \quad x \geq 1, \quad L \geq 1,
  $$
 and Stirling's asymptotic formula \cite[(5.11.1)]{NIST:DLMF}
 to conclude that \eqref{eq:g} is
  $$
  \ll_{\varepsilon} B^{1/2 + 100/K + 3 \varepsilon} (1 + |t|)^{\varepsilon} (1+|\ell|)^{\varepsilon} N(k)^{\varepsilon},
  $$
  say. This concludes the proof.
\end{proof}

A minor variation of the above proof gives a smoothed version of the Lemma.
\begin{lemma}\label{le:primes2}
Assume the Generalized Riemann Hypothesis for the Dedekind zeta function attached to $\mathbb{Q}(\omega)$ twisted by
  Gr{\"o}{\ss}encharaktern. 
Let $V$ and $W$ be a smooth test functions with compact support in $(0, \infty)$, and let $m_V$ and $m_W$ be as in 
Lemma \ref{le:primes}. Let $R \in \mathbb{N}$, $B \geq 10$, $10 \leq w \leq M \leq N \ll B$,
and $P_1,\ldots,P_R>1$ be such that $P_1 \cdots P_R \asymp B$.
 If $R \leq \frac{\log B}{K \log\log B}$ with $K > 1000$, then the sequence
  $$
  \beta_b = \frac{1}{m_W} W \Big ( \frac{N(b)}{B} \Big ) \sum_{\substack{b = \prod_{j = 1}^{R} \varpi_j \\ \varpi_j \equiv 1 \pmod{3} \\ \varpi_1< \varpi_2 <\ldots<\varpi_R   \\ M \leq N(\varpi_j) \leq N}} \prod_{j=1}^{R} \Big( \frac{1}{m_V} V \Big( \frac{N(\varpi_j)}{P_j} \Big) \Big),  $$
   belongs to $\mathcal{C}_{\eta}(B, w)$ (given in Definition \ref{seqdef}) for all $\eta>\frac{100}{K}$.  
\end{lemma}

\begin{remark}
A sum over $R$ running through any subset of $[1, \log B/(K \log \log B)]$
can be introduced in the definition of $\boldsymbol{\beta}$ occurring
in both Lemma \ref{le:primes} and Lemma \ref{le:primes2} with no change to the conclusions.
\end{remark}

\begin{lemma}  \label{convolutionlem}
  Suppose $A,B \geq 10$, $X:=AB$, $0<\eta_1, \eta_2 \leq 1/4$ and $w_1>w_2 \geq 10$.
  Let $\boldsymbol{\alpha}=(\alpha_a) \in \mathcal{C}_{\eta_1}(A,w_1)$, and
  $\boldsymbol{\beta}=(\beta_b) \in \mathcal{C}_{\eta_2}(B,w_2)$ be
  such that 
  \begin{equation} \label{sharpcond}
  \beta_b \neq 0 \implies ( \varpi \mid b \implies w_2<N(\varpi) \leq w_1).
  \end{equation}
  Let $10 \leq M \leq N$
  and $\boldsymbol{\gamma}:=(\gamma_c)$ be given by  
  \begin{equation*}
  \gamma_c:=\sum_{\substack{ M \leq N(c) \leq N \\ a,b \in \mathbb{Z}[\omega] \\ a,b \equiv 1 \pmod{3} \\ ab=c }} \alpha_a \beta_b.
  \end{equation*}
  Then $\boldsymbol{\gamma}=(\gamma_c) \in \mathcal{C}_{\max \{\eta_1,\eta_2\}}(X,w_2)$.
  \end{lemma}
  
  \begin{remark}
  It will be helpful to recall the definition of $\mathscr{P}(z)$ in \eqref{Pzdef}.
  \end{remark}
  
  \begin{proof}
  Observe that the hypotheses \eqref{sharpcond} and $w_2<w_1$ imply that
  $\boldsymbol{\gamma}=(\gamma_c)$ is supported on squarefree $w_2$-rough integers. We also 
  have $\gamma_c \neq 0 \implies N(c) \asymp X$ from the supports of $\boldsymbol{\alpha}=(\alpha_a)$ and 
  $\boldsymbol{\beta}=(\beta_b)$. Each $c \in \mathbb{Z}[\omega]$ with $c \equiv 1 \pmod{3} $ has a unique factorisation $c=ab$
  with $(a,\mathscr{P}(w_1))=1$ and $b \mid \mathscr{P}(w_1)$. Thus hypothesis \eqref{sharpcond} implies that 
  $|\gamma_c|=|\alpha_a \beta_b| \leq 1$. 
  
  It only remains to prove inequality \eqref{eq:can} in 
  Definition \ref{seqdef} for $\boldsymbol{\gamma}=(\gamma_c)$.
  Without loss of generality we can assume that $M$ and $N$ are half-integers and that $M \asymp X$ (resp. $N \asymp  X$), otherwise $\delta_{M \leq N(c)}$ (resp. $\delta_{N(c) \leq N}$) is a redundant condition.  Using \cite[pg.~105]{Dav}, the sharp cut-off can be written as
   \begin{equation*}
   \delta_{M \leq N(c) \leq N}=\frac{1}{(2 \pi i)^2} \iint_{1/\log X-iX^{100}}^{1/\log X+iX^{100}} \frac{N^{s} M^{-w}}{N(c)^{s-w}} \frac{ds dw}{sw}+ O(X^{-50} ),
   \end{equation*}
   say. The integrals incur an acceptable loss of $O((\log X)^2)$. Thus it suffices to show that 
   \begin{equation*}
  \tilde{\gamma}_c:=\sum_{\substack{a,b \in \mathbb{Z}[\omega] \\ a,b \equiv 1 \pmod{3} \\ ab=c}} \alpha_a \beta_b 
  \end{equation*}
  satisfies \eqref{eq:can}.
   In other words, for $t \in \mathbb{R}$, $\ell \in \mathbb{Z}$ and $k, u \in \mathbb{Z}[\omega]$ with $k,u \equiv 1 \pmod{3}$,
   we need to estimate
    \begin{equation} \label{impC}
    \sum_{\substack{c \in \mathbb{Z}[\omega] \\ u | c}} \tilde{\gamma}_c \Big ( \frac{c}{|c|} \Big )^{\ell} \Big ( \frac{k}{c} \Big )_{3} N(c)^{it},
    \end{equation}
    provided that $\ell \neq 0$, or if $\ell = 0$, then provided that $k \neq \cube$.
    Observe $u$ has a unique factorisation $u=u_1 u_2$ such that 
    $(u_1,\mathscr{P}(w_1))=1$ and $u_2 \mid \mathscr{P}(w_1)$.
    Hypothesis \eqref{sharpcond} implies that that \eqref{impC} is equal to 
    \begin{equation} \label{impC2}
   \Big(\sum_{\substack{a \\ u_1 \mid a  } } \alpha_a \Big( \frac{a}{|a|} \Big)^{\ell}  \Big( \frac{k}{a} \Big)_3 N(a)^{it} \Big)
   \Big(\sum_{\substack{b \\ u_2 \mid b  } } \beta_b \Big( \frac{b}{|b|} \Big)^{\ell} \Big( \frac{k}{b} \Big)_3 N(b)^{it} \Big).
    \end{equation}
    Since $\boldsymbol{\alpha}=(\alpha_a) \in \mathcal{C}_{\eta_1}(A,w_1)$ and 
    $\boldsymbol{\beta}=(\beta_b) \in \mathcal{C}_{\eta_2}(B,w_2)$, we see that \eqref{impC2}
    is 
    \begin{align*}
   & \ll_{\varepsilon} (1+|\ell|)^{\varepsilon} N(k)^{\varepsilon} (1 + |t|)^{\varepsilon} \Big ( \frac{A}{N(u_1)} \Big )^{1/2 + \eta_1 + \varepsilon}
    \Big ( \frac{B}{N(u_2)} \Big )^{1/2 + \eta_2 + \varepsilon} \\
   & \ll_{\varepsilon} (1+|\ell|)^{\varepsilon} N(k)^{\varepsilon} (1 + |t|)^{\varepsilon} \Big ( \frac{X}{N(u)} \Big )^{1/2 + \max \{\eta_1,\eta_2\} + \varepsilon},
    \end{align*}
   as required.

  \end{proof}
\section{Narrow Type II/III estimates} \label{narrowtype2sec}
We establish estimates for type-II/III sums that are useful 
in narrow ranges
corresponding to two or three variables of equal size respectively.
In the three variable case, two variables are clumped together 
to reduce to a type-II analysis.
These estimates will be
in ranges where sharp bounds are required (but not asymptotics). 

\subsection{Sieve weights}
We will need to use auxiliary sieve weights in the proof of our narrow range bounds.
\begin{lemma} \label{sievewt1}
  Given $w \geq y^2 > 1$, there exists coefficients $(\lambda_d)_{d \in \mathbb{Z}[\omega]}$ such that
  \begin{enumerate}
  \item \label{eq:first} $\lambda_1 = 1$ and $|\lambda_d| \ll_{\varepsilon} N(d)^{\varepsilon}$ for all $d \in \mathbb{Z}[\omega]$ and all $\varepsilon > 0$;
  \item \label{eq:second} $\lambda_d = 0$ if $N(d) > y^2$ or $d \not \equiv 1 \pmod{3}$;
  \item \label{eq:third} For all $n \in \mathbb{Z}[\omega]$ we have 
    \begin{equation} \label{appsieve1}
      \delta_{(n,\mathscr{P}(w))=1} \leq \sum_{\substack{ d \mid n \\  d \mid \mathscr{P}(w)  \\ N(d) \leq y^2}} \lambda_d;
    \end{equation}
  \item \label{eq:fourth} They satisfy
\begin{equation} \label{sievewtbound1}
\sum_{d \in \mathbb{Z}[\omega]} \frac{\lambda_d}{N(d)} \ll \frac{1}{\log y}.
\end{equation}
  \end{enumerate}
\end{lemma}
\begin{proof}
  Given $d \equiv 1 \pmod{3}$, define
  $$
  \lambda_{d}:= \sum_{\substack{N(e), N(f) \leq y \\ e,f \equiv 1 \pmod{3} \\ d = [e,f]  }} \mu(e) \mu(f)
  \Big (1 - \frac{\log N(e)}{\log y} \Big ) \Big ( 1 - \frac{\log N(f)}{\log y} \Big ).
  $$
  Properties \eqref{eq:first} and \eqref{eq:second} are immediate from the definition. 
Property \eqref{eq:third} follows from 
 \begin{align}
 \sum_{\substack{ d \mid n \\  d \mid \mathscr{P}(w)  \\ N(d) \leq y^2}} \lambda_d=
 \Big( \sum_{\substack{N(e) \leq y \\ e \mid (n,\mathscr{P}(w)) \\ e \equiv 1 \pmod{3}  }}  \mu(e)   \Big (1 - \frac{\log N(e)}{\log y} \Big )   \Big)^2.
 \end{align}
It remains to check property \eqref{eq:fourth}. Observe that
  \begin{align} \label{Hexp}
  \sum_{d \in \mathbb{Z}[\omega]} \frac{\lambda_d}{N(d)} & = \sum_{\substack{ N(e), N(f) \leq y \\ e,f \equiv 1 \pmod{3} }} 
 \frac{\mu(e)\mu(f)}{N([e,f])} \Big ( 1 - \frac{\log N(e)}{\log y} \Big ) \Big ( 1 - \frac{\log N(f)}{\log y} \Big ) \nonumber \\ 
 & = \frac{1}{(2\pi i)^2} \iint_{1/\log y - i \infty}^{1/\log y + i \infty} \frac{H(s,w) \zeta_{\mathbb{Q}(\omega)}(1 + s + w)}{\zeta_{\mathbb{Q}(\omega)}(1 + s) \zeta_{\mathbb{Q}(\omega)}(1 + w)} \cdot \frac{y^{s + w}}{s^2 w^2} \cdot \frac{ds dw}{(\log y)^2},
  \end{align}
where $H(s,w)$ is an analytic and absolutely convergent Euler product for $\Re s, \Re w>-1/4$.  
One can conclude that the display in \eqref{Hexp} is $\ll 1/\log y$
by shifting contours or by carefully bounding the integral using a Taylor expansion around the pole. 
\end{proof}

\subsection{Narrow Type-II/III bound} We are now ready to state the main result of this section. 

\begin{prop}\label{prop:narrow}
   Let $\varepsilon \in (0, \tfrac{1}{10000})$, $A,B \geq 10$, $X:=AB$, $10 \leq w <X^{\varepsilon}$, and $\eta > 0$. 
   Suppose that $\boldsymbol{\alpha} = (\alpha_a)$ is a sequence supported on squarefree $w$-rough $a \in \mathbb{Z}[\omega]$  with
   $a \equiv 1 \pmod{3}$ and $N(a) \in [A / 10, 10 A]$. Suppose that $\boldsymbol{\beta} = (\beta_{b}) \in \mathcal{C}_{\eta}(B, w)$.
  Then
  \begin{align} \label{narrowstate}
  \Big | \sum_{\substack{a,b \in \mathbb{Z}[\omega]}} \alpha_a \beta_b \tilde{g}(a b) \Big | & \leq \frac{\mathcal{K}}{\sqrt{\log w}} 
   \| \boldsymbol{\alpha} \|_{2} \Big( A^{1/2} \| \boldsymbol{\beta} \|_{2} 
  + A^{1/3} \Big( \sum_{b \in \mathbb{Z}[\omega]} \frac{|\beta_{b}|}{N(b)^{1/6}} \Big) \Big) \\
  & + O_{\varepsilon}\Big  ( B^{\eta} \cdot X^{1 + \varepsilon} \Big(\frac{1}{A^{1/2}}+\frac{1}{B} \Big) \Big )+O_{\varepsilon}(X^{1+\varepsilon} \cdot (A/w)^{-1000} ). \nonumber
  \end{align}
  with $\mathcal{K} > 1$ an absolute constant. 
\end{prop}

\begin{proof}
  Without loss of generality we can include the condition $(a,b) = 1$ on the left side \eqref{narrowstate} by \eqref{cuberel}.
  Application of \eqref{twistmult} and Cauchy-Schwarz gives 
 \begin{equation} \label{cauchytwist}
\Big | \sum_{\substack{a,b \in \mathbb{Z}[\omega]}} \alpha_a \beta_b \tilde{g}(a b) \Big | \leq 
\| \boldsymbol{\alpha} \|_{2} \Big ( \sum_{\substack{a \in \mathbb{Z}[\omega] \\ (a,\mathscr{P}(w))=1 \\ a \equiv 1 \pmod{3}}} \Big | \sum_{b} \beta_{b} \tilde{g}(b) \overline{\Big ( \frac{a}{b} \Big )_{3}} \Big |^2 \Big )^{1/2}.
\end{equation}
Let $V:\mathbb{R} \rightarrow \mathbb{R}$ be a fixed smooth positive function
with compact support in $[1/100,100]$. We also stipulate that it
satisfies $V \geq \delta_{[1/10,10]}$. By positivity of the right side 
of \eqref{cauchytwist}, we introduce both the smooth function $V$
and the sieve weight \eqref{appsieve1} on the $a$-sum. Thus the right
side of \eqref{cauchytwist} is
\begin{equation} \label{cauchytwist2}
 \leq \| \boldsymbol{\alpha} \|_{2} \Big ( \sum_{\substack{a \in \mathbb{Z}[\omega] \\ a \equiv 1 \pmod{3}}} 
 V \Big ( \frac{N(a)}{A} \Big ) \Big | \sum_{b} \beta_{b} \tilde{g}(b) \overline{\Big ( \frac{a}{b} \Big )_{3}} \Big |^2 
 \sum_{\substack{d \mid a \\ d \mid \mathscr{P}(w) \\ N(d) \leq y^2}} \lambda_d \Big )^{1/2},
 \end{equation} 
 where $y^2:=w$.
 Expansion of the bracketed sum in \eqref{cauchytwist2} gives 
  \begin{equation} \label{eq:expanded}
 \sum_{\substack{d \mid \mathscr{P}(w)  \\ N(d) \leq y^2}} \lambda_{d} 
 \sum_{b_1, b_2 \in \mathbb{Z}[\omega] } \beta_{b_1} \overline{\beta_{b_2}} \tilde{g}(b_1) \overline{\tilde{g}(b_2)} \sum_{\substack{a \in \mathbb{Z}[\omega] \ , \ d \mid a \\ a \equiv 1 \pmod{3}}} V\Big ( \frac{N(a)}{A} \Big ) \overline{\Big ( \frac{a}{b_1} \Big )_{3}} \Big ( \frac{a}{b_2} \Big )_{3}.
  \end{equation}

\subsubsection*{Diagonal contribution to \eqref{eq:expanded}}
 The diagonal contribution $b_1=b_2=:b$ to \eqref{eq:expanded} is
 \begin{equation} \label{diaginit}
\sum_{\substack{a \in \mathbb{Z}[\omega] \\ a \equiv 1 \pmod{3}}} V \Big( \frac{N(a)}{A} \Big) \sum_{\substack{b \in \mathbb{Z}[\omega] \\ (a,b)=1 }} |\beta_b|^2
\Big(  \sum_{\substack{d \mid a \\ d \mid \mathscr{P}(w) \\ N(d) \leq y^2}} \lambda_d  \Big).
\end{equation}
We can drop the condition $(a,b)=1$ by non-negativity of \eqref{diaginit}
(the bracketed sieve weight divisor sum is non-negative by \eqref{appsieve1}). Thus 
 \begin{align} 
\eqref{diaginit} & \leq  \sum_{\substack{d \mid \mathscr{P}(w) \\ N(d) \leq y^2}} \lambda_d \sum_{\substack{ b \in \mathbb{Z}[\omega] }} |\beta_b|^2 \sum_{\substack{a \in \mathbb{Z}[\omega]  \\ a \equiv 1 \pmod{3} }} V \Big ( \frac{N(a) N(d) }{A} \Big ) \nonumber \\
 &=\frac{4 \pi A}{9 \sqrt{3}}  \|\boldsymbol{\beta} \|^2_2 \Big( \sum_{\substack{d \mid \mathscr{P}(w) \\ N(d) \leq y^2}} \frac{\lambda_d}{N(d)} \Big)
 \sum_{k \in \mathbb{Z}[\omega]} \check{e} \Big({-\frac{k}{3 \lambda}} \Big) \ddot{V} \big ( k \sqrt{A/N(d)} \big )=: \mathscr{D}, \label{diagcont2}
 \end{align}
where \eqref{diagcont2} follows from Poisson summation (in the form of Lemma \ref{radialpois}).

\subsubsection*{Non-diagonal contribution to \eqref{eq:expanded}}
The non-diagonal $b_1 \neq b_2$ contribution to \eqref{eq:expanded} is
\begin{equation} \label{eq:expanded2}
 \sum_{\substack{d \mid \mathscr{P}(w)  \\ N(d) \leq y^2}} \lambda_{d} 
 \sum_{\substack{b_1, b_2 \in \mathbb{Z}[\omega] \\ b_1 \neq b_2 }} 
 \beta_{b_1} \overline{\beta_{b_2}} \tilde{g}(b_1) \overline{\tilde{g}(b_2)} \overline{\Big( \frac{d}{b_1} \Big)_3} \Big( \frac{d}{b_2} \Big)_3 
 \sum_{\substack{a \in \mathbb{Z}[\omega]  \\ a \equiv 1 \pmod{3}}} 
 V\Big ( \frac{N(a) N(d)}{A} \Big ) \overline{\Big ( \frac{a}{b_1} \Big )_{3}} \Big ( \frac{a}{b_2} \Big )_{3}.
  \end{equation}
  For each fixed $b_1,b_2 \in \mathbb{Z}[\omega]$ occurring in
 \eqref{eq:expanded2}, let $e := (b_1, b_2)$. 
 Poisson summation (in the form of Corollary \ref{corpois})
 tells us that
 \begin{align} \label{nondiagpois}
\sum_{\substack{a \in \mathbb{Z}[\omega]  \\ a \equiv 1 \pmod{3}}} &
 V\Big ( \frac{N(a) N(d)}{A} \Big ) \overline{\Big ( \frac{a}{b_1} \Big )_{3}} \Big ( \frac{a}{b_2} \Big )_{3} \nonumber  \\
 &=\frac{4 \pi \overline{\big(\frac{e}{b_1/e} \big)_3} \big( \frac{e}{b_2/e} \big)_3  A \overline{g(b_1 / e)} g(b_2 / e)}{9 \sqrt{3} N(d) N(b_1 b_2 / e)} \sum_{k \in \mathbb{Z}[\omega] } \widetilde{c}_{e}(k) \Big(\frac{k}{b_1/e} \Big)_3 \overline{\Big( \frac{k}{b_2/e} \Big)_3} \ddot{V} \Big ( \frac{k e \sqrt{A}}{\sqrt{N(d)} b_1 b_2} \Big ).
 \end{align}
 
 \begin{remark}
Since $b_1 \neq b_2$, the character $\big(\frac{\cdot}{b_1/e} \big)_3 \overline{\big(\frac{\cdot}{b_2/e}} \big)_3$ in \eqref{nondiagpois}
has conductor with norm $>1$. Hence the dual frequency $k=0$ contributes zero to \eqref{nondiagpois}.
\end{remark}

Observe that \eqref{twistmult} and the squarefree property of $b_1$ and $b_2$ 
imply that
\begin{align} \label{gsimp}
  \tilde{g}(b_1) \overline{\tilde{g}(b_2)}&= \tilde{g}(b_1 / e) \tilde{g}(e) \overline{ \Big(\frac{e}{b_1/e} \Big)_3}
  \overline{\tilde{g}(b_2 / e)} \cdot \overline{\tilde{g}(e)} \Big( \frac{e}{b_2/e} \Big)_3 \nonumber \\
  &= \tilde{g}(b_1 / e) \overline{\tilde{g}(b_2 / e)} \overline{ \Big(\frac{e}{b_1/e} \Big)_3}  \Big( \frac{e}{b_2/e} \Big)_3.
 \end{align}
Upon insertion of \eqref{nondiagpois} and \eqref{gsimp} into \eqref{eq:expanded2},
  we see that \eqref{eq:expanded2} becomes
  \begin{align} \label{nondiagcont}
  \mathscr{N}&:=\frac{4 \pi A}{9 \sqrt{3}} \sum_{\substack{d \mid \mathscr{P}(w) \\ N(d) \leq y^2}} \frac{\lambda_{d}}{N(d)} \sum_{\substack{e \in \mathbb{Z}[\omega] \\ e \equiv 1 \pmod{3} }}  \sum_{\substack{b_1 \neq b_2 \in \mathbb{Z}[\omega] \\ (b_1, b_2) = e}} \frac{ \beta_{b_1} \overline{\beta_{b_2}}}{\sqrt{N(b_1 b_2)}} \nonumber \\
  & \times \sum_{k \in \mathbb{Z}[\omega]} \widetilde{c}_{e}(k) \Big ( \frac{d^2 e k}{b_1 / e} \Big )_{3} \overline{\Big ( \frac{d^2 e k}{b_2 / e} \Big )_{3}} \ddot{V} \Big ( \frac{k e \sqrt{A}}{\sqrt{N(d)} b_1 b_2} \Big ).
  \end{align}
  
  \begin{remark}
 Note that Poisson summation constitutes a key step in the proof - the dual side \eqref{nondiagcont} has no Gauss sum weights.
 \end{remark}
  
 \subsubsection*{Estimates for $\mathscr{D}$ and $\mathscr{N}$}
  We estimate $\mathscr{D}$ and $\mathscr{N}$ displayed 
  in \eqref{diagcont2} and \eqref{nondiagcont} respectively.
  
 Consider $\mathscr{D}$. Lemma \ref{ddotdecay} tells us 
 that 
 \begin{equation*}
 \mathscr{D}=\frac{4 \pi A}{9 \sqrt{3}}  \|\boldsymbol{\beta} \|^2_2 
 \Big( \sum_{\substack{d \mid \mathscr{P}(w) \\ N(d) \leq y^2}} \frac{\lambda_d}{N(d)} \Big) \ddot{V}(0)+O_{\varepsilon}(X^{1+\varepsilon} \cdot (A/w)^{-2000}).
 \end{equation*}
 Application of \eqref{sievewtbound1} gives
 \begin{equation} \label{diagfinal}
 \mathscr{D} \ll  \frac{A}{\log w} \|\boldsymbol{\beta} \|^2_2+O_{\varepsilon}(X^{1+\varepsilon} \cdot (A/w)^{-2000}).
 \end{equation}
 
 Consider $\mathscr{N}$. For a given $d,e \in \mathbb{Z}[\omega]$ 
 in \eqref{nondiagcont}, 
 we split the $k \neq 0$ sum into two subsums: 
 \begin{itemize}
 \item  $0 \neq k \in \mathbb{Z}[\omega]$ such that $d^2 e k=\cube$;
 \item $0 \neq k \in \mathbb{Z}[\omega]$ such that $d^2 e k \neq \cube$.
 \end{itemize}
 Denote the contributions to $\mathscr{N}$ by each of these two cases by $\mathscr{N}_{1}$
 and $\mathscr{N}_2$ respectively. 
 Thus $\mathscr{N}=\mathscr{N}_1+\mathscr{N}_2$.
  
  Since $\mu^2(d)=\mu^2(e)=1$ and $(d,e)=1$, we deduce that
  $d^2 e k=\cube$ iff $k=d e^2 h^3$ for some $h \in \mathbb{Z}[\omega]$. 
  Notice that \eqref{tildec} and Lemma \ref{ramstand} imply that 
  \begin{equation*}
  \widetilde{c}_{e}(de^2 h^3)=\check{e} \Big({-\frac{de^2 h^3}{3 \lambda}} \Big) \varphi(e)=\check{e} \Big({-\frac{h^3}{3 \lambda}} \Big) \varphi(e),
  \end{equation*}
  where the last equality follows from the fact that $d \equiv e \equiv 1 \pmod{3}$ and Remark \ref{dualrem}.
  Thus
  \begin{align} \label{cubenondiag}
 \mathscr{N}_1&=\frac{4 \pi A}{3^3 \sqrt{3}} \sum_{\substack{e \in \mathbb{Z}[\omega] \\ e \equiv 1 \pmod{3}  }} \varphi(e)
 \sum_{\substack{b_1 \neq b_2 \in \mathbb{Z}[\omega] \\ (b_1, b_2) = e}} \frac{ \beta_{b_1} \overline{\beta_{b_2}}}{\sqrt{N(b_1 b_2)}}  \nonumber \\
& \times \sum_{\substack{h \in \mathbb{Z}[\omega] \\ (h, b_1 b_2/e^2) = 1}} 
 \check{e} \Big({-\frac{h^3}{3 \lambda}} \Big) \ddot{V} \Big ( \frac{e^3 h^3 \sqrt{A}}{b_1 b_2} \Big )
 \Big(\sum_{\substack{d \mid \mathscr{P}(w) \\ N(d) \leq y^2}} \frac{\lambda_{d}}{N(d)} \Big). 
 \end{align}
 Note that the extra factor of $1/3$ in the above display for $\mathscr{N}_1$ accounts
 for the fact that $(\omega^{i} h)^3=h^3$ for $i \in \{0,1,2\}$
 and $0 \neq h \in \mathbb{Z}[\omega]$.
 Using Lemma \ref{ddotdecay} and recalling \eqref{sievewtbound1}, we see that all this leads to
 \begin{equation} \label{N1bd}
 \mathscr{N}_1 \ll  \frac{1}{\log w} 
 \Big(A^{2/3} \Big ( \sum_{b \in \mathbb{Z}[\omega]} \frac{|\beta_b|}{N(b)^{1/6}}   \Big )^2 +A \| \boldsymbol{\beta} \|_2^2 \Big).
 \end{equation}

 We now focus on $\mathscr{N}_2$. We have 
   \begin{align*}
  \mathscr{N}_2&:=\frac{4 \pi A}{9 \sqrt{3}} \sum_{\substack{d \mid \mathscr{P}(w) \\ N(d) \leq y^2}} \frac{\lambda_{d}}{N(d)} \sum_{\substack{e \in \mathbb{Z}[\omega] \\ e \equiv 1 \pmod{3} }}  \sum_{\substack{b_1 \neq b_2 \in \mathbb{Z}[\omega] \\ (b_1, b_2) = e}} \frac{ \beta_{b_1} \overline{\beta_{b_2}}}{\sqrt{N(b_1 b_2)}} \\
  & \times \sum_{\substack{k \in \mathbb{Z}[\omega] \\ d^2 e k \neq \tinycube}}
   \widetilde{c}_{e}(k) \Big ( \frac{d^2 e k}{b_1 / e} \Big )_{3} \overline{\Big ( \frac{d^2 e k}{b_2 / e} \Big )_{3}} \ddot{V} \Big ( \frac{k e \sqrt{A}}{\sqrt{N(d)} b_1 b_2} \Big ).
  \end{align*}
  The term $\mathscr{N}_2$ is small because the characters
$\big(\frac{d^2 e k}{\cdot}  \big)_3$ and $\overline{ \big( \frac{d^2 e k}{\cdot} \big)_3 }$
 are both non-principal. Using Lemma \ref{ddotdecay} and Lemma \ref{ramstand}, we re-install the diagonal $b_1=b_2$ in 
 $\mathscr{N}_2$ with acceptable error $O(X^{\varepsilon}(A+B))$. 
  After rescaling the variables $b_1 \rightarrow e b_1$ and $b_2 \rightarrow e b_2$ and using Lemma
  \ref{ddotdecay}, we obtain 
   \begin{align*} 
  \mathscr{N}_2&=\frac{4 \pi A}{9 \sqrt{3}} \sum_{\substack{d \mid \mathscr{P}(w) \\ N(d) \leq y^2}} \frac{\lambda_{d}}{N(d)} 
  \sum_{\substack{e \in \mathbb{Z}[\omega] \\ e \equiv 1 \pmod{3}}}  \frac{1}{N(e)} \sum_{\substack{ b_1,b_2 \in \mathbb{Z}[\omega] \\ (b_1,b_2)=1}} 
  \frac{\beta_{eb_1} \overline{\beta_{eb_2}}  }{\sqrt{N(b_1 b_2)} } \\
  & \times \sum_{\substack{k \in \mathbb{Z}[\omega] \\ d^2 e k \neq \tinycube  \\ N(k) \ll  X^{\varepsilon}( 1+N(d)B^2/(N(e) A) ) }}
  \widetilde{c}_e(k) \Big(\frac{d^2 e k}{b_1 }  \Big)_3  \overline {\Big(\frac{d^2 e k}{b_2 }}  \Big)_3 \ddot{V} \Big ( \frac{k \sqrt{A}}{\sqrt{N(d)} e b_1 b_2} \Big )  \\
  &+O(X^{\varepsilon}(A+B)).
  \end{align*}
  
  After combining \eqref{intparts} and the Mellin--Barnes integral representation \cite[(10.9.22)]{NIST:DLMF} 
  for the $J$-Bessel function, we obtain 
  \begin{equation} \label{ddotmellin}
  \ddot{V}(u)=\frac{(-1)^{L}}{2 \pi i} \int_{0}^{\infty} \int_{-\varepsilon-i \infty}^{-\varepsilon+\infty} V^{(L)}(r^2) r^{2 L+1} 
  \frac{\Gamma(-s)}{\Gamma(L+s+1)} \Big( \frac{2 \pi r |u|}{3 \sqrt{3}} \Big)^{2s}  ds dr,    
  \end{equation}
 for $u \neq 0, L \in \mathbb{Z}_{\geq 1}$. For $L \in \mathbb{Z}_{\geq 1}$ 
 sufficiently large and fixed depending on $\varepsilon>0$, Stirling's asymptotic 
 formula \cite[(5.11.9)]{NIST:DLMF} implies that 
\begin{align} \label{ddottrunc}
 \ddot{V}(u)=\frac{(-1)^{L}}{2 \pi i} \int_{0}^{\infty} \int_{-\varepsilon-iX^{\varepsilon}}^{-\varepsilon+iX^{\varepsilon}} V^{(L)}(r^2) r^{2 L+1} 
   \frac{\Gamma(-s)}{\Gamma(L+s+1)} \Big( \frac{2 \pi r |u|}{3 \sqrt{3}} \Big)^{2s}  ds dr +O_{\varepsilon}(X^{-2000}),  
   \end{align}
  for $u \neq 0$.

  We M\"{o}bius invert the condition $(b_1,b_2)=1$ in the expression for $\mathscr{N}_2$ above, and 
  then separate variables using \eqref{ddottrunc}. A subsequent interchange of the
absolutely convergent finite (recall that $V^{(L)}$ is compactly supported) sums and integrations by Fubini's Theorem gives 
  \begin{align*} 
  \mathscr{N}_2&=(-1)^{L} \cdot \frac{2 A}{9 \sqrt{3} i} \int_{0}^{\infty} \int_{-\varepsilon-iX^{\varepsilon}}^{-\varepsilon+i X^{\varepsilon}} V^{(L)}(r^2) r^{2 L+1} 
   \frac{\Gamma(-s)}{\Gamma(L+s+1)} \Big( \frac{2 \pi r \sqrt{A}}{3 \sqrt{3}} \Big)^{2s}  \\
  & \times \Big ( \sum_{\substack{d \mid \mathscr{P}(w) \\ N(d) \leq y^2}} \frac{\lambda_{d}}{N(d)^{1+s}} \sum_{\substack{f \in \mathbb{Z}[\omega] \\ f \equiv 1 \pmod{3}}} \frac{\mu(f)}{N(f)^{1+2s}} \\
  & \times \sum_{\substack{e \in \mathbb{Z}[\omega] \\ e \equiv 1 \pmod{3}  }}   \sum_{\substack{k \in \mathbb{Z}[\omega] \\ d^2 e k \neq \tinycube  \\ N(k) \ll  X^{\varepsilon} (1+N(d)B^2/(N(e)A)) }}
  \frac{\widetilde{c}_e(k)N(k)^{s}}{N(e)^{1+s}} \Big( \frac{d^2 e k}{f} \Big)_3 \overline{\Big( \frac{d^2 e k}{f} \Big)_3} \nonumber \\
  & \times \Big( \sum_{b_1 \in \mathbb{Z}[\omega] } \frac{\beta_{efb_1}}{\sqrt{N(b_1)}}  \Big(\frac{d^2 e k}{b_1 }  \Big)_3 N(b_1)^{-s} \Big)
  \Big( \sum_{b_2 \in \mathbb{Z}[\omega] } \frac{\overline{\beta_{efb_2}}}{\sqrt{N(b_2)}} \overline {\Big(\frac{d^2 e k}{b_2 }}  \Big)_3 N(b_2)^{-s} \Big) \Big ) ds dr \nonumber \\
  &+O (X^{\varepsilon}(A+B)).
  \end{align*}
We use Axiom 4 of Definition \ref{seqdef} to estimate the sum over $b_1$ and $b_2$ (square root cancellation), and then estimate the
remaining sums trivially using Lemma \ref{ramstand}. We obtain 
\begin{align} \label{N2bd}
\mathscr{N}_2 & \ll A B^{2\eta} X^{\varepsilon} \Big( \sum_{\substack{d \mid \mathscr{P}(w) \\ N(d) \leq y^2}} \frac{1}{N(d)}
\sum_{\substack{f \in \mathbb{Z}[\omega] \\ 1 \leq N(f) \ll B}} \frac{\mu^2(f)}{N(f)} 
\sum_{\substack{ e \in \mathbb{Z}[\omega] \\ 1 \leq N(e) \ll B}} \frac{\mu^2(e)}{N(e)} \nonumber  \\
& \times \sum_{g \mid e} \sum_{\substack{k \in \mathbb{Z}[\omega] \\ (k,e)=g  \\ N(k) \ll  X^{\varepsilon} (1+N(d)B^2/(N(e)A)) } } \frac{\varphi(e)}{\varphi(e/g)} \Big)
+ X^{\varepsilon}(A+B) \nonumber  \\
& \ll  X^{\varepsilon} B^{2\eta} (A+B^2). 
\end{align}

\subsubsection{Conclusion}
Combining \eqref{diagfinal}, \eqref{N1bd} and \eqref{N2bd} tells us that
\begin{equation*}
\eqref{eq:expanded} \leq \frac{\mathcal{K}}{\log w} 
 \Big(A^{2/3} \Big ( \sum_{b \in \mathbb{Z}[\omega]} \frac{|\beta_b|}{N(b)^{1/6}}   \Big )^2 +A \| \boldsymbol{\beta} \|_2^2 \Big)
 + O_{\varepsilon} (B^{2 \eta} (A+B^2) X^{\varepsilon})+O_{\varepsilon}(X^{1+\varepsilon} \cdot (A/w)^{-2000} ),
\end{equation*}
for some absolute constant $\mathcal{K}>1$.
Chasing this bound through \eqref{cauchytwist2} and \eqref{cauchytwist} 
gives the result.

\end{proof}

\section{Type I estimates} \label{type1sec}

We now establish Type-I estimates. In the Proposition below we use the Riemann Hypothesis for the 
Dedekind zeta function attached to $\mathbb{Q}(\omega)$ in order to restrict the sum to squarefree numbers. 

\begin{prop} \label{prop:typeI}
  Assume the Riemann Hypothesis for the family of Dedekind zeta functions attached to $\mathbb{Q}(\omega)$ twisted by
  Gr{\"o}{\ss}encharaktern.    Let $r \in \mathbb{Z}[\omega]$ be squarefree and satisfy $r \equiv 1 \pmod{3}$.
  Let $\varepsilon \in (0, \tfrac{1}{10000})$,
  and $W : \mathbb{R} \rightarrow \mathbb{R}$ be a smooth function with compact support contained in $[1,2]$. Then there exists
  $\rho(\varepsilon) \in (0,\frac{1}{10000})$ such that 
  \begin{align*}
  \sum_{\substack{ u \in \mathbb{Z}[\omega] \\ u \equiv 1 \pmod{3}}} & \tilde{g}(u r) \Big ( \frac{u}{|u|} \Big )^{\ell} \cdot W \Big ( \frac{N(u)}{U} \Big )
   = \delta_{\ell = 0} \cdot \widetilde{W} \Big ( \frac{5}{6} \Big ) \cdot \frac{(2\pi)^{5/3} U^{5/6}}{3^{7/2} \Gamma ( \tfrac 23 )\zeta_{\mathbb{Q}(\omega)}(2;\mathbf{1}_r)} \frac{\varphi(r)}{N(r)^{7/6}} \\ & \qquad \qquad + O_{\varepsilon} \Big( (1 + |\ell|)^{\varepsilon} \cdot \Big( \frac{U^{5/6-\rho(\varepsilon)}}{N(r)^{1/6+\rho(\varepsilon)}} + U^{1/12+\varepsilon} N(r)^{7/12+\varepsilon} \cdot (1 + \ell^6) \Big) \Big),
 \end{align*}
  for all $\ell \in \mathbb{Z}$ and $U \geq 1$.
  \end{prop}

\begin{remark} \label{rhoepsilon}
The function $\rho(\varepsilon)$ is somewhat arbitrary. For instance, it follows from 
\eqref{largeC} that $\rho(\varepsilon)=\frac{11}{12} \varepsilon-\frac{1}{2} \varepsilon^2$ is an acceptable choice. 
\end{remark}

\begin{remark} \label{maintrem}
Mellin inversion of the smooth function, the Class number formula \cite[Chapter VIII, \S 2, Theorem~5]{Lang}, and a contour shift together imply that 
 \begin{align}
\frac{(2\pi)^{2/3}}{3 \Gamma(\tfrac 23)} \sum_{\substack{u \equiv 1 \pmod{3} \\ (u,r) = 1}} &  \frac{\mu^2(u)}{N(u)^{1/6}} W \Big ( \frac{N(u)}{U} \Big ) \cdot \frac{1}{N(r)^{1/6}} \label{altmaint} \\
& = \widetilde{W} \Big(\frac{5}{6} \Big) \frac{(2 \pi)^{5/3} U^{5/6}}{3^{7/2} \Gamma(\tfrac{2}{3}) \zeta_{\mathbb{Q}(\omega)}(2;\mathbf{1}_r) }\frac{\varphi(r)}{N(r)^{7/6}} +O_{\varepsilon}\Big(\frac{U^{1/3+\varepsilon}}{N(r)^{1/6}} \Big).  \nonumber
\end{align}
Thus when $\ell=0$, we can use the main term in \eqref{altmaint} in Proposition \ref{prop:typeI} at negligible cost.
\end{remark}

 \begin{proof}
 M\"{o}bius inversion implies that
 \begin{align} \label{mob}
 \sum_{\substack{u \in \mathbb{Z}[\omega] \\ u \equiv 1 \pmod{3}}} & \tilde{g}(u r) \Big ( \frac{u}{|u|} \Big )^{\ell} W \Big ( \frac{N(u)}{U} \Big ) \nonumber \\
 &=\sum_{\substack{u,e \in \mathbb{Z}[\omega] \\ u,e \equiv 1 \pmod{3} \\ (e,r)=1 }} \tilde{g}(u r) |e| \Big ( \frac{e^3 u}{|e^3 u|} \Big )^{\ell} W \Big ( \frac{N(e)^3 N(u)}{X} \Big ) 
 \Big ( \sum_{\substack{e = c d \\ c,d \equiv 1 \pmod{3}}} \mu(c) \Big ).
 \end{align}
 
 \begin{remark}
 Note that the factor of $|e|$ in \eqref{mob} reflects the periodicity property possessed by the coefficients of the cubic theta function on cubes. See 
 Remark \ref{FSremark}.
 \end{remark}
 
 On the right side of \eqref{mob} we introduce a smooth partition of unity in the $c$ variable.
 Let $V:\mathbb{R} \rightarrow \mathbb{R}$ be a fixed smooth function with compact support 
  contained in $[1,2]$ such that 
\begin{equation} \label{dyadpart}
\sum_{C \text{ dyadic}} V \Big ( \frac{N(c)}{C} \Big ) = 1 \quad \text{for all} \quad 0 \neq c \in \mathbb{Z}[\omega].
 \end{equation}
Insertion of \eqref{dyadpart} into \eqref{mob} yields
\begin{equation} \label{MCUeq}
 \sum_{\substack{u \in \mathbb{Z}[\omega] \\ u \equiv 1 \pmod{3}}} \tilde{g}(u r) \Big ( \frac{u}{|u|} \Big )^{\ell} W \Big ( \frac{N(u)}{U} \Big )=
 \sum_{C \text{ dyadic}} \mathscr{M}(C,U),
\end{equation}
where
  \begin{align} \label{mcudef}
\mathscr{M}(C,U)&:=\sum_{\substack{u,d \in \mathbb{Z}[\omega] \\ u,d \equiv 1 \pmod{3} \\ (d,r)=1 }}  \tilde{g}(u r)  |d| \Big ( \frac{u}{|u|} \Big )^{\ell}
\sum_{\substack{c \in \mathbb{Z}[\omega] \\ c \equiv 1 \pmod{3} \\ (c,r)=1 }} \mu(c) |c| \nonumber  \\
& \times \Big ( \frac{c d}{|c d|} \Big )^{3\ell} V \Big ( \frac{N(c)}{C} \Big )
W \Big ( \frac{N(c d)^3 N(u)}{U} \Big ).
 \end{align}
 We have suppressed the dependence of $\mathscr{M}(C,U)$ 
 on the smooth functions $V$ and $W$ in the notation.
 
 \subsubsection*{Large dyadic $C$}
We estimate the contribution to the right side of \eqref{MCUeq} 
from all dyadic values of $C$ satisfying 
\begin{equation*}
C \geq (U N(r))^{1/12+\varepsilon/2}.
\end{equation*}
We Mellin invert the smooth functions $V$ and $W$ in \eqref{mcudef}. We then use
the rapid decay of their holomorphic Mellin transforms
 $\widetilde{W}$ and $\widetilde{V}$ in vertical strips to truncate the integrations appropriately. A subsequent 
 interchange of the order of absolutely convergent 
 finite sums and integrations by Fubini's Theorem gives
 \begin{align*}
 \mathscr{M}(C,U)&:=\frac{1}{(2 \pi i)^2}   \int_{-i(C(1+|\ell |))^{\varepsilon/1000}}^{i(C(1+|\ell|))^{\varepsilon/1000}} 
 \int_{-i(C(1+|\ell|))^{\varepsilon/1000}}^{i(C(1+|\ell|))^{\varepsilon/1000}}
 \widetilde{V}(s) \widetilde{W}(w ) C^{s} U^{w} \\
& \times \Big( \sum_{\substack{u,d \in \mathbb{Z}[\omega] \\ U/(100 C^3) \leq N(ud^3) \leq 100 U/C^3  \\ u,d \equiv 1 \pmod{3} \\ (d,r)=1 }} 
\tilde{g}(u r)  N(d)^{1/2-3w} \Big ( \frac{u d^3}{|u d^3|} \Big )^{\ell} N(u)^{-w} \\
& \times \sum_{\substack{c \in \mathbb{Z}[\omega] \\ C \leq N(c) \leq 2C \\ c \equiv 1 \pmod{3} \\ (c,r)=1 }} \mu(c) \Big ( \frac{c}{|c|} \Big )^{3 \ell} N(c)^{1/2-s-3w} \Big) ds dw 
+O_{\varepsilon}\big((C(1+|\ell|))^{-2000} \big).
\end{align*}
 To bound the sum over $c$ we appeal to the Riemann Hypothesis for the Dedekind zeta function attached to $\mathbb{Q}(\omega)$ twisted by a Gr{\"o}{\ss}encharakter. 
 Estimating the other summations trivially, we obtain
 \begin{equation} \label{largeC}
 \sum_{\substack {C \text{ dyadic} \\  C \geq (U N(r))^{1/12+\varepsilon/2} }} \mathscr{M}(C,U)
 \ll_{\varepsilon} \Big(\frac{U}{C^3} \Big) C^{1 + \varepsilon} (1 + |\ell|)^{\varepsilon}
 \ll_{\varepsilon} (1 + |\ell|)^{\varepsilon}  \frac{U^{5/6-\rho(\varepsilon)}}{N(r)^{1/6+\rho(\varepsilon)}},
 \end{equation}
 for some $\rho(\varepsilon) \in (0,\frac{1}{10000})$. See also Remark \ref{rhoepsilon}.
 
 \subsubsection*{Small dyadic $C$} 
 It remains to estimate 
 \begin{equation*}
  \sum_{\substack {C \text{ dyadic} \\ 1/2 \leq C  \leq  (U N(r))^{1/12+\varepsilon/2} }} \mathscr{M}(C,U).
 \end{equation*}
 Rearranging the sums in \eqref{mcudef}, and then using \eqref{twistmult}, we obtain 
 \begin{align} \label{mcuarrange}
\mathscr{M}(C,U)&=\tilde{g}(r) \sum_{\substack{c \in \mathbb{Z}[\omega] \\ c \equiv 1 \pmod{3} \\ (c,r)=1 }}  |c| \Big ( \frac{c}{|c|} \Big )^{3\ell} \mu(c) V \Big ( \frac{N(c)}{C} \Big ) \nonumber \\
& \times \sum_{\substack{d,u \in \mathbb{Z}[\omega] \\ d,u \equiv 1 \pmod{3} \\ (d,r)=1 }} 
 |d| \tilde{g}(u) \overline{\Big(\frac{u}{r} \Big)_3} \Big ( \frac{u d^3}{|u d^3|} \Big )^{\ell} W \Big ( \frac{N(ud^3)}{U/N(c)^3} \Big ).
 \end{align}
 
Application of Voronoi summation (in the form of Proposition \ref{voronoiprop}) gives
\begin{align*} \label{voronoiapp}
 \sum_{\substack{d,u \in \mathbb{Z}[\omega] \\ d, u \equiv 1 \pmod{3} \\ (d,r)=1 }} & |d| \tilde{g}(u) \overline{\Big(\frac{u}{r} \Big)_3} \Big ( \frac{u d^3}{|u d^3|} \Big )^{\ell} W \Big ( \frac{N(ud^3)}{U/N(c)^3} \Big ) \\ 
&=\delta_{\ell = 0} \frac{U^{5/6}}{N(c)^{5/2}}   \widetilde{W} \Big ( \frac{5}{6} \Big )  \frac{(2\pi)^{5/3} \varphi(r) \overline{g(r)}}{3^{7/2} \Gamma(\frac{2}{3}) N(r)^{5/3}}  \\
& + \frac{\overline{g(r)}}{3^{7/2} (2\pi)^2} \Big( \frac{\overline{r}}{r}  \Big)^{-\ell} \sum_{\substack{\nu \in \lambda^{-1} \mathbb{Z}[\omega] \\ d \equiv 1 \pmod{3} \\ (d,r)=1 }} \frac{a^{\dagger}(\nu) b_r^{\dagger}(\nu) \big ( \frac{\nu d^3}{|\nu d^3|} \big )^{-\ell}}{N(\nu) N(d)^{5/2}} \widecheck{W_{\ell}} \Big ( \frac{(2\pi)^4 N(d^3 \nu) U}{N(c^3 r^2)} \Big ),   
 \end{align*}
 where the $a^{\dagger}(\cdot)$ and $b_r^{\dagger}(\cdot)$ are given by \eqref{adag} and \eqref{bdag} respectively.
 Insertion of the above display into \eqref{mcuarrange} gives
 \begin{equation} \label{mcudecomp}
 \mathscr{M}(C,U)=\mathscr{T}(C,U)+\mathscr{E}(C,U),
 \end{equation}
 where
 \begin{equation*}
\mathscr{T}(C,U):=\delta_{\ell = 0} \frac{(2 \pi)^{5/3}}{3^{7/2} \Gamma( \frac{2}{3})}  \widetilde{W} \Big( \frac{5}{6} \Big) \frac{\varphi(r)}{N(r)^{7/6}} U^{5/6}
\sum_{\substack{c \in \mathbb{Z}[\omega] \\ c \equiv 1 \pmod{3} \\ (c,r)=1 }}  \frac{\mu(c)}{N(c)^2}V \Big ( \frac{N(c)}{C} \Big ),
\end{equation*}
 and
\begin{align*}
\mathscr{E}(C,U)&:=\frac{N(r)^{1/2}}{3^{7/2} (2 \pi)^2} \Big( \frac{\overline{r}}{r} \Big)^{-\ell}
 \sum_{\substack{\nu \in \lambda^{-1} \mathbb{Z}[\omega] \\ d \equiv 1 \pmod{3} \\ (d,r)=1 }} \frac{a^{\dagger}(\nu) b_r^{\dagger}(\nu)}{N(\nu) N(d)^{5/2}} 
  \Big ( \frac{d^3 \nu}{|d^3 \nu|} \Big )^{-\ell}  \\
& \times \sum_{\substack{c \in \mathbb{Z}[\omega] \\ c \equiv 1 \pmod{3} \\ (c,r)=1 }}  |c| \mu(c) \Big( \frac{c}{|c|} \Big)^{3 \ell} V \Big ( \frac{N(c)}{C} \Big ) \widecheck{W_{\ell}} \Big ( \frac{(2\pi)^4 N(d^3 \nu) U}{N(c^3 r^2)} \Big ).
 \end{align*}
We now collect the main term from the various $\mathscr{T}(C,U)$. We have 
\begin{align} \label{mainvor}
\sum_{\substack {C \text{ dyadic} \\ 1/2 \leq C  \leq  (U N(r))^{1/12+\varepsilon} }} \mathscr{T}(C,U)&=
\sum_{C \text{ dyadic}} \mathscr{T}(C,U)+O_{\varepsilon} \Big(\frac{U^{3/4}}{N(r)^{1/4}} \Big) \nonumber \\
&=\delta_{\ell = 0} \widetilde{W} \Big ( \frac{5}{6} \Big ) \frac{(2\pi)^{5/3} U^{5/6}}{3^{7/2} \Gamma( \tfrac 23  ) \zeta_{\mathbb{Q}(\omega)}(2;\mathbf{1}_r)} \frac{\varphi(r)}{N(r)^{7/6}}
+O_{\varepsilon} \Big(\frac{U^{3/4}}{N(r)^{1/4}} \Big),
\end{align}
where the error term follows from a trivial estimation of the tail of $\zeta_{\mathbb{Q}(\omega)}$.

The various $\mathscr{E}(C,U)$ 
contribute the error term in the statement of the result.
By the rapid decay of $\widecheck{W}$ in \eqref{Wcheckdecay},
we truncate the $d$ and $\nu$ sums in $\mathscr{E}(C,U)$ to 
\begin{equation*}
N(d^3 \nu) \ll (UN(r)(1+|\ell |) )^{\varepsilon/1000} (1 + \ell^4) \cdot \Big(1+\frac{C^3 N(r)^2}{U} \Big)=:\mathcal{Z}, 
\end{equation*}
with negligible error. To separate variables, we subsequently use the definition \eqref{Wcheckdef}
of $\widecheck{W}$ and Mellin inversion on $V$. We truncate the resulting integrations appropriately using 
the rapid decay of $\widetilde{V}$ and $\widetilde{W}$. A subsequent 
 interchange of the order of absolutely convergent 
 finite sums and integrations by Fubini's Theorem gives
 \begin{align*}
 \mathscr{E}(C,U) &=\frac{N(r)^{1/2}}{3^{7/2} (2 \pi)^4} \Big( \frac{\overline{r}}{r} \Big)^{-\ell} \int_{-i(UN(r)(1+|\ell|))^{\varepsilon/1000}}^{i(UN(r) (1+|\ell| ))^{\varepsilon/1000}} \int_{-i(UN(r) (1+|\ell|))^{\varepsilon/1000}}^{i(UN(r) (1+|\ell|))^{\varepsilon/1000}} 
 \frac{\Gamma(\frac{5}{6}+ \frac{|\ell|}{2} -w) \Gamma(\frac{7}{6}+ \frac{|\ell|}{2} -w)}{\Gamma(w + \frac{|\ell|}{2} -\frac{1}{6}) \Gamma(w + \frac{|\ell|}{2} +\frac{1}{6})}  \nonumber \\
& \times \Big( \frac{(2 \pi)^4 U}{N(r)^2} \Big)^{w}  C^{s} \widetilde{V}(s) \widetilde{W}(w)  \\
& \times
 \Big( \sum_{\substack{\nu \in \lambda^{-1} \mathbb{Z}[\omega] \\ d \equiv 1 \pmod{3} \\ (d,r)=1  \\ N(d^3 \nu) \ll \mathcal{Z}}}  \frac{a^{\dagger}(\nu) b_{r}^{\dagger}(\nu)}{N(\nu)^{1-w} N(d)^{5/2-3w}} \Big ( \frac{d^3 \nu}{|d^3 \nu|} \Big )^{-\ell} \hspace{-0.4cm}
 \sum_{\substack{c \in \mathbb{Z}[\omega] \\ c \equiv 1 \pmod{3} \\ (c,r)=1 \\ C \leq N(c) \leq 2C }}  N(c)^{1/2-s-3w} \mu(c) \Big( \frac{c}{|c|} \Big)^{3 \ell} \Big) ds dw \\
 &+ O_{\varepsilon} \big((UN(r)(1+|\ell|) )^{-1000} \big). 
 \end{align*}
 We estimate the sum over $c$ using the Riemann hypothesis for the Dedekind zeta function attached to $\mathbb{Q}(\omega)$,
 and the quotient of Gamma factors using Stirling's asymptotic \cite[(5.11.1)]{NIST:DLMF}
 The other sums are estimated trivially using \eqref{adag}, \eqref{bdag}, \eqref{cubiccoeff}, \eqref{cubiccoeff1} and \eqref{cubiccoeff2}.
We obtain 
\begin{align*}
\mathscr{E}(C,U) & \ll_{\varepsilon} N(r)^{1/2+\varepsilon/10} \cdot (1 + \ell^2)  \\
& \times \sum_{e \mid r} \sum_{\substack{\nu \in \lambda^{-1} \mathbb{Z}[\omega] \\ d \equiv 1 \pmod{3} \\ (d,r)=1 \\
(\lambda \nu,r)=e \\ N(d^3 \nu) \ll \mathcal{Z} } } 
\frac{|a^{\dagger}(\nu)|}{N(\nu) N(d)^{5/2}} \frac{\varphi(r)}{\varphi(r/e)}    C^{1+\varepsilon/10} (1+| \ell |)^{\varepsilon/10}, \\
& \ll  N(r)^{1/2+\varepsilon/4} U^{\varepsilon/4} (1 + \ell^6) C^{1+\varepsilon/4} (1+|\ell|)^{\varepsilon/4},
\end{align*}
and so
\begin{equation} \label{errorvor}
\sum_{\substack {C \text{ dyadic} \\ 1/2 \leq C  \leq  (U N(r))^{1/12+\varepsilon/2} }} \mathscr{E}(C,U) 
\ll_{\varepsilon} U^{1/12+\varepsilon} N(r)^{7/12+\varepsilon} \cdot (1 + \ell^6) (1+|\ell|)^{\varepsilon}. 
\end{equation}
After combining \eqref{largeC}, \eqref{mainvor} and \eqref{errorvor}, we obtain the result.
  
\end{proof}

 We also record the following nearly immediate Corollary.

 \begin{corollary} \label{cor:typeI}
   Assume the Riemann Hypothesis for the family of Dedekind zeta functions attached to $\mathbb{Q}(\omega)$ twisted by Gr{\"o}{\ss}encharaktern.
   Let $r \in \mathbb{Z}[\omega]$ be squarefree and satisfy $r \equiv 1 \pmod{3}$. Let
   $\varepsilon \in (0, \tfrac{1}{10000})$,
  and $V, W : \mathbb{R} \rightarrow \mathbb{R}$ be smooth functions with compact support contained in $[\tfrac 14,4]$. 
  Then there exists $\rho(\varepsilon) \in (0,\frac{1}{10000})$ such that
  \begin{align*}
&  \sum_{\substack{ u \in \mathbb{Z}[\omega] \\ u \equiv 1 \pmod{3}}} \tilde{g}(u r) \Big ( \frac{u}{|u|} \Big )^{\ell} \cdot V \Big ( \frac{N(u)}{U} \Big ) W \Big ( \frac{N(u r)}{X} \Big ) 
    \\ & 
  \qquad \qquad = \delta_{\ell=0} \cdot \frac{(2\pi)^{2/3}}{3 \Gamma(\tfrac 23)} \sum_{\substack{u \equiv 1 \pmod{3} \\ (u,r) = 1}} \frac{\mu^2(u)}{N(u)^{1/6}} V \Big ( \frac{N(u)}{U} \Big ) W \Big ( \frac{N(u r)}{X} \Big ) \cdot \frac{1}{N(r)^{1/6}} \\
 & \qquad \qquad \qquad + O_{\varepsilon} \Big(  (1 + |\ell|)^{\varepsilon} \cdot \Big (  \frac{U^{5/6-\rho(\varepsilon)}}{N(r)^{1/6+\rho(\varepsilon)}}  +U^{1/12+\varepsilon} N(r)^{7/12+\varepsilon} (1 + \ell^6) \Big) +\frac{U^{1/3+\varepsilon}}{N(r)^{1/6}}  \Big),
 \end{align*}
 for all $\ell \in \mathbb{Z}$ and $X \geq U \geq 1$.
 \end{corollary}
 \begin{proof}
   If $U N(r) \asymp X$ then we simply apply the previous result with a different weight function and use the Remark \ref{maintrem}. 
   If $U N(r)$ is not of the order of magnitude of $X$ then both main terms are zero. 
 \end{proof}
 
  \begin{remark} \label{type1rangerem}
  The main term in Corollary \ref{cor:typeI} is larger than the error term when $U \geq (N(r)(1+\ell^6))^{1+o(1)}$.
  \end{remark}

\section{Improved cubic large sieve}\label{sec:dispersion}
The cubic large sieve of Heath--Brown is as follows.
\begin{theorem} \label{cubicHBloc} \emph{\cite[Theorem~2]{HB}}
Let $A,B \geq 1$, $\varepsilon>0$, and $(\beta_b)_{b \in \mathbb{Z}[\omega]}$ be an arbitrary sequence of complex numbers
with support contained in the set of squarefree elements of $\mathbb{Z}[\omega]$.
Then
\begin{equation} \label{HBbd}
\sum_{\substack{ N(a) \leq A \\ a \equiv 1 \pmod{3} }} \mu^2(a)  \Big | \sum_{ \substack{N(b) \leq B \\  b \equiv 1 \pmod{3} } } \beta_b \Big( \frac{b}{a} \Big)_3   \Big|^2 
\ll_{\varepsilon} \big(A+B+(AB)^{2/3} \big)(AB)^{\varepsilon} \sum_{b \in \mathbb{Z}[\omega]} |\beta_b|^2.
\end{equation}
\end{theorem}

Recall the operator norm $\mathcal{B}(A,B)$ defined in \eqref{opnorm}.
The Duality Principle \cite[(7.9)--(7.11)]{IwK} and cubic reciprocity imply that
\begin{equation} \label{duality}
\mathcal{B}(A,B)=\mathcal{B}(B,A).
\end{equation}
See also \cite[Lemma 4]{HB}. We also have the following simple monotonicity property.

\begin{lemma} \emph{\cite[Lemma~5]{HB}} \label{monlemma}
  There exists an absolute constant $C \geq 1$ as follows. Let $A, B_1, B_2 \gg 1$
  and $B_2 \geq C B_1 \log (2AB_1)$. Then,
  $$
  \mathcal{B}(A,B_1) \ll \mathcal{B}(A, B_2). 
  $$
\end{lemma} 

We now prove that Heath-Brown's cubic large sieve is optimal under the Generalized Riemann Hypothesis for 
Hecke $L$-functions over $\mathbb{Q}(\omega)$.

\begin{proof}[Proof of Theorem \ref{hboptimal}]
The upper bounds in the given ranges follow from the cubic large sieve Theorem \ref{cubicHBloc},
and are unconditional.

We now focus on the conditional lower bounds.
Let $\xi>0$ be small and fixed, $A,B \geq 10$, $X:=AB$ and $A \in [10,X^{1/2-\xi}]$. 
Consider the sequence
\begin{equation}
\beta_b=\overline{\tilde{g}(b)} W \Big( \frac{N(b)}{B} \Big) ,
\end{equation}
where $W$ is a smooth compactly supported function in $(1,2)$.
It is supported only on squarefree elements by \eqref{cuberel}.
Then 
\begin{align}
& \sum_{\substack{A<N(a) \leq 2A \\ a \equiv 1 \pmod{3} }} \mu^2(a)
\Big | \sum_{ \substack{B<N(b) \leq 2B \\  b \equiv 1 \pmod{3} } } 
\overline{\tilde{g}(b)} W \Big( \frac{N(b)}{B} \Big) \Big( \frac{b}{a} \Big)_3   \Big|^2 \nonumber \\
& = \sum_{\substack{A< N(a) \leq 2A \\ a \equiv 1 \pmod{3} }} \mu^2(a)
\Big | \sum_{ \substack{b \equiv 1 \pmod{3} } } 
\overline{W \Big( \frac{N(b)}{B}  \Big)} \tilde{g}(a b) \Big|^2 \quad (\text{by }  \eqref{cuberel} \text{ and } \eqref{twistmult})   \nonumber  \\
& \gg A \Big( \frac{B^{5/6}}{A^{1/6}}    \Big)^2 +O \big (X^{o(1)} (A^{17/12} B^{11/12}+A^{13/6} B^{1/6}  ) \big) \label{voro} \\
& \gg_{\xi,W} A^{2/3}B^{5/3}, \label{voro2}
\end{align}
where display \eqref{voro} follows from Voronoi summation (Proposition \ref{prop:typeI})
and the GRH hypothesis, and \eqref{voro2} follows from the fact that we are in the range $A \ll (AB)^{1/2-\xi}$.
Thus $\mathcal{B}(A,B) \gg_{\xi} (AB)^{2/3} \gg A+B$ for $A \ll (AB)^{1/2-\xi}$ and $A \in [ X^{1/3},X^{1/2}]$. 
Combining this result with \eqref{duality} then gives the claim
when $A \in [\sqrt{B}, B^2] \setminus [B^{1-\xi},B^{1+2\xi}] $. 
The result in the range $A \in [B^{1 - \xi}, B^{1 +2 \xi}]$ then follows from \eqref{duality} and Lemma \ref{monlemma}, 
so $\mathcal{B}(A, B) \gg \mathcal{B}(A X^{-3\xi}, B) \gg (A B)^{2/3 - 3 \xi}$.
\end{proof}

In light of the proof of Theorem \ref{hboptimal}, we renormalise the sequences we consider in 
the cubic large sieve by setting 
$c_b:=\tilde{g}(b) \beta_b$ where 
$\boldsymbol{\beta}:=(\beta_b)_{b \in \mathbb{Z}[\omega]}$ is a sequence
supported on squarefree $b \equiv 1 \pmod{3}$.
Note that $|c_b|=|\beta_b|$ on squarefree $b \equiv 1 \pmod{3}$ by \eqref{cuberel}.
We are able to refine Theorem \ref{cubicHBloc} in a special case by:
\begin{enumerate}
\item Introducing a non-trivial asymptotic main term;
\item Assuming additional cancellations/density restrictions for the sequence $\boldsymbol{\beta} = (\beta_{b})$.
\end{enumerate}

\begin{prop} \label{prop:main}
Let $V:(0,\infty) \rightarrow \mathbb{R}_{\geq 0}$ be a smooth compactly supported function,
  $0<\eta \leq 1/4$, $A, B, w \geq 10$ and $X := A B$. 
 Suppose that $w > (\log X)^{10}$ and $\boldsymbol{\beta} = (\beta_{b}) \in \mathcal{C}_{\eta}(B, w)$. Let $\varepsilon \in (0,\frac{1}{10000})$
 and $\pi \equiv 1 \pmod{3}$
be a prime or $1$.
 Then there exists $\rho(\varepsilon) \in (0,\frac{1}{10000})$ such that uniformly in 
 $1 \leq N(\pi)<w$ we have
  \begin{align*} 
 \sum_{\substack{a \in \mathbb{Z}[\omega] \\ a \equiv 1 \pmod{3} \\ \pi \mid a}} & \mu^2(a) V \Big ( \frac{N(a)}{A} \Big ) \Big |  \sum_{\substack{b \in \mathbb{Z}[\omega] \\ (b,a) = 1}} \beta_b \tilde{g}(b) \overline{\Big ( \frac{a}{b} \Big )_{3}} - \frac{(2\pi)^{2/3}}{3 \Gamma(\tfrac 23)} \frac{\overline{\tilde{g}(a)}}{N(a)^{1/6}} \sum_{\substack{ b \in \mathbb{Z}[\omega] \\ (b,a)=1 }} \frac{\beta_b}{N(b)^{1/6}} \Big |^2 \\ & \ll_{\varepsilon} \frac{A^{2/3} B^{5/3}}{N(\pi)} \Big(\frac{1}{w^{9/10}}+\frac{\delta_{\pi \neq 1}}{N(\pi)} \Big) + \frac{A^{2/3-\rho(\varepsilon)} B^{5/3-\rho(\varepsilon)}}{N(\pi)}
 +\frac{A^{1/6+\varepsilon} B^{5/3}}{N(\pi)^{1/2+\varepsilon}} \\
 &+ (N(\pi) X)^{\varepsilon} \Big( N(\pi)^{1/2} B^{29/12} A^{-1/12}+B^{2 + 2 \eta} +\frac{X}{N(\pi)} \Big( 1+(B^2/A)^{-1000} \Big) \Big).
    \end{align*}
\end{prop}

Proposition \ref{prop:main} will follow from a modification of the proof of
Proposition \ref{prop:adj} using sieve weights, and from Lemma \ref{squarefreesieve} below.
At the close of this section we sketch how Proposition \ref{prop:main} follows.

  \begin{prop} \label{prop:adj}
  Let $V:(0,\infty) \rightarrow \mathbb{R}_{\geq 0}$ be a smooth compactly supported function,
   $0<\eta \leq 1/4$,
    $A,B,w \geq 10$ and $X:=AB$.
  Suppose that $w > (\log X)^{10}$ and  $\boldsymbol{\beta} = (\beta_{b}) \in \mathcal{C}_{\eta}(B, w)$.
Let $\varepsilon \in (0,\frac{1}{10000})$, $\Delta \geq 1$, $\pi \equiv 1 \pmod{3}$
be a prime or $1$, and $\gamma \equiv 1 \pmod{3}$ be squarefree such that $(\pi,\gamma)=1$.
Then there exists $\rho(\varepsilon) \in (0,\frac{1}{10000})$
such that uniformly in $1 \leq N(\gamma) \leq \Delta$ and $1 \leq N(\pi)<w$
we have
\begin{align}
& \sum_{\substack{a \in \mathbb{Z}[\omega] \\ a \equiv 1 \pmod{3} \\ \pi \gamma^2 \mid a}}  V \Big ( \frac{N(a)}{A} \Big ) \Big | \sum_{\substack{b \in \mathbb{Z}[\omega] \\ (b, a) = 1}} \beta_{b} \tilde{g}(b) \overline{\Big ( \frac{b}{a} \Big )_{3}} - \frac{(2\pi)^{2/3}}{3 \Gamma(\tfrac 23)} \frac{\overline{\tilde{g}(a)}}{N(a)^{1/6}} \sum_{\substack{b \in \mathbb{Z}[\omega] 
\\ (b,a)=1}} \frac{\beta_{b}}{N(b)^{1/6}} \Big |^2  \label{proplhs}
\\ &  \leq 2 \Big ( \frac{1}{N(\gamma)^2} - \frac{\delta_{\gamma = 1} }{\zeta_{\mathbb{Q}(\omega)}(2;\mathbf{1}_{\pi})} \Big ) \widetilde{V} \Big ( \frac{2}{3} \Big ) \frac{2\pi}{9 \sqrt{3}} \Big ( \frac{(2\pi)^{2/3}}{3 \Gamma(\tfrac 23)} \Big )^{3} \frac{A^{2/3}}{N(\pi)} \Big | \sum_{b \in \mathbb{Z}[\omega] } \frac{\beta_{b}}{N(b)^{1/6}} \Big |^2 \nonumber \\ 
& + O_{\varepsilon} \Big ( \frac{1}{N(\pi \gamma^2)} \Big(  \frac{A^{2/3} B^{5/3}}{w^{9/10}}  
+ (\Delta N(\pi) X)^{\varepsilon}\Big (B^{2 + 2 \eta} N(\pi) \Delta^2+X \Big(1+(B^2/A)^{-1000} \Big)   \Big) \Big) \Big ) \nonumber  \\
& +\delta_{\gamma=1} \cdot \Big( O_{\varepsilon} \Big (  \frac{A^{2/3-\rho(\varepsilon)} B^{5/3-\rho(\varepsilon)}}{N(\pi)}
+ \frac{A^{2/3} B^{5/3}}{N(\pi)} \Big( \frac{1}{w^{9/10}}+\frac{\delta_{\pi \neq 1}}{N(\pi)} \Big)  \nonumber \\
& + X^{\varepsilon} N(\pi)^{1/2} B^{29/12} A^{-1/12} +\frac{A^{1/6+\varepsilon} B^{5/3}}{N(\pi)^{1/2+\varepsilon}} \Big ) \Big).  \nonumber
\end{align}
\end{prop}

\begin{proof}
Using \eqref{cuberel} and inclusion-exclusion with the condition $(a,b)=1$, we see that
the expression in \eqref{proplhs}
is equal to
\begin{align}  \label{incexc}
& \sum_{\substack{a \in \mathbb{Z}[\omega] \\ a \equiv 1 \pmod{3} \\ \pi \gamma^2 \mid a}}  V \Big ( \frac{N(a)}{A} \Big )  \Big | \Big (\sum_{\substack{b \in \mathbb{Z}[\omega] \\ (b, \gamma a)= 1}} \beta_{b} \tilde{g}(b) \overline{\Big ( \frac{b}{a} \Big )_{3}} - \frac{(2\pi)^{2/3}}{3 \Gamma(\tfrac 23)} \frac{\delta_{\gamma=1} \overline{\tilde{g}(a)}}{N(a)^{1/6}} \sum_{\substack{ b \in \mathbb{Z}[\omega] }} \frac{\beta_{b}}{N(b)^{1/6}} \Big ) \nonumber \\
& \qquad \qquad \qquad \qquad \qquad  + \Big ( \frac{(2\pi)^{2/3}}{3 \Gamma(\tfrac 23)} \frac{\delta_{\gamma=1} \overline{ \tilde{g}(a)}}{N(a)^{1/6}}  \sum_{\substack{ b \in \mathbb{Z}[\omega] \\ (b,a) \neq 1 } } \frac{\beta_b}{N(b)^{1/6}} \Big )  \Big |^2.
\end{align}
Since $V$ is non-negative we can apply the parallelogram identity
\begin{equation} \label{trivineq}
|X+Y|^2 \leq 2(|X|^2+|Y|^2) \quad \text{for all} \quad X,Y \in \mathbb{C},
\end{equation}
to the display above. This shows that \eqref{incexc} is
\begin{align}
& \leq 2 \Big( \sum_{\substack{a \in \mathbb{Z}[\omega] \\ a \equiv 1 \pmod{3} \\ \pi \gamma^2 \mid a}}  V \Big ( \frac{N(a)}{A} \Big )  \Big |\sum_{\substack{b \in \mathbb{Z}[\omega] \\ (b, \gamma a)= 1}} \beta_{b} \tilde{g}(b) \overline{\Big ( \frac{b}{a} \Big )_{3}} - \frac{(2\pi)^{2/3}}{3 \Gamma(\tfrac 23)} \frac{\delta_{\gamma=1} \overline{\tilde{g}(a)}}{N(a)^{1/6}} \sum_{\substack{ b \in \mathbb{Z}[\omega] }} \frac{\beta_{b}}{N(b)^{1/6}} \Big |^2 \label{trivred1} \\
&\qquad \qquad \qquad + \delta_{\gamma=1} \frac{(2\pi)^{4/3}}{9 \Gamma(\tfrac 23)^2} \sum_{\substack{a \in \mathbb{Z}[\omega] \\ a \equiv 1 \pmod{3} \\ \pi \mid a}} \frac{\mu^2(a)}{N(a)^{1/3}} V \Big( \frac{N(a)}{A} \Big) \Big |  \sum_{\substack{ b \in \mathbb{Z}[\omega] \\ (b,a) \neq 1 } } \frac{\beta_b}{N(b)^{1/6}}  \Big |^2     \Big), \label{trivred2}
\end{align}
where we used \eqref{cuberel} to obtain the last display. The term in \eqref{trivred2} is equal to 
\begin{equation} \label{trivO}
\delta_{\gamma=1} \frac{(2\pi)^{4/3}}{9 \Gamma(\tfrac 23)^2} \sum_{b_1,b_2 \in \mathbb{Z}[\omega]} \frac{\beta_{b_1} \overline{ \beta_{b_2}}}{N(b_1 b_2)^{1/6}} 
\sum_{\substack{a \in \mathbb{Z}[\omega] \\ a \equiv 1 \pmod{3} \\ \pi \mid a \\ (a,b_1) \neq 1 \\ (a,b_2) \neq 1 }} \frac{\mu^2(a)}{N(a)^{1/3}}  V \Big ( \frac{N(a)}{A} \Big )
=\delta_{\gamma=1} O \Big(\frac{A^{2/3} B^{5/3} }{N(\pi) w^{9/10}}  \Big). 
\end{equation}
The estimate in \eqref{trivO} follows from the triangle inequality, the fact $(b_1 b_2, \pi)=1$
($b_1 b_2$ is $w$-rough and $\pi$ is $w$-smooth) and
\begin{equation*}
\sum_{\substack{\varpi \mid b_1 b_2 \\ \varpi \text { prime}}} \frac{1}{N(\varpi)} \ll \frac{\log B^2}{\log \log B^2} \frac{1}{w} \ll \frac{1}{w^{9/10}},
\quad \text{for} \quad w>(\log X)^{10} \quad \text{say}.
\end{equation*}
We repeatedly use this $w$-roughness argument in the course of the proof.

It suffices to compute the term in \eqref{trivred1}. We make the change of variable $a \mapsto \pi \gamma^2 a$. After using cubic reciprocity
  and \eqref{cuberel}, it suffices to compute
  \begin{align} \label{firstexp}
  \sum_{\substack{a \in \mathbb{Z}[\omega] \\ a \equiv 1 \pmod{3}}}  V \Big ( \frac{N(\pi \gamma^2) N(a)}{A} \Big ) \Big | & \sum_{\substack{b \in \mathbb{Z}[\omega] \\ (b,  \pi \gamma a) = 1}} \beta_{b} \tilde{g}(b) \overline{\Big ( \frac{a}{b} \Big )_{3}} \Big ( \frac{\pi^2 \gamma}{b} \Big )_{3} 
  - \frac{(2\pi)^{2/3}}{3 \Gamma(\tfrac 23)} \frac{\delta_{\gamma = 1} \overline{\tilde{g}(a \pi )}}{N(a \pi)^{1/6}} \sum_{\substack{b \in \mathbb{Z}[\omega] }} \frac{\beta_{b}}{N(b)^{1/6}} \Big |^2. 
 \end{align}
Expansion of the square in \eqref{firstexp} shows that we need to evaluate 
 the diagonal term
  \begin{equation} \label{diagonalterm}
  \mathscr{D} := \sum_{\substack{a \in \mathbb{Z}[\omega] \\ a \equiv 1 \pmod{3}}}  V \Big ( \frac{N(\pi \gamma^2) N(a)}{A} \Big ) \Big | \sum_{\substack{b \in \mathbb{Z}[\omega] \\ (b, \pi\gamma a) = 1}} \beta_{b} \tilde{g}(b) \overline{\Big ( \frac{a}{b} \Big )_{3}} \Big ( \frac{\pi^2 \gamma}{b} \Big )_{3} \Big |^2;
  \end{equation}
  the cross term
  \begin{align} \label{crossterm}
  \mathscr{C} := - 2 \frac{(2\pi)^{2/3}}{3 \Gamma(\tfrac 23)} \delta_{\gamma=1} \text{Re} \Big ( & \sum_{\substack{b_1, b_2 \in \mathbb{Z}[\omega]}} \frac{\beta_{b_1} \tilde{g}(b_1) \overline{\beta_{b_2}}}{N(b_2)^{1/6}} \nonumber \\
  & \times \sum_{\substack{a \in \mathbb{Z}[\omega] \\ a \equiv 1 \pmod{3} \\ (a\pi,b_1)=1 }} 
  V \Big ( \frac{N(a) N(\pi)}{A} \Big ) \overline{\Big ( \frac{a}{b_1} \Big )_{3}} \Big ( \frac{\pi^2}{b_1} \Big )_{3} \frac{\tilde{g}(a \pi)}{N(a \pi)^{1/6}} \Big ); 
 \end{align}
  and the trivial term,
  \begin{equation} \label{trivialterm}
  \mathscr{T} := 
\Big ( \frac{(2\pi)^{2/3}}{3 \Gamma(\tfrac 23)} \Big )^2 \delta_{\gamma=1} \sum_{\substack{a \in \mathbb{Z}[\omega] \\ a \equiv 1 \pmod{3}}} V \Big ( \frac{N(a) N(\pi)}{A} \Big ) \frac{\mu^2(a \pi)}{N(a \pi)^{1/3}} \Big | \sum_{\substack{b \in \mathbb{Z}[\omega]}} \frac{\beta_{b}}{N(b)^{1/6}} \Big |^2.  
 \end{equation}
The appearance of $\mu^2(a \pi)$ in $\mathscr{T}$ is due to $|\widetilde{g}(a \pi)|^2 = \mu^2(a \pi)$ (a consequence of \eqref{cuberel}).

In the course of this proof we will establish the three asymptotic estimates:
\begin{align} \label{totalD}
\mathscr{D}&=\frac{2\pi}{9 \sqrt{3}}  \Big ( \frac{(2\pi)^{2/3}}{3 \Gamma(\tfrac 23)} \Big )^2 \widetilde{V}\Big ( \frac{2}{3} \Big ) \frac{A^{2/3}}{N(\pi \gamma^2)} 
 \Big | \sum_{b \in \mathbb{Z}[\omega] } \frac{\beta_{b}}{N(b)^{1/6}} \Big |^2 
 + O \Big ( \frac{A^{2/3} B^{5/3}}{N(\pi \gamma^2) w^{9/10}} \Big)  \nonumber \\
& + O_{\varepsilon} \Big(  (\Delta N(\pi) X)^{\varepsilon} \Big(\frac{B^{2 + 2 \eta} N(\pi) \Delta^2}{N(\pi \gamma^2)}+\frac{X}{N(\pi \gamma^2)}
 \big(1+(B^2/A)^{-1000} \big)  \Big) \Big ),
\end{align}
\begin{align} \label{crosstermfinal}
\mathscr{C}&=- 2 \widetilde{V} \Big ( \frac{2}{3} \Big ) \delta_{\gamma=1} \frac{2 \pi}{9 \sqrt{3} \zeta_{\mathbb{Q}(\omega)}(2;\mathbf{1}_{\pi})} \Big ( \frac{(2\pi)^{2/3}}{3 \Gamma(\tfrac 23)} \Big )^2 \frac{A^{2/3}}{N(\pi)} \Big | \sum_{b \in \mathbb{Z}[\omega]} \frac{\beta_{b}}{N(b)^{1/6}} \Big |^2 \nonumber \\
&+\delta_{\gamma=1} \Big (O_{\varepsilon} \Big( \frac{A^{2/3-\rho(\varepsilon)} B^{5/3-\rho(\varepsilon)}}{N(\pi)} \Big)+O_{\varepsilon}( X^{\varepsilon} N(\pi)^{1/2} B^{29/12} A^{-1/12}) \nonumber \\
&+O \Big( \frac{A^{2/3} B^{5/3}}{N(\pi) } \Big( \frac{1}{w^{9/10}}+\frac{\delta_{\pi \neq 1}}{N(\pi)} \Big)    \Big) \Big),
\end{align}
and
\begin{align} \label{trivialfinal}
\mathscr{T}&=\delta_{\gamma=1} \widetilde{V}\Big ( \frac{2}{3} \Big ) \frac{2\pi}{9 \sqrt{3} \zeta_{\mathbb{Q}(\omega)}(2;\mathbf{1}_{\pi})} \Big ( \frac{(2\pi)^{2/3}}{3 \Gamma(\tfrac 23)} \Big )^2 \frac{A^{2/3}}{N(\pi)} \Big | \sum_{b \in \mathbb{Z}[\omega]} \frac{\beta_{b}}{N(b)^{1/6}} \Big |^2 \nonumber \\
&+ \delta_{\gamma=1} \cdot \delta_{\pi \neq 1} \cdot O \Big ( \frac{A^{2/3} B^{5/3}}{N(\pi)^2} \Big)+
\delta_{\gamma=1} O_{\varepsilon} \Big( \frac{A^{1/6+\varepsilon} B^{5/3}}{N(\pi)^{1/2+\varepsilon}} \Big).
\end{align}
Thus $\mathscr{D}+\mathscr{C}+\mathscr{T}$ using the asymptotics \eqref{totalD}, \eqref{crosstermfinal} and \eqref{trivialfinal} respectively
gives an asymptotic expression for \eqref{firstexp}. Substitution of this asymptotic expression into 
\eqref{trivred1}, and subsititution of \eqref{trivO} into \eqref{trivred2}, will give the result.

We now turn our attention to proving \eqref{totalD}--\eqref{trivialfinal}.

\subsubsection*{The diagonal term $\mathscr{D}$}
After expansion of the square in \eqref{diagonalterm},
we obtain 
 $$
\mathscr{D} = \sum_{\substack{b_1, b_2 \in \mathbb{Z}[\omega]  \\ (b_1 b_2, \pi \gamma) = 1}} \beta_{b_1} \tilde{g}(b_1)\overline{\beta_{b_2} \tilde{g}(b_2)} \Big ( \frac{\pi^2 \gamma}{b_1} \Big )_{3} \overline{\Big ( \frac{\pi^2 \gamma}{b_2} \Big )_{3}} \sum_{\substack{a \in \mathbb{Z}[\omega] \\ a \equiv 1 \pmod{3}}} V \Big ( \frac{N(\pi \gamma^2) N(a)}{A} \Big ) \overline{\Big ( \frac{a}{b_1} \Big )_{3}} \Big ( \frac{a}{b_2} \Big )_{3}.
 $$
If $(b_1,b_2)=d$, then recall that \eqref{gsimp} tells us that
\begin{equation*}
 \tilde{g}(b_1) \overline{\tilde{g}(b_2)}=\tilde{g}(b_1 / d) \overline{\tilde{g}(b_2 / d)} \overline{ \Big(\frac{d}{b_1/d} \Big)_3}  \Big( \frac{d}{b_2/d} \Big)_3.
 \end{equation*}
Thus an application of Poisson summation (in the form of Corollary \ref{corpois}) on the $a$ sum shows that
 \begin{align} \label{diag2}
\mathscr{D}&=\frac{4\pi A}{9 \sqrt{3} N(\pi \gamma^2)}  
\sum_{\substack{d \in \mathbb{Z}[\omega] \\ d \equiv 1 \pmod{3}  \\ (d,\pi \gamma) = 1 }} \sum_{\substack{b_1, b_2 \in \mathbb{Z}[\omega] \\  (b_1,b_2)=d }} 
\frac{\beta_{b_1} \overline{\beta_{b_2}}}{\sqrt{N(b_1 b_2)}} \nonumber  \\
& \times \sum_{k \in \mathbb{Z}[\omega]} \widetilde{c_d}(k) \Big ( \frac{d \pi^2 \gamma k}{b_1/d} \Big )_{3} \overline{\Big ( \frac{d \pi^2 \gamma k}{b_2/d} \Big )_{3}} 
 \ddot{V} \Big ( \frac{k d \sqrt{A}}{b_1 b_2 N(\pi \gamma^2)^{1/2}} \Big ).  
 \end{align}
 For a given $d,\pi,\gamma \in \mathbb{Z}[\omega]$ 
 in \eqref{diag2}, 
 we split the $k$ sum into two subsums: 
 \begin{itemize}
 \item  $k \in \mathbb{Z}[\omega]$ such that $d \pi^2 \gamma k=\cube$;
 \item $k \in \mathbb{Z}[\omega]$ such that $d \pi^2 \gamma k \neq \cube$.
 \end{itemize}
 Denote the contributions to $\mathscr{D}$ from each of these two cases
 by $\mathscr{D}_{1}$ and $\mathscr{D}_2$ respectively. 
 Thus $\mathscr{D}=\mathscr{D}_1+\mathscr{D}_2$.
 
Consider $\mathscr{D}_1$. Since $\mu^2(d \pi \gamma)=1$,
we deduce that $d \pi^2 \gamma k=\cube$ if and only if $k=(d \gamma)^2 \pi H$ for some $H \in \mathbb{Z}[\omega]$ with $H=\cube$. 
Observe that \eqref{tildec} and Lemma \ref{ramstand} imply that
$\widetilde{c_d} \big((d \gamma)^2 \pi H \big)=\check{e} \big({-\frac{(d \gamma)^2 \pi H}{3 \lambda}} \big) \varphi(d)$.
Thus
\begin{align} \label{diagintermed}
\mathscr{D}_1&=\frac{4\pi A}{9 \sqrt{3} N(\pi \gamma^2)}  
\sum_{\substack{d \in \mathbb{Z}[\omega] \\ d \equiv 1 \pmod{3}  \\ (d,\pi \gamma) = 1 }} \varphi(d) \sum_{\substack{b_1, b_2 \in \mathbb{Z}[\omega] \\  (b_1,b_2)=d \\ (b_1 b_2/d^2,
\pi \gamma)=1}} 
\frac{\beta_{b_1} \overline{\beta_{b_2}}}{\sqrt{N(b_1 b_2)}} \nonumber  \\
& \times \sum_{\substack{ H \in \mathbb{Z}[\omega]  \\ H=\tinycube  \\ (H,b_1 b_2/d^2)=1 }} \check{e} \Big({-\frac{(d \gamma)^2 \pi H}{3 \lambda}} \Big) 
\ddot{V} \Big( \frac{ d^3 H \sqrt{A}}{b_1 b_2} \Big ).
\end{align}
We further write
$\mathscr{D}_1=\mathscr{D}^{\star}_1+\mathscr{D}^{\diamond}_1$,
where $\mathscr{D}^{\star}_1$ denotes the sum in \eqref{diagintermed}
restricted to $d=1$, and $\mathscr{D}^{\diamond}_1$ denotes the 
sum in \eqref{diagintermed} restricted to $d \neq 1$.
The support of $\boldsymbol{\beta}$ guarantees that 
$d \neq 1$ implies that $N(d) > w$. Thus by
Lemma \ref{ddotdecay} we have 
 \begin{align} \label{D1w}
\mathscr{D}_1^{\diamond} & \ll \frac{A}{N(\pi \gamma^2)} \sum_{\substack{d \in \mathbb{Z}[\omega] \\ (d,\pi \gamma) = 1 \\ d \equiv 1 \pmod{3} \\ N(d)>w }} \varphi(d)
\sum_{\substack{b_1, b_2 \in \mathbb{Z}[\omega] \nonumber \\  (b_1, b_2)=d \\ (b_1 b_2/d^2,\pi \gamma)=1 }} \frac{|\beta_{b_1} \beta_{b_2}|}{\sqrt{ N(b_1) N(b_2)}}  \Big( \frac{N(b_1 b_2)^{1/3}}{N(d) A^{1/3}}+1   \Big) \\
& \ll_{\varepsilon} \frac{A^{2/3} B^{5/3}}{w N(\pi \gamma^2)}+\frac{AB X^{\varepsilon} }{N(\pi \gamma^2)} \ll_{\varepsilon} \frac{A^{2/3} B^{5/3}}{w N(\pi \gamma^2)}+\frac{X^{1+\varepsilon} }{N(\pi \gamma^2)}.
\end{align}

We now consider $\mathscr{D}_1^{\star}$.
We write $H=h^3$ with $0 \neq h \in \mathbb{Z}[\omega]$
($h$ is necessarily non-zero in this case).
We have
\begin{equation} \label{cubeexp}
\check{e} \Big ({ - \frac{\gamma^2 \pi h^3}{3 \lambda}}\Big ) = \check{e} \Big ({- \frac{h^3}{3 \lambda }} \Big ) = 1.
\end{equation}
This can be seen by writing $h=\zeta \lambda^{i} u$ with $u \equiv 1 \pmod{3}$, $\zeta \in \{\pm 1, \pm \omega, \pm \omega^2\}$,
 and $i \in \mathbb{Z}_{\geq 0}$. Then the last equality in \eqref{cubeexp} follows from
  $$
 \check{e} \Big ({- \frac{(\zeta \lambda^i u)^3}{3 \lambda}} \Big ) = \check{e} \Big ( \pm \frac{\lambda^{3i}}{3 \lambda} \Big ) = 1.
 $$
Thus
\begin{equation} \label{D1starintermed}
\mathscr{D}_1^{\star}=\frac{A}{N(\pi \gamma^2)} \sum_{\substack{b_1, b_2 \in \mathbb{Z}[\omega] \\ (b_1, b_2) = 1 \\ (b_1 b_2, \pi \gamma) = 1 }} \frac{\beta_{b_1} \overline{\beta_{b_2}}}{\sqrt{N(b_1) N(b_2)}} \cdot \frac{4 \pi }{3^3 \sqrt{3}} \sum_{\substack{h \in \mathbb{Z}[\omega] \\ (h, b_1 b_2) = 1}} \ddot{V} \Big ( \frac{h^3 \sqrt{A}}{b_1 b_2} \Big ).
\end{equation}
Note that the extra factor of $1/3$ in the above display accounts for the fact that $(\omega^{i} h)^3 = h^3$ for $i \in \{0,1,2\}$
and $0 \neq h \in \mathbb{Z}[\omega]$. We remove the condition $(h, b_1 b_2) = 1$
at negligible cost since $\boldsymbol{\beta}$ is supported on $w$-rough squarefree integers in $\mathbb{Z}[\omega]$
with $w>(\log X)^{10}$. 
Thus 
\begin{align} \label{D1star}
\mathscr{D}_1^{\star}&=\frac{A}{N(\pi \gamma^2)} \sum_{\substack{b_1, b_2 \in \mathbb{Z}[\omega] \\ (b_1, b_2) = 1 \\ (b_1 b_2, \pi \gamma) = 1 }} \frac{\beta_{b_1} \overline{\beta_{b_2}}}{\sqrt{N(b_1) N(b_2)}} \cdot \frac{4 \pi }{3^3 \sqrt{3}} \sum_{h \in \mathbb{Z}[\omega] } \ddot{V} \Big ( \frac{h^3 \sqrt{A}}{b_1 b_2} \Big ) \nonumber \\
&+ O \Big( \frac{A^{2/3} B^{5/3}}{N(\pi \gamma^2) w^{9/10} } \Big)  +
O \Big( \frac{AB}{N(\pi \gamma^2) w^{9/10} } \Big).
\end{align}
Observe that $\ddot{V}(u)=\ddot{V}(|u|)$ is a Schwarz function by Lemma \ref{ddotdecay}.
Application of Poisson summation (in the form of Lemma \ref{pois1})
to the sum over $h \in \mathbb{Z}[\omega]$ yields
\begin{align} \label{eq:posd}
\frac{4\pi}{3^3 \sqrt{3}} & \sum_{\substack{h \in \mathbb{Z}[\omega]}} \ddot{V} \Big ( \frac{h^3 \sqrt{A}}{b_1 b_2} \Big ) = \frac{8 \pi}{3^4} \sum_{\substack{m \in \mathbb{Z}[\omega]}} \int_{\mathbb{R}^2} \ddot{V} \Big ( \frac{(x + i y)^3 \sqrt{A}}{b_1 b_2} \Big ) \check{e} \Big ( \frac{m (x + i y)}{\lambda} \Big ) dx dy.
\end{align}
We simplify the right side of \eqref{eq:posd}.
Recall that $\ddot{V}(u) = \ddot{V}(|u|)$ is radial. After changing $x + iy$ into polar coordinates $r e^{i \vartheta}$, the right side 
of \eqref{eq:posd} becomes
\begin{align}
\frac{8 \pi}{3^4} \sum_{\substack{m \in \mathbb{Z}[\omega]}} & \int_{0}^{2\pi}  \int_{0}^{\infty} \ddot{V} \Big ( \frac{r^3 \sqrt{A}}{\sqrt{N(b_1 b_2)}} \Big ) \check{e} \Big ( \frac{m r e^{i \vartheta}}{\lambda} \Big ) r dr d \vartheta \nonumber \\
& = \frac{8 \pi}{3^4}  \frac{N(b_1 b_2)^{1/3}}{A^{1/3}}  \sum_{\substack{m \in \mathbb{Z}[\omega]}} \int_{0}^{2\pi} \int_{0}^{\infty} \ddot{V} (r^3) \check{e} \Big ( \frac{m r e^{i \vartheta}}{\lambda} \frac{N(b_1 b_2)^{1/6}}{A^{1/6}} \Big ) r dr d\vartheta. \label{poisson2} 
\end{align}
For all $m \in \mathbb{Z}[\omega]$, Lemma \ref{ddotdecay} implies that 
\begin{align} \label{intpartstrunc}
\int_{0}^{2\pi} \int_{0}^{\infty} & \ddot{V} (r^3) \check{e} \Big ( \frac{m r e^{i \vartheta}}{\lambda} \frac{N(b_1 b_2)^{1/6}}{A^{1/6}} \Big ) r dr d\vartheta  \nonumber \\
&=\int_{0}^{2\pi} \int_{0}^{\infty} \ddot{V} (r^3) \delta_{[0,X^{\varepsilon}]}(r) \check{e} \Big ( \frac{m r e^{i \vartheta}}{\lambda} \frac{N(b_1 b_2)^{1/6}}{A^{1/6}} \Big ) r dr d\vartheta
+O_{\varepsilon}(X^{-2000}).
\end{align}
For $0 \neq m \in \mathbb{Z}[\omega]$,
repeated integration by parts on the right side of \eqref{intpartstrunc} shows that 
\begin{equation} \label{intpartspois}
\int_{0}^{2\pi} \int_{0}^{\infty} \ddot{V} (r^3) \delta_{[0,X^{\varepsilon}]}(r) \check{e} \Big ( \frac{m r e^{i \vartheta}}{\lambda} \frac{N(b_1 b_2)^{1/6}}{A^{1/6}} \Big ) r dr d\vartheta
\ll_{\varepsilon} X^{\varepsilon} (B^2/A)^{-1000} N(m)^{-1000}.
\end{equation}
We substitute \eqref{intpartstrunc} and \eqref{intpartspois} into \eqref{poisson2}, and then sum trivially over $0 \neq m \in \mathbb{Z}[\omega]$.
Chasing the result through \eqref{eq:posd} gives
\begin{equation} \label{poisson3}
\frac{4\pi}{3^3 \sqrt{3}} \sum_{\substack{h \in \mathbb{Z}[\omega]}} \ddot{V} \Big ( \frac{h^3 \sqrt{A}}{b_1 b_2} \Big )
=\frac{16 \pi^2}{3^4} \frac{N(b_1 b_2)^{1/3}}{A^{1/3}} \int_{0}^{\infty} \ddot{V} (r^3) r dr + O_{\varepsilon} \big( X^{\varepsilon} (B^2/A)^{-1000} \big).
\end{equation}
We now evaluate the main term on the right side of \eqref{poisson3}.
We open $\ddot{V}$ using the definition \eqref{Vddot}, and find that the main term is 
\begin{equation} \label{diagmainterm}
\frac{16 \pi^2}{3^4} \frac{N(b_1 b_2)^{1/3}}{A^{1/3}} \int_{0}^{\infty} u V(u^2) \int_{0}^{\infty}  J_0 \Big ( \frac{4\pi r^3 u}{3 \sqrt{3}} \Big ) r dr du.
\end{equation}
For each fixed $u \in (0,\infty)$, we make the change of variable 
$w=4 \pi r^3 u/(3 \sqrt{3})$ in the $r$-integral. 
Thus \eqref{diagmainterm} becomes 
\begin{equation} \label{eq:close}
\frac{ (4 \pi)^{4/3}}{3^4} \frac{N(b_1 b_2)^{1/3}}{A^{1/3}} \int_{0}^{\infty} u^{1/3} V(u^2) du \int_{0}^{\infty} J_0(w) w^{-1/3} dw. 
\end{equation}
A change of variable shows that 
\begin{equation} \label{int1}
\int_{0}^{\infty} u^{s - 1} V(u^2) du = \frac{1}{2} \widetilde{V} \Big ( \frac{s}{2} \Big ), \quad \text{for} \quad s \in \mathbb{C},
\end{equation}
and \cite[(10.22.43)]{NIST:DLMF} implies that
\begin{equation} \label{int2}
\int_{0}^{\infty} w^{s - 1} J_0(w) dw = \frac{2^{s - 1} \Gamma(\frac{s}{2})}{\Gamma(1 - \frac{s}{2})}, \quad \text{for}  -1<\Re(s-1)<1/2.
\end{equation}
Using \eqref{int1}, \eqref{int2} and Euler's reflection formula \cite[(5.5.3)]{NIST:DLMF},
we see that \eqref{eq:close}
becomes
\begin{equation} \label{eq:close2}
 \frac{2\pi}{9 \sqrt{3}} \Big ( \frac{(2\pi)^{2/3}}{3\Gamma(\frac{2}{3})}\Big )^2 \widetilde{V} \Big ( \frac{2}{3} \Big )  \frac{N(b_1 b_2)^{1/3}}{A^{1/3}}. 
\end{equation}
After retracing \eqref{poisson3} \eqref{diagmainterm}, \eqref{eq:close} and \eqref{eq:close2},  we obtain 
\begin{equation} \label{close3}
\frac{4\pi}{3^3 \sqrt{3}} \sum_{\substack{h \in \mathbb{Z}[\omega]}} \ddot{V} \Big ( \frac{h^3 \sqrt{A}}{b_1 b_2} \Big )
=\frac{2\pi}{9 \sqrt{3}} \Big ( \frac{(2\pi)^{2/3}}{3\Gamma(\frac{2}{3})}\Big )^2 \widetilde{V} \Big ( \frac{2}{3} \Big )  \frac{N(b_1 b_2)^{1/3}}{A^{1/3}} 
+O_{\varepsilon} \big( X^{\varepsilon} (B^2/A)^{-1000} \big).
\end{equation}
Substitution of \eqref{close3} into \eqref{D1star} gives
\begin{align} \label{D1starclose}
 \mathscr{D}_1^{\star} &= \frac{2\pi}{9 \sqrt{3}}  \Big ( \frac{(2\pi)^{2/3}}{3 \Gamma(\tfrac 23)} \Big )^2 \widetilde{V}\Big ( \frac{2}{3} \Big ) \frac{A^{2/3}}{N(\pi \gamma^2)} \sum_{\substack{b_1, b_2 \in \mathbb{Z}[\omega] \\ (b_1, b_2) = 1 \\ (b_1 b_2, \pi \gamma) = 1}} \frac{\beta_{b_1} \overline{\beta_{b_2}}}{N(b_1)^{1/6} N(b_2)^{1/6}} \nonumber \\
&+ O \Big ( \frac{A^{2/3} B^{5/3}}{N(\pi \gamma^2) w^{9/10}} \Big )+
O \Big( \frac{AB}{N(\pi \gamma^2) w^{9/10} } \Big)+O_{\varepsilon} \Big( \frac{X^{1+\varepsilon}}{N(\pi \gamma^2)} (B^2/A)^{-1000}  \Big). 
\end{align}
Using $w$-roughness of the support of $\boldsymbol{\beta}$,
 we drop the conditions $(b_1, b_2) = 1$ and $(b_1 b_2,\pi \gamma)=1$ at the expense of the error term of the same order
 of magnitude of that occurring in \eqref{D1starclose}. 
 After recalling that $\mathscr{D}_1=\mathscr{D}_1^{\star}+\mathscr{D}_1^{\diamond}$, \eqref{D1w}, and \eqref{D1starclose}, we 
 obtain
  \begin{align} \label{D1final}
 \mathscr{D}_1&= \frac{2\pi}{9 \sqrt{3}}  \Big ( \frac{(2\pi)^{2/3}}{3 \Gamma(\tfrac 23)} \Big )^2 \widetilde{V}\Big ( \frac{2}{3} \Big ) \frac{A^{2/3}}{N(\pi \gamma^2)} 
 \Big | \sum_{b \in \mathbb{Z}[\omega] } \frac{\beta_{b}}{N(b)^{1/6}} \Big |^2 \nonumber \\
 &+ O \Big ( \frac{A^{2/3} B^{5/3}}{N(\pi \gamma^2) w^{9/10}} \Big)+O_{\varepsilon} \Big( \frac{X^{1+\varepsilon}}{N(\pi \gamma^2)} \big(1+(B^2/A)^{-1000} \big) \Big ).
\end{align}
 
We now consider $\mathscr{D}_2$,
\begin{align} \label{diagnonprinc}
\mathscr{D}_2&=\frac{4\pi A}{9 \sqrt{3} N(\pi \gamma^2)}  
\sum_{\substack{d \in \mathbb{Z}[\omega] \\ d \equiv 1 \pmod{3} \\ (d,\pi \gamma) = 1 }} \sum_{\substack{b_1, b_2 \in \mathbb{Z}[\omega] \\  (b_1,b_2)=d }} 
\frac{\beta_{b_1} \overline{\beta_{b_2}}}{\sqrt{N(b_1 b_2)}} \nonumber  \\
& \times 
\sum_{\substack{ k \in \mathbb{Z}[\omega] \\ d \pi^2 \gamma k \neq \tinycube }} \widetilde{c_d}(k) \Big ( \frac{d \pi^2 \gamma k}{b_1/d} \Big )_{3} \overline{\Big ( \frac{d \pi^2 \gamma k}{b_2/d} \Big )_{3}} 
 \ddot{V} \Big ( \frac{k d \sqrt{A}}{b_1 b_2 N(\pi \gamma^2)^{1/2}} \Big ).  
 \end{align}
We rescale $b_1 \rightarrow db_1$ and $b_2 \rightarrow db_2$ and use Lemma \ref{ddotdecay} in \eqref{diagnonprinc}.
We obtain
\begin{align} \label{diagnonprinc2}
\mathscr{D}_2&=\frac{4\pi A}{9 \sqrt{3} N(\pi \gamma^2)} \sum_{\substack{d \in \mathbb{Z}[\omega] \\ d \equiv 1 \pmod{3} \\ (d,\pi \gamma)=1}}  \frac{1}{N(d)} \sum_{\substack{b_1, b_2 \in \mathbb{Z}[\omega] 
\\ (b_1 b_2, \pi \gamma) = 1 \\ (b_1, b_2) = 1}} \frac{\beta_{d b_1} \overline{\beta_{d b_2}}}{\sqrt{N(b_1) N(b_2)}} \nonumber  \\
& \times \sum_{\substack{ k \in \mathbb{Z}[\omega] \\ d \pi^2 \gamma k \neq \tinycube \\ N(k) \ll \mathcal{Z}}} \widetilde{c_d}(k)
\Big ( \frac{d \pi^2 \gamma k}{b_1} \Big )_{3} \overline{\Big ( \frac{d \pi^2 \gamma k}{b_2} \Big )_{3}} \ddot{V} \Big ( \frac{k \sqrt{A}}{d b_1 b_2 N(\pi \gamma^2)^{1/2}} \Big ) \nonumber \\
&+O_{\varepsilon}( (N(\pi) \Delta X)^{-1000}),
\end{align}
where 
\begin{equation*}
\mathcal{Z}:=(\Delta N(\pi) X)^{\varepsilon} \Big(1+ \frac{B^2 N(\pi) \Delta^2}{N(d)A} \Big).
\end{equation*}
We M\"{o}bius invert $(b_1,b_2)=1$ and
separate variables by opening $\ddot{V}$ using \eqref{ddotmellin} and \eqref{ddottrunc}.
Rearranging the absolutely convergent finite sums and integrals by Fubini's theorem gives
\begin{align*} 
\mathscr{D}_2&=\frac{2 A (-1)^{L}}{9 i \sqrt{3} N(\pi \gamma^2)} \int_{0}^{\infty} \int_{-\varepsilon-i X^{\varepsilon}}^{-\varepsilon+i X^{\varepsilon}} V^{(L)}(r^2) r^{2 L+1} 
\frac{\Gamma(-s)}{\Gamma(L+s+1)} \Big( \frac{2 \pi \sqrt{A}}{3 \sqrt{3} N(\pi)^{1/2} N(\gamma)} \Big)^{2s}  \\
& \times \sum_{\substack{f \in \mathbb{Z}[\omega] \\ f \equiv 1 \pmod{3} \\ (f,\pi \gamma)=1 }} \frac{\mu(f)}{N(f)^{1+2s}}
\sum_{\substack{d \in \mathbb{Z}[\omega] \\ d \equiv 1 \pmod{3} \\ (d,\pi \gamma)=1}}  \frac{1}{N(d)^{1+s}}  
\sum_{\substack{ k \in \mathbb{Z}[\omega] \\ d \pi^2 \gamma k \neq \tinycube \\ N(k) \ll \mathcal{Z}}} \widetilde{c_d}(k) \Big ( \frac{d \pi^2 \gamma k}{f} \Big )_{3} \overline{\Big ( \frac{d \pi^2 \gamma k}{f} \Big )_{3}}  N(k)^{s}  \\
& \times \Big( \sum_{\substack{b_1 \in \mathbb{Z}[\omega] \\ (b_1, \pi \gamma) = 1}} \frac{\beta_{fd b_1}}{\sqrt{N(b_1)}} \Big ( \frac{d \pi^2 \gamma k}{b_1} \Big )_{3} N(b_1)^{-s} \Big)
 \Big( \sum_{\substack{b_2 \in \mathbb{Z}[\omega] \\ (b_2, \pi \gamma) = 1}} \frac{\overline{\beta_{fd b_2}}}{\sqrt{N(b_2)}} \overline{\Big ( \frac{d \pi^2 \gamma k}{b_2} \Big )_{3}} N(b_2)^{-s} \Big)
 ds dr   \\
&+O_{\varepsilon}((N(\pi) \Delta X)^{-1000}),
\end{align*}
for any fixed $L \in \mathbb{Z}_{\geq 1}$. We use Axiom 4 of Definition \ref{seqdef} to estimate the sum over $b_1$ and $b_2$, and then estimate the
remaining sums trivially using Lemma \ref{ramstand}. 
We obtain  
\begin{equation} \label{D2final}
\mathscr{D}_2 \ll_{\varepsilon} (\Delta N(\pi) X)^{\varepsilon} \Big( \frac{A B^{2\eta}}{N(\pi \gamma^2)}+\frac{B^{2 + 2 \eta} N(\pi) \Delta^2}{N(\pi \gamma^2)}  \Big). 
\end{equation}
Since $\eta \leq \tfrac 14$ we have $A B^{2\eta} \leq X$. 
After recalling \eqref{D1final}, \eqref{D2final} and the fact that $\mathscr{D}=\mathscr{D}_1+\mathscr{D}_2$,
we obtain \eqref{totalD}.

\subsubsection*{The cross terms $\mathscr{C}$}
Recall that \eqref{crossterm} records the cross term.
Observe that \eqref{twistmult} tells us that
 \begin{equation} \label{gaussident}
 \tilde{g}(b_1) \tilde{g}(a \pi)  \overline{\Big( \frac{a}{b_1} \Big)_3} \Big( \frac{\pi^2}{b_1} \Big)_3=\tilde{g}(a \pi b_1).
 \end{equation} 
Substituting \eqref{gaussident} into \eqref{crossterm} gives
\begin{align} \label{crossterm2}
  \mathscr{C} := - 2 \frac{(2\pi)^{2/3}}{3 \Gamma(\tfrac 23)} \delta_{\gamma=1} \text{Re} \Big ( & \sum_{\substack{b_1, b_2 \in \mathbb{Z}[\omega]}} \frac{\beta_{b_1}  \overline{\beta_{b_2}}}{N(b_2)^{1/6}} 
\sum_{\substack{a \in \mathbb{Z}[\omega] \\ a \equiv 1 \pmod{3} \\ (a \pi,b_1)=1 }} V \Big ( \frac{N(a) N(\pi)}{A} \Big )  \frac{\tilde{g}(a \pi b_1)}{N(a \pi)^{1/6}} \Big ).
 \end{align}
 
 We now evaluate the sum over $a \in \mathbb{Z}[\omega]$ in \eqref{crossterm2} using our asymptotic formula for type-I sums 
in Proposition \ref{prop:typeI} (for level $\pi b_1$). Thus there exists $\rho(\varepsilon) \in (0,\frac{1}{10000})$ such that
\begin{align} \label{b1intermed}
  \sum_{\substack{b_1 \in \mathbb{Z}[\omega]}} & \beta_{b_1} \sum_{\substack{a \in \mathbb{Z}[\omega] \\ a \equiv 1 \pmod{3}}} V \Big ( \frac{N(a) N(\pi) }{A} \Big ) \frac{\tilde{g}(a \pi b_1)}{N(a \pi)^{1/6}} \nonumber \\ & = \widetilde{V} \Big ( \frac{2}{3} \Big ) \frac{(2\pi)^{5/3} A^{2/3}}{3^{7/2} \Gamma(\tfrac 23) N(\pi)^{5/6}}
  \sum_{\substack{b_1 \in \mathbb{Z}[\omega]}} \frac{\beta_{b_1} \varphi(\pi b_1)}{\zeta_{\mathbb{Q}(\omega)}(2;\mathbf{1}_{\pi b_1}) N(\pi b_1)^{7/6}} \nonumber \\
  &+ O_{\varepsilon} \Big(\frac{A^{2/3-\rho(\varepsilon)} B^{5/6-\rho(\varepsilon)}}{N(\pi)} \Big)
  +O_{\varepsilon} \big( X^{\varepsilon} N(\pi)^{1/2} B^{19/12} A^{-1/12} \big).
\end{align}
We now use the fact that $\boldsymbol{\beta}$ is supported on $w$-rough squarefree elements of $\mathbb{Z}[\omega]$
that are congruent to $1$ modulo $3$. We have
\begin{equation*}
\log \Big( \frac{\varphi(\pi b_1)}{N(\pi b_1)}\Big)=\sum_{\substack{\varpi \mid \pi b_1 \\ \varpi \text{ prime} }} \log \big(1-N(\varpi)^{-1} \big) 
=-\sum_{\substack{\varpi \mid \pi b_1 \\ \varpi \text{ prime} }   } \sum_{L=1}^{\infty}  \frac{1}{L N(\varpi)^{L}}=O \Big(\frac{1}{w^{9/10}}+\frac{\delta_{\pi \neq 1}}{N(\pi)} \Big).
\end{equation*}
Thus 
\begin{equation} \label{varphib1}
\frac{\varphi(\pi b_1)}{N(\pi b_1)}=1+O \Big(\frac{1}{w^{9/10}}+\frac{\delta_{\pi \neq 1}}{N(\pi)} \Big).
\end{equation}
Similarly, we also have 
\begin{equation} \label{zeta2approx}
\frac{1}{\zeta_{\mathbb{Q}(\omega)}(2;\mathbf{1}_{\pi b_1})}=\frac{1}{\zeta_{\mathbb{Q}(\omega)}(2;\mathbf{1}_{\pi})}+O \Big( \frac{1}{w^{19/10}}     \Big).
\end{equation}
Insertion of \eqref{varphib1} and \eqref{zeta2approx} into \eqref{b1intermed} gives
\begin{align} \label{b1intermed2}
 & \sum_{\substack{b_1 \in \mathbb{Z}[\omega]}} \beta_{b_1} \sum_{\substack{a \in \mathbb{Z}[\omega] \\ a \equiv 1 \pmod{3}}} V \Big ( \frac{N(a) N(\pi)}{A} \Big ) \frac{\tilde{g}(a \pi b_1)}{N(a \pi)^{1/6}} \nonumber \\ & = \widetilde{V} \Big ( \frac{2}{3} \Big ) \frac{(2\pi)^{5/3} A^{2/3}}{3^{7/2} \Gamma(\tfrac 23) \zeta_{\mathbb{Q}(\omega)}(2;\mathbf{1}_{\pi}) N(\pi)} \sum_{\substack{b_1 \in \mathbb{Z}[\omega]}} \frac{\beta_{b_1}}{N(b_1)^{1/6}} 
 +O_{\varepsilon} \Big( \frac{A^{2/3-\rho(\varepsilon)} B^{5/6-\rho(\varepsilon)}}{N(\pi)} \Big) \nonumber  \\
 &+O_{\varepsilon}( X^{\varepsilon} N(\pi)^{1/2} B^{19/12} A^{-1/12})+O \Big( \frac{A^{2/3} B^{5/6}}{N(\pi) } \Big( \frac{1}{w^{9/10}}+\frac{\delta_{\pi \neq 1}}{N(\pi)} \Big)    \Big).
\end{align}
Insertion of \eqref{b1intermed2} into \eqref{crossterm2} gives \eqref{crosstermfinal}.

\subsubsection*{The trivial term $\mathscr{T}$}
Recall that \eqref{trivialterm} records
\begin{equation} \label{trivialterm2}
\mathscr{T}=\Big ( \frac{(2\pi)^{2/3}}{3 \Gamma(\tfrac 23)} \Big )^2 \delta_{\gamma=1} \sum_{\substack{a \in \mathbb{Z}[\omega] \\ a \equiv 1 \pmod{3}}} V \Big ( \frac{N(a) N(\pi) }{A} \Big ) \frac{\mu^2(a \pi)}{N(a \pi)^{1/3}} \Big | \sum_{\substack{b \in \mathbb{Z}[\omega] }} \frac{\beta_{b}}{N(b)^{1/6}} \Big |^2.  
\end{equation}
Mellin inversion of the smooth function, the Class number formula \cite[Chapter VIII, \S 2, Theorem~5]{Lang} and subsequent contour shift to the right of the $1/6$-line (in the $s$-variable) gives 
\begin{align} \label{contourshift}
\sum_{\substack{a \in \mathbb{Z}[\omega] \\ a \equiv 1 \pmod{3}}}  & V \Big ( \frac{N(a) N(\pi)}{A} \Big ) \frac{\mu^2(a \pi)}{N(a \pi)^{1/3}} \nonumber \\
&=\frac{1}{2 \pi i} \int_{(2)} \widetilde{V}(s) \frac{\zeta_{\mathbb{Q}(\omega)}(s+1/3;\mathbf{1}_{\pi})}{\zeta_{\mathbb{Q}(\omega)}(2s+2/3;\mathbf{1}_{\pi})} \frac{A^s}{N(\pi)^{s+1/3}} ds  \nonumber    \\
& = \widetilde{V} \Big ( \frac{2}{3} \Big )( 1-\delta_{\pi \neq 1} \cdot N(\pi)^{-1}) \frac{2\pi A^{2/3}}{9 \sqrt{3} \zeta_{\mathbb{Q}(\omega)}(2;\mathbf{1}_{\pi}) N(\pi)}+O_{\varepsilon} \Big(\frac{A^{1/6+\varepsilon}}{N(\pi)^{1/2+\varepsilon}} \Big) \nonumber  \\
& = \widetilde{V} \Big ( \frac{2}{3} \Big ) \frac{2\pi A^{2/3}}{9 \sqrt{3} \zeta_{\mathbb{Q}(\omega)}(2;\textbf{1}_{\pi}) N(\pi)}+ \delta_{\pi \neq 1} O \Big(\frac{A^{2/3}}{N(\pi)^2} \Big)+O_{\varepsilon} \Big(\frac{A^{1/6+\varepsilon}}{N(\pi)^{1/2+\varepsilon}} \Big).
 \end{align}
Insertion of \eqref{contourshift} into \eqref{trivialterm2} gives \eqref{trivialfinal}.
\end{proof}

\begin{lemma} \label{squarefreesieve}
  Given $y \geq 1$, there exists coefficients $(\lambda_d)_{d \in \mathbb{Z}[\omega]}$ such that
  \begin{enumerate}
  \item \label{eq:first1} $\lambda_1 = 1$ and $|\lambda_d| \ll_{\varepsilon} N(d)^{\varepsilon}$ for all $d \in \mathbb{Z}[\omega]$ and all $\varepsilon > 0$; 
  \item \label{eq:second2} $\lambda_d = 0$ if $N(d) > y^2$ or $d \not \equiv 1 \pmod{3}$;
  \item \label{eq:third3} For all $n \in \mathbb{Z}[\omega]$ we have
    \begin{equation} \label{appsieve2}
      \mu^2(n) \leq \sum_{\substack{ d^2 | n}} \lambda_d;
    \end{equation}
  \item \label{eq:fourth4} For any $\varepsilon > 0$ and $\pi \in \mathbb{Z}[\omega]$ a prime $\pi \equiv 1 \pmod{3}$ (or 1) we have
\begin{equation} \label{sievewtbound2}
\sum_{\substack{ d \in \mathbb{Z}[\omega] \\ (d,\pi)=1 }} \frac{\lambda_d}{N(d)^2}  = \frac{1}{\zeta_{\mathbb{Q}(\omega)}(2;\mathbf{1}_{\pi} )} + O_{\varepsilon} ( y^{-1/2 + \varepsilon} ).
\end{equation}
  \end{enumerate}
\end{lemma}

\begin{proof}
  Given $d \equiv 1 \pmod{3}$, let
  $$
  \lambda_{d}:= \sum_{\substack{N(e), N(f) \leq y \\ e,f \equiv 1 \pmod{3} \\ d=[e,f]  }} \mu(e) \mu(f)
  $$
Properties \eqref{eq:first1} and \eqref{eq:second2} are immediate from the definition. Property \eqref{eq:third3} follows from 
\begin{equation*}
\sum_{d^2 \mid n} \lambda_d= \Big( \sum_{\substack{ N(e) \leq y \\ e^2 \mid n }} \mu(e) \Big)^2.
\end{equation*}
It remains to check property \eqref{eq:fourth4}. We have
  \begin{align}
  \sum_{\substack{ d \in \mathbb{Z}[\omega] \\ (d,\pi)=1 }} \frac{\lambda_d}{N(d)^2} & = \sum_{\substack{N(e), N(f) \leq y \\ e,f \equiv 1 \pmod{3} \\ (ef,\pi)=1 }} \frac{\mu(e)\mu(f)}{N([e,f])^2} = \sum_{\substack{ e,f \equiv 1 \pmod{3} \\ (ef,\pi)=1 }} \frac{\mu(e) \mu(f)}{N([e,f])^2} + O( y^{-1/2+\varepsilon}) \label{trivmaint}.
  \end{align} 
  The main term in \eqref{trivmaint} is equal to
  $$
  \prod_{\substack{\varpi \equiv 1 \pmod{3} \\ \varpi \neq \pi }} \Big ( 1 + \frac{1}{N(\varpi)^2}( -2 + 1) \Big ) = \prod_{\substack{ \varpi \equiv 1 \pmod{3} \\ \varpi \neq \pi }} \Big ( 1 - \frac{1}{N(\varpi)^2} \Big ) = \frac{1}{\zeta_{\mathbb{Q}(\omega)}(2;\mathbf{1}_{\pi} )}, 
  $$
as required.
\end{proof}

\begin{proof}[Proof of Proposition \ref{prop:main}]
In the first display of the statement of Proposition
\ref{prop:main} we use property (\ref{eq:third3}) in Lemma \ref{squarefreesieve} with $y=X^{\varepsilon}$ with $\varepsilon>0$ small and fixed,
and inclusion-exclusion on the condition $(a,b)=1$.
We see that the first display in Proposition \ref{prop:main} is 
\begin{align} \label{sieveinitial}
& \leq \sum_{\substack{a \in \mathbb{Z}[\omega] \\ a \equiv 1 \pmod{3} \\ \pi \mid a}} V \Big ( \frac{N(a)}{A} \Big ) \Big | \Big (\sum_{\substack{b \in \mathbb{Z}[\omega] \\ (b, \gamma a)= 1}} \beta_{b} \tilde{g}(b) \overline{\Big ( \frac{b}{a} \Big )_{3}} - \frac{(2\pi)^{2/3}}{3 \Gamma(\tfrac 23)} \frac{\overline{\tilde{g}(a)}}{N(a)^{1/6}} \sum_{\substack{ b \in \mathbb{Z}[\omega] }} \frac{\beta_{b}}{N(b)^{1/6}} \Big ) \nonumber \\
& \qquad \qquad \qquad \qquad \qquad \qquad + \Big ( \frac{(2\pi)^{2/3}}{3 \Gamma(\tfrac 23)} \frac{\overline{ \tilde{g}(a)}}{N(a)^{1/6}}  \sum_{\substack{ b \in \mathbb{Z}[\omega] \\ (b,a) \neq 1 } } \frac{\beta_b}{N(b)^{1/6}} \Big ) \Big |^2 
\Big( \sum_{\substack{ \gamma^2 | a}} \lambda_{\gamma} \Big).
\end{align}
Since $V$ is non-negative and $\sum_{\substack{ \gamma^2 | a}} \lambda_{\gamma} \geq 0$, we can apply the parallelogram inequality \eqref{trivineq}
to \eqref{sieveinitial}. 
We see that the right side of \eqref{sieveinitial} is
\begin{align} \label{sieveinitial2}
& \leq 2 \sum_{\substack{a \in \mathbb{Z}[\omega] \\ a \equiv 1 \pmod{3} \\ \pi \mid a}}  V \Big ( \frac{N(a)}{A} \Big )  \Big( \Big |\sum_{\substack{b \in \mathbb{Z}[\omega] \\ (b, \gamma a)= 1}} \beta_{b} \tilde{g}(b) \overline{\Big ( \frac{b}{a} \Big )_{3}} - \frac{(2\pi)^{2/3}}{3 \Gamma(\tfrac 23)} \frac{\overline{\tilde{g}(a)}}{N(a)^{1/6}} \sum_{\substack{ b \in \mathbb{Z}[\omega] }} \frac{\beta_{b}}{N(b)^{1/6}} \Big |^2 \nonumber  \\
&\qquad \qquad \qquad + \frac{(2\pi)^{4/3}}{9 \Gamma(\tfrac 23)^2} \sum_{\substack{a \in \mathbb{Z}[\omega] \\ a \equiv 1 \pmod{3} \\ \pi \mid a}} \frac{\mu^2(a)}{N(a)^{1/3}} V \Big( \frac{N(a)}{A} \Big) \Big |  \sum_{\substack{ b \in \mathbb{Z}[\omega] \\ (b,a) \neq 1 } } \frac{\beta_b}{N(b)^{1/6}}  \Big |^2 \Big) 
\Big( \sum_{\substack{ \gamma^2 | a}} \lambda_{\gamma} \Big). 
\end{align}
We interchange the order of summation and use \eqref{cuberel} in \eqref{sieveinitial2}. We see that the right side of \eqref{sieveinitial2} is equal to
\begin{align}  
& 2 \Big( \sum_{\gamma \equiv 1 \pmod{3}} \lambda_{\gamma}  \nonumber \\
& \times \Big( \sum_{\substack{a \in \mathbb{Z}[\omega] \\ a \equiv 1 \pmod{3} \\ \pi \mid a,  \gamma^2 \mid a}}  V \Big ( \frac{N(a)}{A} \Big )  
\Big |\sum_{\substack{b \in \mathbb{Z}[\omega] \\ (b, \gamma a)= 1}} \beta_{b} \tilde{g}(b) \overline{\Big ( \frac{b}{a} \Big )_{3}} - \frac{(2\pi)^{2/3}}{3 \Gamma(\tfrac 23)} \frac{\delta_{\gamma=1} \overline{\tilde{g}(a)}}{N(a)^{1/6}} \sum_{\substack{ b \in \mathbb{Z}[\omega] }} \frac{\beta_{b}}{N(b)^{1/6}} \Big |^2  \label{sievenontriv} \\
&\qquad \qquad + \delta_{\gamma=1} \frac{(2\pi)^{4/3}}{9 \Gamma(\tfrac 23)^2} \sum_{\substack{a \in \mathbb{Z}[\omega] \\ a \equiv 1 \pmod{3} \\ \pi \mid a}} \frac{\mu^2(a)}{N(a)^{1/3}} V \Big( \frac{N(a)}{A} \Big) \Big |  \sum_{\substack{ b \in \mathbb{Z}[\omega] \\ (b,a) \neq 1 } } \frac{\beta_b}{N(b)^{1/6}}  \Big |^2 \Big) \Big).  \label{sievetriv}
\end{align}
Observe that \eqref{sievetriv} matches \eqref{trivred2}, and is estimated by \eqref{trivO}. 
If $(\pi,\gamma)=1$, then the divisibility condition in \eqref{sievenontriv} becomes $\pi \gamma^2 \mid a$.
Thus \eqref{sievenontriv} in this case matches \eqref{trivred1}, and is estimated asymptotically by
by $\mathscr{D}+\mathscr{C}+\mathscr{T}$, where $\mathscr{D}$, $\mathscr{C}$
and $\mathscr{T}$ have asymptotic expressions given by  \eqref{totalD}, \eqref{crosstermfinal},
and \eqref{trivialfinal} respectively.
Summing the asymptotic expression for $\mathscr{D}+\mathscr{C}+\mathscr{T}$ over $\gamma$ with $(\gamma,\pi)=1$
(with sieve weights $\lambda_{\gamma}$)
using properties (\ref{eq:first1}), (\ref{eq:second2}), and (\ref{eq:fourth4}) of Lemma \ref{squarefreesieve} yields the bound stated 
in Proposition \ref{prop:main}. If $(\pi,\gamma) \neq 1$, then the divisibility condition in \eqref{sievenontriv} becomes $\gamma^2 \mid a$
with $\pi \mid \gamma$ and $\pi \neq 1$. Performing a similar computation to the previous case gives the result.
\end{proof}

\section{Broad Type II estimates} \label{broad2sec}

We prove the following type-II estimates for sequences in $\mathcal{C}_{\eta}(\cdot, w)$.  

\begin{prop} \label{prop:broad}
 Let $W$ be a smooth function compactly supported in $[1,2]$,
$0<\eta  \leq 1/4$, $A, B \geq 10$ and set $X := A B$. 
  Let $\boldsymbol{\alpha}$ be a sequence supported in $N(a) \in [A / 10, 10 A]$ with $a \equiv 1 \pmod{3}$ squarefree. Suppose that $w>(\log X)^{10}$, $\boldsymbol{\beta} = (\beta_b) \in \mathcal{C}_{\eta}(B, w)$, $\varepsilon \in (0,\frac{1}{10000})$ and $\pi \equiv 1 \pmod{3}$ a prime or $1$ satisfying $1 \leq N(\pi)<w$.
 Then there exists $\rho(\varepsilon) \in (0,\frac{1}{10000})$ such that
  \begin{align*}
 & \sum_{\substack{a,b \in \mathbb{Z}[\omega] \\ \pi \mid a }} \alpha_a \beta_b \tilde{g}(a b) W \Big ( \frac{N(ab)}{X} \Big )=\frac{(2\pi)^{2/3}}{3 \Gamma(\tfrac 23)} \sum_{\substack{a,b \in \mathbb{Z}[\omega] \\ \pi \mid a}} \frac{\alpha_a \beta_b \mu^2 (ab)}{N(a b)^{1/6}} W \Big ( \frac{N(ab)}{X} \Big )  \\
  &  + O_{\varepsilon} \Big ( \Big(\sum_{\substack{a \in \mathbb{Z}[\omega] \\ \pi \mid a }} |\alpha_a|^2  \Big)^{1/2} \cdot
  \Big ( \frac{A^{1/3} B^{5/6}}{N(\pi)^{1/2}} \Big(\frac{1}{w^{9/20}}+\frac{\delta_{\pi \neq 1}}{N(\pi)^{1/2}}  \Big)
  + \frac{A^{1/3-\rho(\varepsilon)/2} B^{5/6-\rho(\varepsilon)/2}}{N(\pi)^{1/2}}  \\
  &+\frac{A^{1/12+\varepsilon/2} B^{5/6}}{N(\pi)^{1/4+\varepsilon/2}} 
  + (X N(\pi))^{\varepsilon/2} \Big( N(\pi)^{1/4} B^{29/24} A^{-1/24} 
  + B^{1 + \eta}+\frac{X^{1/2}}{N(\pi)^{1/2}} \Big (1+(B^2/A)^{-500} \Big) \Big) \Big ) \Big ). 
  \end{align*}
\end{prop}
\begin{remark}
  Suppose $\pi=1$ say. Then for dense sequences $\boldsymbol{\alpha}, \boldsymbol{\beta}$ and given $\xi > 0$, the result is non-trivial in the range $X^{1/3+\xi} \leq B \leq X^{1/2-\xi}$,
  for some appropriate choice of $\varepsilon, \eta$ and $w$. 
 \end{remark}
\begin{proof}
Observe that \eqref{twistmult} gives
 \begin{align} 
 & \sum_{\substack{a,b \in \mathbb{Z}[\omega] \\ \pi \mid a }} \alpha_a \beta_b \Big ( \tilde{g}(a b) - \frac{(2\pi)^{2/3}}{3 \Gamma(\tfrac 23)} \frac{\mu^2(ab)}{N(ab)^{1/6}} \Big ) W \Big ( \frac{N(a b)}{X} \Big ) \nonumber \\ 
 &=\sum_{\substack{a,b \in \mathbb{Z}[\omega] \\ \pi \mid a \\ (a,b)=1 }} \alpha_a \beta_b \Big ( \tilde{g}(a b) - \frac{(2\pi)^{2/3}}{3 \Gamma(\tfrac 23)} \frac{1}{N(ab)^{1/6}} \Big ) W \Big ( \frac{N(a b)}{X} \Big )
 \nonumber \\
 &=\frac{1}{2\pi i} \int_{- i \infty}^{i \infty} \widetilde{W}(s) X^s \sum_{\substack{a \in \mathbb{Z}[\omega] \\ \pi \mid a }} \frac{\alpha_a \tilde{g}(a)}{N(a)^{s}} \Big ( \sum_{\substack{b \in \mathbb{Z}[\omega] \\ (b,a)=1 }} \frac{\beta_b \tilde{g}(b)}{N(b)^{s}} \overline{\Big ( \frac{a}{b} \Big )_{3}}   - \frac{(2\pi)^{2/3}}{3 \Gamma(\tfrac 23)} \frac{\overline{\tilde{g}(a)}}{N(a)^{1/6}} \sum_{\substack{b \in \mathbb{Z}[\omega] \\ (b,a)=1 }} \frac{\beta_b}{N(b)^{1/6 + s}} \Big ) ds. \label{keyexpr} 
\end{align}
Application of triangle inequality and then Cauchy-Schwarz to the $a$-sum shows that
\begin{align*}
|\eqref{keyexpr}|^2 \ll &  \Big( \sum_{\substack{a  \in \mathbb{Z}[\omega] \\ a \equiv 1 \pmod{3} \\ \pi \mid a  }} | \alpha_a|^2 \Big) \int_{- \infty}^{\infty} |\widetilde{W}(it)| \Big ( \sum_{\substack{a \in \mathbb{Z}[\omega] \\ \pi \mid a }} \mu^2(a) V \Big ( \frac{N(a)}{A} \Big ) \\ & \qquad \times \Big |  \sum_{\substack{b \in \mathbb{Z}[\omega] \\ (b,a)=1 }} \frac{\beta_b \tilde{g}(b)}{N(b)^{it}} \overline{\Big ( \frac{a}{b} \Big )_{3}} - \frac{(2\pi)^{2/3}}{3 \Gamma(\tfrac 23)} \frac{\overline{\tilde{g}(a)}}{N(a)^{1/6}} \sum_{\substack{b \in \mathbb{Z}[\omega] \\ (b,a)=1 }} \frac{\beta_b}{N(b)^{1/6 + it}} \Big |^2 \Big ) dt, 
 \end{align*}
 where $V:\mathbb{R} \rightarrow \mathbb{R}$ a smooth compactly supported function such that $V \geq \boldsymbol{1}_{[1/10,10]}$. 
 Using Proposition \ref{prop:main} gives the result.
\end{proof}
\section{Average Type-I estimates} \label{avgtype1sec}
In this section we prove an average Type-I estimate. This average Type-I estimate will be more directly useful to us than the pointwise Type-I
estimate proved in Section \ref{type1sec}. Recall that $\ell \in \mathbb{Z}$, $c \equiv 1 \pmod{3}$, and
$$
g_{\ell}(c):= \tilde{g}(c) \Big ( \frac{c}{|c|} \Big )^{\ell}. 
$$

\begin{prop} \label{prop:avgtypeI}
  Assume the Generalized Riemann Hypothesis for the Dedekind zeta function attached to $\mathbb{Q}(\omega)$
  twisted by
  Gr{\"o}{\ss}encharaktern.
  Let $\ell \in \mathbb{Z}$,
  $\varepsilon \in (0, \frac{1}{1000000})$, $ A > 1000$ be large and fixed, and $(\log X)^{A} \leq w \leq X^{\varepsilon}$.  
  Let $V, W$ be smooth functions, compactly supported in $[\tfrac 12, 4]$. Let $0<\eta<100/A$
 and $\boldsymbol{\alpha}=(\alpha_r) \in \mathcal{C}_{\eta}(X,w)$. Then uniformly for $C > X^{2/3 - \varepsilon}$ and $|\ell| \leq X^{1/100}$ we have 
  \begin{align} \label{eq:yt}
  \sum_{\substack{r \in \mathbb{Z}[\omega] \\ c \equiv 1 \pmod{3} \\ \varpi | c \implies N(\varpi) > w}}  \alpha_r \tilde{g}_{\ell}(c r) V \Big ( \frac{N(c)}{C} \Big ) W \Big ( \frac{N(c r)}{X} \Big ) &= \frac{(2\pi)^{2/3}}{3 \Gamma(\tfrac 23)} \sum_{\substack{r \in \mathbb{Z}[\omega] \\ c \equiv 1 \pmod{3} \\ \varpi | c \implies N(\varpi) > w}} \frac{\alpha_r \mu^2(c r) \big ( \frac{c r}{|c r|} \big )^{\ell}}{N(c r)^{1/6}} \nonumber \\
  & + O_{\varepsilon, A} \Big (\frac{X^{5/6} }{(\log X)^{\frac{9}{20} A-10}} \Big ).
  \end{align}
\end{prop}

\begin{proof}[Proof of Proposition \ref{prop:avgtypeI}] 
M\"obius inversion asserts that
\begin{equation} \label{roughmobius}
\delta_{\varpi | c \implies N(\varpi) > w} = \sum_{\substack{u | c \\ u \equiv 1 \pmod{3} \\ \varpi | u \implies N(\varpi) \leq w}} \mu(u). 
\end{equation}
Using \eqref{roughmobius}, we express the left side of \eqref{eq:yt} as
\begin{equation} \label{eq:yt2}
\sum_{k \geq 0} (-1)^{k} \sum_{N \text{ dyadic}} 
\mathscr{B}_{k,\ell}(N),
\end{equation}
where
\begin{align} \label{Bkldef}
\mathscr{B}_{k,\ell}&(N)
:=\sum_{\substack{r,u,n \in \mathbb{Z}[\omega] \\ u,n \equiv 1 \pmod{3}}} \alpha_r \tilde{g}_{\ell}(n r u)
\Big( \delta_{ \substack{\omega(u)=k \\ \varpi \mid u \implies N(\varpi) \leq w }} \Big)
 V \Big( \frac{N(n)}{N} \Big)  V \Big ( \frac{N(nu)}{C} \Big ) W \Big ( \frac{N(n r u)}{X} \Big ).
\end{align} 
 
 \subsection*{Case 1: $N \geq X^{1/2+1/20}$}
 Consulting Remark \ref{type1rangerem} we see that we in the range
 where Corollary \ref{cor:typeI} is non-trivial. Thus Corollary \ref{cor:typeI} and Remark \ref{rhoepsilon} guarantee a small fixed $\delta_0>0$ 
 such that we have (uniformly in $|\ell| \leq X^{1/100}$),
\begin{align} \label{case1final}
\mathscr{B}_{k,\ell}(N) 
:=\delta_{\ell = 0} \cdot \frac{(2\pi)^{2/3}}{3 \Gamma(\tfrac 23)} & \sum_{\substack{r,u,n \in \mathbb{Z}[\omega] \\ u,n  \equiv 1 \pmod{3}}} 
\frac{\alpha_r \mu^2(n r u) \big ( \frac{n r u}{|n r u|} \big )^{\ell}}{N(n r u)^{1/6}} \Big( \delta_{ \substack{\omega(u)=k \\ \varpi \mid u \implies N(\varpi) \leq w }} \Big) \nonumber \\
& \times V \Big( \frac{N(n)}{N} \Big) V \Big ( \frac{N(n u)}{C} \Big ) W \Big ( \frac{N(n r u)}{X} \Big )+O_{\varepsilon}(X^{5/6-\delta_0}). 
\end{align}
We can drop the condition $\delta_{\ell = 0}$ since if $\ell \neq 0$ then the sum over $n$ majorised by the error term in the above display. 

\subsection*{Case 2: $N<X^{1/2+1/20}$}
Suppose we are given a squarefree $1 \neq u \in \mathbb{Z}[\omega]$ satisfying $u \equiv 1 \pmod{3}$, 
and such that all prime factors of $u$ have norm $\leq w$. 
Then, given
a prime $\pi | u$, there is a unique factorisation $u = b \pi d$ such that all the prime factors of $d$ (resp. $b$)
satisfy $< \pi$ (resp. satisfy $> \pi)$. See Remark \ref{idealorder} for the ordering $<$ on ideals.
Since $N(r) \asymp X/C \ll X^{1/3+\varepsilon}$  and $N(ru) \asymp X/N \gg X^{1/2-1/20}$,
we are guaranteed a unique prime $\pi \in \mathbb{Z}[\omega]$  such that $u = b \pi d$ with $N(rb) < X^{13/30}$ and $N(r b \pi) \geq X^{13/30}$.
Hence
\begin{align} \label{Bkldef3}
\mathscr{B}_{k,\ell}(N):=\sum_{\substack{r,n \in \mathbb{Z}[\omega] \\ n \equiv 1 \pmod{3}}} \alpha_r   V \Big( \frac{N(n)}{N} \Big)   
& \sum_{\substack{ u \in \mathbb{Z}[\omega] \\ u \equiv 1 \pmod{3}} } \tilde{g}_{\ell}(n r u) V \Big ( \frac{N(nu)}{C} \Big ) W \Big ( \frac{N(n r u)}{X} \Big ) \nonumber \\
& \times \Big(\sum_{\substack{ \pi \in \mathbb{Z}[\omega]  \\  \pi \text{ prime} \\ N(\pi) \leq w \\ \pi \equiv 1 \pmod{3}  }} \hspace{0.2cm} 
\sum_{\substack{u=b \pi d \\ \omega(u)=k \\  b,d \equiv 1 \pmod{3} \\ N(rb)<X^{13/30} , N(rb \pi) \geq X^{13/30} \\  \varpi \mid b \implies w \geq N(\varpi) \geq N(\pi)  \\ \varpi \mid d \implies N(\varpi) < N(\pi)}}    
1 \Big).
\end{align} 
We introduce smooth partitions of unity in the $N(r),N(b)$ and $N(d)$ in \eqref{Bkldef3}.
Then \eqref{twistmult} and Mellin inversion imply that
\begin{equation} \label{DRBdyadic}
\mathscr{B}_{k,\ell}(N)=\sum_{D,R,B \text{ dyadic}} \hspace{0.2cm}  \sum_{\substack{(k_1,k_2) \in (\mathbb{Z}_{\geq 0})^2 \\ k_1+k_2=k-1 }} 
\mathscr{B}_{(k_1,k_2),\ell}(N,D,R,B),
\end{equation}
where
\begin{align} \label{Bkldef5}
  & \mathscr{B}_{(k_1,k_2),\ell}(N,D,R,B) =\frac{1}{(2\pi i)^2} \int_{-i \infty}^{i \infty} \int_{-i \infty}^{i \infty} C^{s} X^{v} \widetilde{V}(s) \widetilde{W}(v)
 \nonumber \\ & \times \sum_{\substack{\pi \in \mathbb{Z}[\omega] ,  \pi \text{ prime} \\ N(\pi) \leq w , \pi \equiv 1 \pmod{3} }}
 \frac{\tilde{g}_{\ell}(\pi)}{N(\pi)^{s+v}}  \Big ( \sum_{\substack{j,h\in \mathbb{Z}[\omega] \\ j,h \equiv 1 \pmod{3}}}  
\beta_j(\pi,k_1;s,v,\ell) \gamma_h(\pi,k_2;s, v,\ell)  \tilde{g}(jh) \Big ) ds dv,
\end{align}
and
\begin{align}
\beta_{j}(\pi,k_1;s,v,\ell)&:=\frac{\mu^2(j)}{N(j)^{v}} \overline{\Big(\frac{\pi}{j} \Big)_3} \Big(\frac{j}{|j|} \Big)^{\ell} \sum_{\substack{j=rb \\ \omega(b)=k_1  \\ r,b \equiv 1 \pmod{3},  
(j,\pi)=1 \\  N(j) < X^{13/30}, N(j \pi) \geq X^{13/30} \\ \varpi \mid b \implies w \geq N(\varpi) \geq N(\pi)}}  
\alpha_r V \Big( \frac{N(r)}{R} \Big) V \Big(\frac{N(b)}{B} \Big) N(b)^{-s}; \label{betawt} \\
\gamma_{h}(\pi,k_2;s,v,\ell)&:=\frac{\mu^2(h)}{N(h)^{s+v}} \overline{\Big( \frac{\pi}{h} \Big)_3} \Big( \frac{h}{|h|} \Big)^{\ell} \sum_{\substack{h=nd \\ \omega(d)=k_2 
\\  n,d \equiv 1 \pmod{3}, 
 (h,\pi)=1 \\ \varpi \mid d \implies N(\varpi) < N(\pi) }} V \Big(\frac{N(n)}{N} \Big) V \Big( \frac{N(d)}{D} \Big). \nonumber
\end{align}
We write 
\begin{equation}
\beta_{j}(B, R, \pi,k_1; s, v,\ell) := \beta_j(\pi,k_1; s, v,\ell) \quad \text{and} \quad \gamma_h(N, D, k_2, \pi; s, v,\ell) := \gamma_h(\pi,k_2; s,v,\ell),
\end{equation}
when we care to emphasise the dyadic ranges $B, R$ and $N, D$ that are present in the definitions of $\boldsymbol{\beta}$ 
and $\boldsymbol{\gamma}$ respectively. 
 
For each given $\pi \in \mathbb{Z}[\omega]$ prime, the 
sum over $j$ and $h$ in  \eqref{Bkldef5} is empty unless
\begin{equation} \label{RBcond}
X^{13/30}/(1000 N(\pi)) \leq RB \leq 1000 X^{13/30},
\end{equation}
and $N D R B \asymp X / N(\pi)$.
Thus $N D \gg X^{17/30}$. Since $N < X^{1/2 + 1/20}$, we must have $D \gg X^{1/15-1/20 }$
whenever the sum over $j,h$ in \eqref{Bkldef5} is non-zero. We now write
\begin{equation} \label{daggerstar}
\mathscr{B}_{(k_1,k_2),\ell}(N, D, R, B)=\mathscr{B}^{\dagger}_{(k_1,k_2),\ell}(N, D, R, B)+\mathscr{B}^{\star}_{(k_1,k_2),\ell}(N, D, R, B),
\end{equation}
where $\mathscr{B}^{\dagger}_{(k_1,k_2),\ell}$ corresponds the part of \eqref{Bkldef5}
with $N(\pi) \leq (\log D)^{A}$, and $\mathscr{B}^{\star}_{(k_1,k_2),\ell}$
corresponds to $(\log D)^{A}<N(\pi) \leq w$.

\subsubsection*{Treatment of $\mathscr{B}^{\dagger}_{(k_1,k_2),\ell}(N, D, R, B)$}
Since $(\alpha_r)$ is supported only on $w$-rough elements of $\mathbb{Z}[\omega]$,
the factorisation $j=rb$ occurring in the definition of the sequence $\beta_j$ is unique.
Thus $|\beta_j(\pi,k_1;s,v,\ell) | \leq 1$ for $\Re s, \Re v = 0$. 

On the other hand, the sequence $\gamma_h(N,D,\pi,k_2;s,v,\ell)$ is sparse
when $N(\pi)$ is on log-power scales.
We pause the proof to illustrate this 
in the following Lemma. We also make the crude observations that
$\gamma_h$ is supported on $h \in \mathbb{Z}[\omega]$ 
with $N(h) \asymp N D$,
and also satisfies $|\gamma_h(\pi,k_2; s, v,\ell)| \leq 2^{\omega(h)}$ for $\Re s, \Re v = 0$. 

\begin{lemma} \label{sparse}
Let $N, D, A \geq 10$, $k_2 \in \mathbb{Z}_{\geq 0}$, 
and let $\pi \in \mathbb{Z}[\omega]$ be a prime that satisfies $\pi \equiv 1 \pmod{3}$ and $N(\pi) \leq (\log D)^A$. Then
  $$
  \sum_{h \in \mathbb{Z}[\omega]} |\gamma_h(N, D, \pi,k_2; s, v,\ell)|^2 \ll_A (N D)^{o(1)} \cdot N D^{1 - K / A},
  $$
  with $K > 0$ a small absolute constant. 
\end{lemma}
\begin{proof}[Proof of Lemma \ref{sparse}]
We first refine our bound for $\| \boldsymbol{\gamma} \|_{\infty}$. We have 
\begin{equation*}
|\gamma_{h}(\pi,k_2;s,v,\ell)| \leq 2^{\omega(h)} \cdot  \mathbf{1}_{h \in \mathscr{U}_{\pi}} \ll  2^{\frac{2 \log N(h)}{\log \log N(h)}} \cdot \mathbf{1}_{h \in \mathscr{U}_{\pi}}, 
\end{equation*}
where  $\mathscr{U}_{\pi}$ is the set of squarefree integers of the form $n d$ with $n , d \equiv 1 \pmod{3}$, $(nd,\pi)=1$, $N(n) \asymp N$, $N(d) \asymp D$,
and such that all of the prime factors of $d$ have norm $<N(\pi)$. 
Observe that $d$ has necessarily  $\geq \frac{\log D}{100 \log N(\pi)}$ (say) prime factors.
Therefore
\begin{equation} \label{gammaU}
  \sum_{h \in \mathbb{Z}[\omega]} |\gamma_{h}(N,D,\pi,k_2; s, v,\ell)|^2 \ll (N D)^{\frac{4}{\log\log (N D)}} \cdot |\mathscr{U}_{\pi}|. 
\end{equation}

Let $\rho>0$ be chosen later. We have 
\begin{align} \label{mathscrU}
|\mathscr{U}_{\pi}| & \ll (ND) \sum_{\substack{d \equiv 1 \pmod{3} \\ \varpi | d \implies N(\varpi) < N(\pi) \\ \omega(d) \geq \frac{\log D}{100 \log N(\pi)}}} \frac{\mu^2(d)}{N(d)} \nonumber \\ 
& \leq (ND) \cdot \exp \Big ({ -  \frac{\rho \log D}{100 \log N(\pi)}} \Big ) \sum_{\substack{d \equiv 1 \pmod{3} \\ \varpi | d \implies N(\varpi) < N(\pi)}} \frac{\mu^2(d) e^{\rho \omega(d)}}{N(d)} \nonumber \\ 
& \leq (ND) \cdot \exp \Big ({ - \frac{\rho \log D}{100 \log N(\pi)}} \Big ) \prod_{N(\varpi)<N(\pi)} \Big (1 + \frac{e^{\rho}}{N(\varpi)} \Big )  \nonumber \\
& \ll (ND) \cdot \exp \Big ( 2 e^{\rho} \log \log N(\pi) - \frac{\rho \log D}{100 \log N(\pi)} \Big ) \nonumber   \\
&  \ll (ND) \cdot \exp \Big (2 e^{\rho} \log N(\pi) - \frac{\rho \log D}{100 \log N(\pi)} \Big ) \nonumber  \\
&  \ll (ND) \cdot \exp \Big (2 A e^{\rho} \log \log D - \frac{\rho \log D}{100 A \log \log D} \Big ).
\end{align}
We choose
\begin{equation*}
 \rho: = \log \log D -  1000 \log \log \log D.
\end{equation*}
Thus \eqref{mathscrU} implies that 
\begin{equation} \label{Usize}
|\mathscr{U}_{\pi}| \ll_A N D^{1-1/(1000A)}  
\end{equation}
for all $D$ (hence $X$) sufficiently large.
Thus \eqref{gammaU} now implies the Lemma.
\end{proof}
We now resume the proof of Proposition \ref{prop:avgtypeI}.
We now use \eqref{twistmult}, the Cauchy-Schwarz inequality, Heath-Brown's cubic large sieve (Theorem \ref{cubicHBloc})
Lemma \ref{sparse}, and the conditions \eqref{RBcond}, $NDRB \asymp X/N(\pi)$ and $D \gg X^{1/15-1/20}$ to obtain
\begin{align*} 
&  |\mathscr{B}^{\dagger}_{(k_1,k_2),\ell}(N,D,R,B)|  \\
  & \ll_A X^{o(1)} \hspace{-0.4cm} \sum_{\substack{\pi \in \mathbb{Z}[\omega] \\  \pi \text{ prime} \\ N(\pi) \leq (\log D)^{A}  \\ \pi \equiv 1 \pmod{3} }} 
\Big (
(RB)^{1/2} \Big(RB+ND+(RBND)^{2/3}  \Big)^{1/2}  N^{1/2} D^{1/2-1/(500A)} \Big )  \\
& \ll_A X^{5/6-1/(1000000A)},
\end{align*}
say. We include a redundant main term of size that is absorbed by the error term
i.e. we can write
\begin{align} \label{Bkldef6}
& \mathscr{B}^{\dagger}_{(k_1,k_2),\ell}(N,D,R,B) \nonumber \\
&=\frac{(2\pi)^{2/3}}{3 \Gamma(\tfrac 23)} \sum_{\substack{r,u,n \in \mathbb{Z}[\omega] \\ n \equiv 1 \pmod{3}}}   
 \frac{\alpha_r \mu^2(nru) (\frac{nru}{|nru|})^{\ell} }{N(nru)^{1/6}} V \Big( \frac{N(r)}{R} \Big)  V \Big( \frac{N(n)}{N} \Big) 
 V \Big ( \frac{N(nu)}{C} \Big ) W \Big ( \frac{N(n r u)}{X} \Big ) \nonumber \\
& \times \Big(\sum_{\substack{ \pi \in \mathbb{Z}[\omega]  \\  \pi \text{ prime} \\ N(\pi) \leq (\log D)^{A} \\ \pi \equiv 1 \pmod{3}  }}  \sum_{\substack{u=b \pi d \\ \omega(b)=k_1, \omega(d)=k_2 \\ b,d \equiv 1 \pmod{3} \\ N(rb)<X^{13/30} \\ N(rb \pi) \geq X^{13/30} \\  \varpi \mid b \implies w \geq N(\varpi) \geq N(\pi)  \\ \varpi \mid d \implies N(\varpi) < N(\pi)}}    
V \Big( \frac{N(b)}{B} \Big) V \Big(\frac{N(d)}{D} \Big) \Big)+O_{A}(X^{5/6-1/(1000000A)}).
\end{align} 

\subsubsection*{Treatment of $\mathscr{B}^{\star}_{k,\ell}(\cdots)$}
Recall that $(\log X)^{A} \leq w \leq X^{\varepsilon}$, and that 
$(\log X)^{A}<N(\pi) \leq w$. We reassemble the integral in the $v$-variable 
in \eqref{Bkldef5}, and recover the smooth weight $W(N(jh \pi)/X)$. 
By Lemma \ref{le:primes} and Lemma \ref{convolutionlem} we have 
$\boldsymbol{\beta}(R,B;\pi;s,\ell)  \in \mathcal{C}_{\eta}(RB,N(\pi))$ (from \eqref{betawt}) (after re-scaling by an appropriate absolute non-zero constant)
for all $\eta>100/A$.
We then apply Proposition \ref{prop:broad} and see that there is a $\rho(\varepsilon) \in (0,\frac{1}{10000})$ such that
\begin{align}  \label{Bkldef7}
& \mathscr{B}^{\star}_{(k_1,k_2),\ell}(N,D,R,B) \nonumber \\
&=\frac{(2\pi)^{2/3}}{3 \Gamma(\tfrac 23)} \sum_{\substack{r,u,n \in \mathbb{Z}[\omega] \\ n \equiv 1 \pmod{3}}}   
 \frac{\alpha_r \mu^2(nru) (\frac{nru}{|nru|})^{\ell} }{N(nru)^{1/6}} V \Big( \frac{N(r)}{R} \Big)  V \Big( \frac{N(n)}{N} \Big) 
 V \Big ( \frac{N(nu)}{C} \Big ) W \Big ( \frac{N(n r u)}{X} \Big ) \nonumber \\
& \times \Big(\sum_{\substack{ \pi \in \mathbb{Z}[\omega]  \\  \pi \text{ prime} \\  (\log D)^{A}<N(\pi) \leq w \\ \pi \equiv 1 \pmod{3}  }}  \sum_{\substack{u=b \pi d \\ \omega(b)=k_1,\omega(d)=k_2 \\ b,d \equiv 1 \pmod{3} \\ N(rb)<X^{13/30} \\ N(rb \pi) \geq X^{13/30} \\  \varpi \mid b \implies w \geq N(\varpi) \geq N(\pi)  \\ \varpi \mid d \implies N(\varpi) < N(\pi)}}    
V \Big( \frac{N(b)}{B} \Big) V \Big(\frac{N(d)}{D} \Big) \Big) +E,
\end{align}
where
\begin{align*}
E&=O_{A,\varepsilon} \Big(  \log \log w \Big(  \frac{X^{5/6}}{w^{9/20}}+X^{5/6-\rho(\varepsilon)} \Big)+\frac{X^{5/6}}{(\log X)^{A/2}}  \nonumber \\
& \qquad \qquad \qquad + X^{83/120+o(1)} w^{1/4}
+ X^{47/60+o(1)} w^{3/4} 
+ X^{17/60+(13/30)(1+100/A)+o(1)} w^{1/2} \Big).
\end{align*}
Note that both the error terms in \eqref{Bkldef6} and \eqref{Bkldef7} are uniform with respect to $\ell, k_1$ and $k_2$.
\subsection*{Conclusion}
After combining \eqref{Bkldef6} and \eqref{Bkldef7} in \eqref{daggerstar}, 
we obtain an asymptotic expression for  $\mathscr{B}_{(k_1,k_2),\ell}(N,D,R,B)$
for each dyadic value of $N$ satisfying $N<X^{1/2+1/20}$. 
We reassemble the sum over $(k_1,k_2) \in (\mathbb{Z}_{\geq 0})^2$
(satisfying $k_1+k_2=k-1$),
as well as the 
partitions of unity in $N(b)$,
$N(d)$ and $N(r)$ in \eqref{DRBdyadic}.
We then collapse the weights in the main term back to $\big(\delta_{ \substack{\omega(u)=k \\ \varpi \mid u \implies N(\varpi) \leq w }} \big)$,
and obtain an asymptotic expression for $\mathscr{B}_{k,\ell}(N)$
for each dyadic value $N$ satisfying $N<X^{1/2+1/20}$.
Recall that \eqref{case1final} gives an 
asymptotic expression 
for $\mathscr{B}_{k,\ell}(N)$
for each dyadic value $N$ satisfying $N \geq X^{1/2+1/20}$.
We combine these two results
in \eqref{eq:yt2}, 
and reassemble the partition of unity over $N(n)$.
Note that the reassembly of partitions of unity and the sums over $k_i$ do not overwhelm 
the error terms (one only has losses of $O((\log X)^{10})$ say.
Inserting this asymptotic expression 
into \eqref{eq:yt2}, and noting that
$$
\sum_{k \geq 0} (-1)^{k} \Big( \mathbf{\delta}_{\substack{\omega(u)=k \\ \varpi | u \implies N(\varpi) \leq w}} \Big) \mu^2(nru)
=\mu^2(nru) \Big( \mathbf{\delta}_{ \varpi | u \implies N(\varpi) \leq w} \Big) \mu(u),
$$
as well as \eqref{roughmobius},
we obtain the result.
\end{proof}
\section{Combinatorial decompositions}
We will use the following combinatorial decomposition.
\begin{lemma} \label{le:comb}
  Let $W : \mathbb{R} \rightarrow \mathbb{R}$ be a smooth function compactly supported in $(0, C)$.
  Let $(s(n))_{n \in \mathbb{Z}[\omega]}$ be a sequence satisfying $|s(n)| \leq 1$ and 
  have support on squarefree $n$ satisfying $n \equiv 1 \pmod{3}$.
   Then 
  for $2 \leq w \leq C X^{1/3} \leq z$ we have
  \begin{align*}
 & \sum_{\substack{ \varpi \equiv 1 \pmod{3} \\ N(\varpi)>z }} s(\varpi) W \Big ( \frac{N(\varpi)}{X} \Big )  = - \frac{1}{2} \sum_{\substack{\varpi_1, \varpi_2 \equiv 1 \pmod{3} \\ N(\varpi_1), N(\varpi_2) > z}} s(\varpi_1 \varpi_2) W \Big (\frac{N(\varpi_1 \varpi_2)}{X} \Big ) \\ & + \sum_{k \geq 0} \frac{(-1)^k}{k!} \sum_{\substack{w < N(\varpi_1) , \ldots , N(\varpi_k) \leq z \\ \forall i : \varpi_i \equiv 1 \pmod{3} \\ c \equiv 1 \pmod{3} \\ k=0 \implies c \neq 1  \\ \varpi | c \implies N (\varpi) > w  }} s(c \varpi_1 \ldots \varpi_k) W \Big ( \frac{N(c \varpi_1 \ldots \varpi_k)}{X} \Big ) + O(\sqrt{X}).
  \end{align*}
\end{lemma}
\begin{proof}
We assume that $\Re s>1$ throughout this proof.
We have
  $$
  \zeta_{> z}(s) := \prod_{\substack{N (\varpi) > z \\ \varpi \equiv 1 \pmod{3}}} \Big ( 1 - \frac{1}{N(\varpi)^s} \Big )^{-1},
  $$
and
  \begin{align} \label{eq11}
    & \sum_{L \geq 1} \frac{1}{L} \sum_{\substack{N(\varpi)>z \\ \varpi \equiv 1 \pmod{3}}} \frac{1}{N(\varpi)^{L s}} =  \log \zeta_{> z}(s) =
    \log (1 + (\zeta_{>z}(s) - 1)) \nonumber \\
    & = (\zeta_{>z}(s) - 1) - \frac{1}{2} \cdot (\zeta_{>z}(s) - 1)^2 + \sum_{j \geq 3} \frac{(-1)^{j+1}}{j} \cdot (\zeta_{> z}(s) - 1)^{j}. 
  \end{align}
  Furthermore,
  $$
  \zeta_{>z}(s) = \zeta_{>w}(s) \prod_{\substack{w < N(\varpi) \leq z \\ \varpi \equiv 1 \pmod{3}}} \Big ( 1 - \frac{1}{N(\varpi)^s} \Big ),
  $$
  where
  \begin{equation} \label{eulerproduct}
  \zeta_{>w}(s) := \prod_{\substack{N(\varpi) > w \\ \varpi \equiv 1 \pmod{3}}} \Big ( 1 - \frac{1}{N(\varpi)^s} \Big )^{-1} = \sum_{\substack{c \equiv 1 \pmod{3} \\ \varpi | c \implies N(\varpi) > w}} \frac{1}{N(c)^s}.
  \end{equation}
  The equation \eqref{eulerproduct} is valid since every $c \equiv 1 \pmod{3}$ has a unique factorisation
  $c=\varpi_1 \ldots \varpi_k$ with $\varpi_i \equiv 1 \pmod{3}$ for $i=1,\ldots,k$. Expand the product
  $$
  \prod_{\substack{w < N(\varpi) \leq z \\ \varpi \equiv 1 \pmod{3}}} \Big (1 - \frac{1}{N(\varpi)^s} \Big ) =1+\sum_{k \geq 1} \frac{(-1)^k}{k!} \sum_{\substack{w < N(\varpi_1) , \ldots , N(\varpi_k) \leq z \\ \forall i : \varpi_i \equiv 1 \pmod{3} \\ \varpi_i \text{ all distinct}}} \frac{1}{N(\varpi_1 \ldots \varpi_k)^s}.
  $$
 Therefore
  \begin{equation} \label{eq22}
\zeta_{> z}(s)-1= \sum_{k \geq 0} \frac{(-1)^k}{k!} \sum_{\substack{w < N(\varpi_1), \ldots, N(\varpi_k) \leq z \\ \forall \varpi_i \equiv 1 \pmod{3}  \\
 \varpi_i \text{ all distinct} \\ c \equiv 1 \pmod{3} \\
k=0 \implies c \neq 1 \\  \varpi | c \implies N(\varpi) > w  }} \frac{1}{N(c \varpi_1 \ldots \varpi_k)^s}.
\end{equation}
  Substitution of \eqref{eq22} into \eqref{eq11} gives
  \begin{align} \label{sumidentity}
    \sum_{L \geq 1} \frac{1}{L} \sum_{\substack{N(\varpi)>z \\ \varpi \equiv 1 \pmod{3}}} \frac{1}{N(\varpi)^{L s}} = \sum_{k \geq 0} & \frac{(-1)^k}{k!} \sum_{\substack{w < N(\varpi_1), \ldots, N(\varpi_k) \leq z \\  \forall \varpi_i \equiv 1 \pmod{3} \\ \varpi_i \text{ all distinct} \\ c \equiv 1 \pmod{3} \\ k=0 \implies c \neq 1  \\ \varpi | c \implies N(\varpi) > w }} \frac{1}{N(c \varpi_1 \ldots \varpi_k)^s}    \\
    & - \frac{1}{2} \cdot (\zeta_{>z}(s) - 1)^2 + \sum_{j \geq 3} \frac{(-1)^{j+1}}{j} \cdot (\zeta_{> z}(s) - 1)^{j}.  \nonumber
  \end{align}
  The result follows from a comparison of coefficients. Observe that that the total contribution from terms $N(\varpi)^{k} \leq X$ with $k \geq 2$ on the left side
  of \eqref{sumidentity} is $O(\sqrt{X})$.
  Since $z > C X^{1/3}$ and $W$ is compactly supported in $(0,C)$, we see that the contribution from all terms $(\zeta_{>z}(s) - 1)^{j}$ with $j \geq 3$ is zero.
  Notice that $s(c \varpi_1 \ldots \varpi_k)$ is zero if $c \varpi_1 \ldots \varpi_k$ is not squarefree by hypothesis, 
  so we can drop the requirement that the $\varpi_i$ are all distinct. 
  \end{proof}

\section{Proof of Theorems \ref{thm:main} and \ref{thm:main3} } \label{mainsec1}
We first record a useful Lemma due to Polymath that classifies 
the Type-I, Type-II and Type-III information that occurs in the proof of
our main theorems.

\begin{lemma} \label{polymath} \emph{\cite[Lemma 3.1]{Polymath}}
  Given an integer $n \geq 1$ and $\tfrac{1}{10} \leq \sigma < \tfrac 12$,
  let $t_1, \ldots, t_n$ be non-negative real numbers such that $t_1 + \ldots + t_n = 1$. Then at least one of the following three statement holds:
  \begin{enumerate}
  \item[(Type-I)] There is an $i \in [1,n]$ such that $t_i \geq \tfrac 12 + \sigma$;
  \item[(Type-II)] There is a partition $\{1, \ldots, n\} = S \cup T$ such that
    $$
    \frac{1}{2} - \sigma<\sum_{i \in S} t_i \leq \sum_{i \in T} t_i<\frac{1}{2} + \sigma;
    $$
  \item[(Type-III)] There exists distinct $i,j,v \in [1,n]$ such that $2 \sigma \leq t_i \leq t_j \leq t_v \leq \frac{1}{2} - \sigma$ and
    $$
    t_i + t_j, t_j + t_v, t_v + t_i \geq \frac{1}{2} + \sigma.
    $$
    \end{enumerate}
Furthermore, if $\sigma>1/6$, then the Type-III alternative can't occur.
\end{lemma}
\begin{proof}[Proof of Theorems \ref{thm:main} and \ref{thm:main3}]
We first explain some initial manipulations.
\subsection*{Initial reduction}
For any rational prime $p \equiv 1 \pmod{3}$ we have 
$$
\frac{S_p}{2 \sqrt{p}} = \Re \tilde{g}(\varpi),
$$
where $\varpi \in \mathbb{Z}[\omega]$ is a prime such that
$\varpi \equiv 1 \pmod{3}$ and
$p = \varpi \overline{\varpi}$.
The number of primes $\varpi \equiv 1 \pmod{3}$ for
which $N(\varpi)$ is not prime is $O(\sqrt{X})$. Such primes
are those that lie over rational primes $p \equiv 2 \pmod{3}$. 
To prove Theorem \ref{thm:main} it suffices to estimate
the quantity
$$
\sum_{\varpi \equiv 1 \pmod{3}} \tilde{g}(\varpi) W \Big ( \frac{N(\varpi)}{X} \Big ) .
$$
Observe that \eqref{cuberel} implies that
$$
\tilde{g}(\varpi)^{3} = - \frac{\varpi}{|\varpi|}.
$$
Thus
$$
\tilde{g}(\varpi)^{k} = (-1)^{\ell} \Big ( \frac{\varpi}{|\varpi|} \Big )^{\ell}
\times  \begin{cases}
  \tilde{g}(\varpi) & \text{ if } k \equiv 1 \pmod{3} \text{ with } \ell = \frac{k - 1}{3} \\
  \overline{\tilde{g}(\varpi)} & \text{ if } k \equiv 2 \pmod{3} \text{ with } \ell = \frac{k + 1}{3} \\
  1 & \text{ if } k \equiv 0 \pmod{3} \text{ with } \ell = \frac{k}{3}
  \end{cases}.
$$
In particular, Theorem \ref{thm:main3} with $k \equiv 0 \pmod{3}$ follows directly from the assumption of the Generalized Riemann Hypothesis.

To establish Theorem \ref{thm:main3}, it suffices to show that 
$$
\sum_{\varpi \equiv 1 \pmod{3}} \tilde{g}(\varpi) \Big ( \frac{\varpi}{|\varpi|} \Big )^{\ell} W \Big ( \frac{N(\varpi)}{X} \Big )  = o \Big ( \frac{X^{5/6}}{\log X} \Big ),
$$
as $X \rightarrow \infty$ and uniformly in $0 < |\ell| \leq X^{1/100}$.
To prove both Theorem \ref{thm:main} and Theorem \ref{thm:main3} simultaneously it is enough to estimate
$$
\sum_{\varpi \equiv 1 \pmod{3}} \tilde{g}(\varpi) \Big ( \frac{\varpi}{|\varpi|} \Big )^{\ell} W \Big ( \frac{N(\varpi)}{X} \Big ),
$$
to a precision better than $o(X^{5/6} / \log X)$. For $c \equiv 1 \pmod{3}$ define 
$$
\tilde{g}_{\ell} (c) := \tilde{g}(c) \Big ( \frac{c}{|c|} \Big )^{\ell}. 
$$
Let $\varepsilon \in (0,10^{-6})$ be fixed.
Let 
\begin{equation*}
w:=X^{\varepsilon} \quad \text{and} \quad z:= X^{1/3+\varepsilon}.
\end{equation*}
By Lemma \ref{le:comb} we have
\begin{align} 
 & \sum_{\varpi \equiv 1 \pmod{3}} \tilde{g}_{\ell}(\varpi) W \Big ( \frac{N(\varpi)}{X} \Big )  = - \frac{1}{2} \sum_{\substack{\varpi_1, \varpi_2 \equiv 1 \pmod{3} \\ N(\varpi_1), N(\varpi_2) > z}} \tilde{g}_{\ell}(\varpi_1 \varpi_2) W \Big (\frac{N(\varpi_1 \varpi_2)}{X} \Big ) \label{eq:bgpf1} \\ & + \sum_{k \geq 0} \frac{(-1)^k}{k!} \sum_{\substack{(\varpi_1, \ldots, \varpi_k,c) \in \mathcal{S}(w, z)}} \tilde{g}_{\ell}(\varpi_1 \ldots \varpi_k c) W \Big ( \frac{N(\varpi_1 \ldots \varpi_k c)}{X} \Big ) + O(\sqrt{X}),
 \label{eq:bgpf2}
\end{align}
where $\mathcal{S}(w, z)$ denotes the set of tuples $(\varpi_1, \ldots, \varpi_k,c)$ with $k \geq 0$ such that
\begin{itemize}
\item $\varpi_1, \ldots, \varpi_k$ are primes congruent to $1 \pmod{3}$ (when $k \geq 1$);
\item For all $1 \leq i \leq k$ we have $w<N(\varpi_i) \leq z$ (when $k \geq 1$); 
\item $c$ is $w$-rough, $c \equiv 1 \pmod{3}$, and $k=0 \implies c \neq 1$.
\end{itemize}
When $k=0$, the sum is understood just to be over the variable $c$.

Let $\xi \in (0,10^{-6})$ be a small fixed quantity to be decided at a later point in the proof
(it will ultimately depend on $\varepsilon$).

\begin{remark}
Uniformity of error terms in $\ell$ is not an issue when deploying Type II/III estimates 
(i.e. Proposition \ref{prop:narrow} and Proposition \ref{prop:broad}). This is because \eqref{twistmult} is applied to $\tilde{g}_{\ell}(ab)$, 
and the dependence on $\ell$ is absorbed into the
coefficients $\boldsymbol{\alpha}$ and $\boldsymbol{\beta}$ that satisfy $\|\boldsymbol{\alpha} \|_{\infty}, \| \boldsymbol{\beta} \|_{\infty} \leq 1$. 
The dependence on $\ell$ issue emanates from the
application of the average Type-I estimate in Proposition \ref{prop:avgtypeI}.
\end{remark}

\subsection*{Sum on the right side side of \eqref{eq:bgpf1}}
We introduce a smooth partition of unity on each of the $N(\varpi_i)$ to evaluate the (Type-II) sum over $N(\varpi_1), N(\varpi_2) > z$. 
Thus it is sufficient to estimate
\begin{equation} \label{eq:f}
\mathscr{F}_{\ell}(X,P_1,P_2;z):=\sum_{\substack{\varpi_1, \varpi_2 \equiv 1 \pmod{3} \\ N(\varpi_1), N(\varpi_2) > z}} \tilde{g}_{\ell}(\varpi_1 \varpi_2) W \Big (\frac{N(\varpi_1 \varpi_2)}{X} \Big ) V \Big ( \frac{N(\varpi_1)}{P_1} \Big  ) V \Big ( \frac{N(\varpi_2)}{P_2} \Big ),
\end{equation}
for all dyadic partitions $(P_1,P_2)$ that satisfy $z/2 \leq P_1, P_2 \leq 2X$ and $P_1 P_2 \asymp X$. When $z/2 \leq \min \{P_1,P_2 \} \leq X^{1/2 -\xi}$ 
we can apply Proposition \ref{prop:broad} with $\pi=1$, and $\eta>0$ arbitrarily small and fixed by Lemma \ref{le:primes2}
(the only requirement is that $\eta>100 \log \log X/\log X$).
Thus there exists $\delta_0(\xi,\varepsilon)>0$
such that 
\begin{align} \label{broad2est}
\mathscr{T}_{\ell}(X,P_1,P_2,z)&=
\frac{(2\pi)^{2/3}}{3 \Gamma(\tfrac 23)} \sum_{\substack{N(\varpi_1), N(\varpi_2) > z}} \frac{\mu^2(\varpi_1 \varpi_2) \big ( \frac{\varpi_1 \varpi_2}{|\varpi_1 \varpi_2|} \big )^{\ell}}{N(\varpi_1 \varpi_2)^{1/6}} V \Big ( \frac{N(\varpi_1)}{P_1} \Big ) V \Big ( \frac{N(\varpi_2)}{P_2} \Big ) \nonumber   \\
& \times W \Big ( \frac{N(\varpi_1 \varpi_2)}{X} \Big ) + O_{\xi,\varepsilon} ( X^{5/6-\delta_0(\xi,\varepsilon)} ),
\quad \text{when} \quad z/2 \leq \min \{P_1,P_2 \} \leq X^{1/2 - \xi}.
\end{align}

When $X^{1/2 - \xi} \leq P_1, P_2 \leq X^{1/2 + \xi}$, we appeal
to Proposition \ref{prop:narrow}. In particular, the smooth coefficients here are supported on $z=X^{1/3+\varepsilon}>X^{\varepsilon}$
-rough integers.
We obtain
\begin{align*}
\mathscr{F}_{\ell}(X,P_1,P_2;z) \ll \frac{X}{(\varepsilon \log X)^{3/2}} \frac{1}{\sqrt{\min (P_1,P_2) }} &
+\frac{X^{5/6}}{(\varepsilon \log X)^2}, \\
& \text{when} \quad X^{1/2 - \xi} \leq P_1, P_2 \leq X^{1/2 + \xi},
\end{align*}
where the implied constant is absolute. We can include a redundant main term that is majorised by the error term i.e.
\begin{align} \label{narrow2est}
\mathscr{F}_{\ell}(X,P_1,P_2;z)&=\frac{(2\pi)^{2/3}}{3 \Gamma(\tfrac 23)} \sum_{\substack{N(\varpi_1), N(\varpi_2) > z}} \frac{\mu^2(\varpi_1 \varpi_2) \big ( \frac{\varpi_1 \varpi_2}{|\varpi_1 \varpi_2|} \big )^{\ell}}{N(\varpi_1 \varpi_2)^{1/6}} V \Big ( \frac{N(\varpi_1)}{P_1} \Big ) V \Big ( \frac{N(\varpi_2)}{P_2} \Big ) W \Big ( \frac{N(\varpi_1 \varpi_2)}{X} \Big ) \nonumber  \\
& +O \Big(\frac{X}{(\varepsilon \log X)^{3/2}} \cdot \frac{1}{\sqrt{\min (P_1,P_2) }}  \Big) + O \Big ( \frac{X^{5/6}}{(\varepsilon \log X)^{2}} \Big ), \nonumber  \\
& \quad \quad \quad \quad \quad \quad \quad \quad \quad \quad \quad \quad \quad \quad \quad \quad \text{when} \quad  X^{1/2 - \xi} \leq  P_1,P_2  \leq X^{1/2 + \xi}.
\end{align}
Since $P_1 P_2 \asymp X$ there are $O(\xi \log X)$ choices of $P_1, P_2$
in the narrow range $X^{1/2 - \xi} \leq P_1, P_2 \leq X^{1/2 +\xi}$.
Summing \eqref{broad2est} and \eqref{narrow2est} over all possible dyadic tuples $(P_1,P_2)$ gives
\begin{align} \label{firsttermfinal}
 \sum_{\substack{\varpi_1, \varpi_2 \equiv 1 \pmod{3} \\ N(\varpi_1), N(\varpi_2) > z}} & \tilde{g}_{\ell}(\varpi_1 \varpi_2) W \Big (\frac{N(\varpi_1 \varpi_2)}{X} \Big ) \nonumber \\
 &=\frac{(2\pi)^{2/3}}{3 \Gamma(\tfrac 23)} \sum_{\substack{\varpi_1, \varpi_2 \equiv 1 \pmod{3} \\ N(\varpi_1), N(\varpi_2) > z}} \frac{\mu^2(\varpi_1 \varpi_2) \big ( \frac{\varpi_1 \varpi_2}{|\varpi_1 \varpi_2|} \big )^{\ell}}{N(\varpi_1 \varpi_2)^{1/6}} W \Big ( \frac{N(\varpi_1 \varpi_2)}{X} \Big ) \nonumber  \\
 &+O \Big(\frac{X^{3/4+\xi/2}}{(\varepsilon \log X)^{3/2} } \Big)+ O \Big ( \frac{\xi X^{5/6}}{\varepsilon^2 \log X} \Big ) + O_{\xi,\varepsilon}(X^{5/6-\delta_1(\xi,\varepsilon)}),
\end{align}
for any fixed $0<\delta_1(\xi,\varepsilon)<\delta_0(\xi,\varepsilon)$. 

\subsection*{Sum in \eqref{eq:bgpf2}}
For each $0 \leq k \leq 1/\varepsilon$, we analyse the sum
\begin{equation} \label{initialdecomp}
\sum_{\substack{ (\varpi_1, \ldots, \varpi_k,c) \in \mathcal{S}(w, z) }} \tilde{g}_{\ell}(\varpi_1 \ldots \varpi_k c) W \Big ( \frac{N(\varpi_1 \ldots \varpi_k c)}{X} \Big ).
\end{equation}

We insert a smooth partition of unity in $N(c)$ and each $N(\varpi_i)$ for $i=1,\ldots,k$ in \eqref{initialdecomp}. 
Thus its suffices to 
estimate 
\begin{align} \label{eq:starter}
\mathscr{S}_{\ell}(P_1, \ldots, P_{k + 1}) := \sum_{\substack{ (\varpi_1, \ldots, \varpi_k,c) \in \mathcal{S}(w, z) }} \tilde{g}_{\ell}(\varpi_1 \ldots \varpi_k c) W \Big ( \frac{N(\varpi_1 \ldots \varpi_k c)}{X} \Big ) V \Big(\frac{N(c)}{P_{k+1}} \Big) \prod_{i=1}^{k} V \Big( \frac{N(\pi_i)}{P_i} \Big),
\end{align}
for all dyadic partitions $H=(P_1,\ldots,P_{k+1})$ satisfying $P_1 \ldots P_{k+1} \asymp X$, $w/2 \leq P_i \leq 2z$ for all $i=1,\ldots,k$,
and $P_{k+1} \geq 1/2$. Our goal will be to show that $\mathscr{S}(P_1, \ldots, P_{k + 1})$ is asymptotically equal to 
(either for individual tuples $(P_1, \ldots, P_{k + 1})$ or on average)
\begin{align*}
& \mathscr{M}_{\ell}(P_1, \ldots, P_{k + 1}) \\ & = \frac{(2\pi)^{2/3}}{3 \Gamma(\tfrac 23)} \sum_{(\varpi_1, \ldots, \varpi_k, c) \in \mathcal{S}(w, z)} \frac{\mu^2(\varpi_1 \ldots \varpi_k c) \big ( \frac{\varpi_1 \ldots \varpi_k c}{|\varpi_1 \ldots \varpi_k c|} \big )^{\ell}}{N(\varpi_1 \ldots \varpi_k c)^{1/6}} W \Big ( \frac{N(\varpi_1 \ldots \varpi_k c)}{X} \Big ) V \Big ( \frac{N(c)}{P_{k + 1}} \Big ) \prod_{i = 1}^{k} V \Big ( \frac{N(\varpi_i)}{P_i} \Big ). 
\end{align*}
For a given $(P_1,\ldots,P_{k+1})$, 
let 
\begin{equation} \label{vardefn2}
t_i:=\frac{\log P_i}{\log (P_1 \ldots P_{k + 1})} \geq 0 \quad \text{for} \quad i=1,\ldots,k+1. 
\end{equation}
We necessarily have 
\begin{align}
t_1 + \ldots &+ t_{k + 1} = 1   \label{unitsoln}; \\
\frac{\log w}{\log X}  \leq t_i  \leq \frac{\log z}{\log X} & \quad \text{for} \quad i=1, \ldots,k. \label{prime} 
\end{align}
We now apply Lemma \ref{polymath} with choice $\sigma:=1/6-\xi$ to decompose the proof into cases. 

\subsubsection*{Narrow Type-III sums}
In this case we necessarily have $k \geq 2$, and
\begin{equation} \label{eq:condi1}
 \exists \text{ three distinct indices } i,j,\ell \in \{1,\ldots,k+1\} \text{ such that } t_i, t_j, t_{\ell} \in (\tfrac 13 - 2 \xi, \tfrac 13 + \xi).
\end{equation}
In particular, either
  \begin{enumerate}
  \item $\exists$ an index $i$ such that $t_i \in [\tfrac 13, \tfrac 13  + \xi)$, or
  \item we have $t_i,t_j,t_{\ell} \in (\tfrac 13-2 \xi,\tfrac 13)$.
\end{enumerate}
The sum over all dyadic partitions $(P_1, \ldots, P_{k + 1})$ for which there exists an index $i$ such that $t_i \in [\tfrac 13, \tfrac 13 + \xi)$ (and two additional indices $j,\ell$ such that $t_j, t_{\ell} \in (\tfrac 13  - 2 \xi, \tfrac 13 + \xi)$) is
\begin{equation} \label{narrowquant}
\leq (k + 1)! \cdot \sum_{\substack{ X^{1/3} \leq P \leq X^{1/3 + \xi} \\ P \text{ dyadic} }}
\sup_{\substack{ \|\boldsymbol{\alpha} \|_{\infty},  \|\boldsymbol{\beta} \|_{\infty} \leq 1 \\ \boldsymbol{\beta} \in \mathcal{C}_{\eta}(P,w)}} \Big | \sum_{\substack{a,b \equiv 1 \pmod{3} \\ \pi | a \implies N(\pi) > w \\ \pi \mid b \implies N(\pi) > w  \\ N(a) \asymp X/P , N(b) \asymp P }} \alpha_a \beta_b \tilde{g}_{\ell}(a b) W \Big ( \frac{N(a b)}{X} \Big ) \Big |,
\end{equation}
where $\eta>0$ is arbitrarily small and fixed by Lemma \ref{le:primes2} (the only requirement is that $\eta >100 \log \log X/(\varepsilon \log X)$).
Notice that the factor $(k + 1)! = k! \cdot (k + 1)$ arises from the fact that there are $k + 1$ ways of choosing the first index $i$ for which $P_i \in [\tfrac 13, \tfrac 13 + \xi)$ (and this index becomes our $P$) and there are $k!$ ways of representing $a$ as a product of the remaining $k$ variables. 
Application of Proposition \ref{prop:narrow} shows that \eqref{narrowquant} is
\begin{align*}
& \ll (k + 1)! \sum_{\substack{ X^{1/3} \leq P \leq X^{1/3 + \xi} \\ P \text{ dyadic} }}
\Big ( \frac{1}{(\varepsilon \log X)^{3/2}} \cdot \frac{X}{\sqrt{P}} + \frac{X^{5/6}}{(\varepsilon \log X)^2} \Big ) \\
& \ll  (k+1)!  \Big( \frac{X^{5/6}}{(\varepsilon \log X)^{3/2}}   + \frac{\xi X^{5/6}}{\varepsilon^2 \log X}  \Big),
\end{align*}
where the implied constants are absolute.

We now handle the remaining case in which $t_i,t_j,t_{\ell} \in (\tfrac 13-2 \xi, \tfrac 13)$. We group together two variables coming from
the indices $i$ and $j$ say. 
We sum over all dyadic partitions $(P_1, \ldots, P_{k + 1})$ for which $t_i, t_j, t_{\ell} \in (\tfrac 13 - 2 \xi, \tfrac 13)$.
This sum is
\begin{equation} \label{narrowquant2}
\leq (k + 1)! \cdot \sum_{\substack{ X^{2/3 - 4 \xi} \leq U \leq X^{2/3} \\ U \text{ dyadic} }} \sup_{\substack{ \|\boldsymbol{\alpha} \|_{\infty}, \|\boldsymbol{\beta} \|_{\infty} \leq 1 \\ \boldsymbol{\beta} \in \mathcal{C}_{\eta}(X/U,w)  }} \Big | \sum_{\substack{a,b \equiv 1 \pmod{3} \\ \pi | a \implies N(\pi) > w \\ \pi \mid b \implies N(\pi) > w \\ N(a) \asymp U , N(b) \asymp X / U }} \alpha_a \beta_b \tilde{g}_{\ell}(a b) W \Big ( \frac{N(a b)}{X} \Big ) \Big |,
\end{equation}
where $\eta>0$ is arbitrarily small and fixed by Lemma \ref{le:primes2} (the only requirement is that $\eta>100 \log \log X/(\varepsilon \log X)$).
The factor $(k + 1)! = (k - 1)! \cdot 2 \binom{k + 1}{2}$ arises from the fact that there are $2 \binom{k + 1}{2}$ ordered choices
of $i$ and $j$ such that $t_i, t_j \in (\tfrac 13 - 2 \xi, \tfrac 13)$, and 
$(k - 1)!$ ways of representing $a$ as the product of the remaining $k - 1$ variables. 
Applying Proposition \ref{prop:narrow} and arguing in a similar way to the above shows that \eqref{narrowquant2} is 
$$
\ll (k+1)! \Big( \frac{X^{5/6}}{(\varepsilon \log X)^{3/2}}+\frac{\xi X^{5/6}}{\varepsilon^2 \log X} \Big),
$$
where the implied constant is absolute.

Combining the two cases we conclude that
\begin{align*}
\sum_{\substack{(P_1, \ldots, P_{k + 1}) \\ \eqref{eq:condi1} \text{ holds}}} \mathscr{S}(P_1, \ldots, P_{k + 1})& = \sum_{\substack{(P_1, \ldots, P_{k + 1}) \\ \eqref{eq:condi1} \text{ holds}}} \mathscr{M}(P_1, \ldots, P_{k + 1}) +O \Big( \frac{(k+1)! X^{5/6}}{(\varepsilon \log X)^{3/2}} \Big)  \\
&+ O \Big (\frac{(k+1)! \xi X^{5/6}}{\varepsilon^2 \log X} \Big ). 
\end{align*}
Notice that the main term is absorbed by the error term in this case. 

\subsubsection*{Narrow Type-II sums}
In this case we necessarily have $k \geq 1$,
and
\begin{equation} \label{eq:condi2}
\exists  \text{ a partition $S \cup T = \{1, \ldots, k + 1\}$ such that } 
\tfrac 12 - \xi < \sum_{i \in S} t_i \leq \sum_{j \in T} t_j < \tfrac 12 + \xi. 
\end{equation}
The contribution of all such $(P_1,\ldots,P_{k+1})$ is
\begin{equation} \label{narrowquant3}
\leq  (k + 1)! \sum_{\substack{ X^{1/2 - \xi} \leq U \leq X^{1/2} \\ U \text{ dyadic} }} \sup_{\substack{\|\boldsymbol{\alpha} \|_{\infty},  \|\boldsymbol{\beta} \|_{\infty} \leq 1 \\ \boldsymbol{\beta} \in \mathcal{C}_{\eta}(U,w) }} \Big | \sum_{\substack{a , b \equiv 1 \pmod{3} \\ \pi | a \implies N(\pi) > w \\ \pi \mid b \implies N(\pi) > w \\ N(a) \asymp X/U , N(b) \asymp U}} \alpha_a \beta_b \tilde{g}_{\ell}(a b) W \Big ( \frac{N(a b)}{X} \Big ) \Big |,
\end{equation}
where $\eta>0$ is arbitrarily small and fixed by Lemma \ref{le:primes2} (the only requirement is that $\eta>100 \log \log X/(\varepsilon \log X)$).
The term $(k + 1)!$ arises from the fact that for each $1 \leq i \leq k$, there are $ i! \binom{k + 1}{i}$ ordered choices for the set $S$ containing $i$ elements,
and there are $(k + 1 - i)!$ ways of representing $b$ as a product of the remaining $k + 1 - i$ variables indicated by the set $T$. Applying Proposition \ref{prop:narrow}, we see that \eqref{narrowquant3} is 
\begin{align*}
& \ll (k + 1)! \sum_{\substack{ X^{1/2 - \xi} \leq U \leq X^{1/2} \\ U \text{ dyadic} }} \Big ( \frac{1}{(\varepsilon \log X)^{3/2}} \cdot \frac{X}{\sqrt{U}} + \frac{X^{5/6}}{(\varepsilon \log X)^2} \Big ) \\
& \ll (k+1)!  \Big(  \frac{X^{3/4+\xi/2}}{(\varepsilon \log X)^{3/2}}+ \frac{\xi X^{5/6}}{\varepsilon^2 \log X}  \Big),
\end{align*}
where the implied constants are absolute.

In particular, 
\begin{align*}
\sum_{\substack{(P_1, \ldots, P_{k + 1}) \\ \eqref{eq:condi2} \text{ holds}}} \mathscr{S}(P_1, \ldots, P_{k + 1})& = \sum_{\substack{(P_1, \ldots, P_{k + 1}) \\ \eqref{eq:condi2} \text{ holds}}} \mathscr{M}(P_1, \ldots, P_{k + 1}) + O \Big( \frac{(k+1)! X^{3/4+\xi/2}}{(\varepsilon \log X)^{3/2}} \Big) \\
&+O \Big ( \frac{(k+1)! \xi X^{5/6}}{\varepsilon^2 \log X} \Big ),
\end{align*}
where the main term is absorbed by the error term. 

\subsubsection*{Remaining ranges}
We now consider all of the remaining dyadic partitions $(P_1,\ldots,P_{k+1})$ one by one. For each remaining tuple $(P_1, \ldots, P_{k + 1})$ we will show that \begin{equation} \label{eq:goali}
\mathscr{S}_{\ell}(P_1, \ldots, P_{k + 1}) = \mathscr{M}_{\ell}(P_1, \ldots, P_{k + 1}) + O_{A, \xi, \varepsilon} \Big ( \frac{X^{5/6}}{\log^A X} \Big ),
\end{equation}
for any given $A > 10$ (depending on $\varepsilon>0$). Recall that $k \leq 1/\varepsilon$. Since there are at most $(\log X)^{k}$ dyadic partitions $(P_1 , \ldots, P_{k + 1})$  
satisfying $P_1 \cdots P_{k+1} \asymp X$, we can sum over
the error term in \eqref{eq:goali} without overwhelming the main term. 
Notice that each of the remaining configurations of $(P_1, \ldots, P_{k + 1})$ now fall into either of two cases:
\begin{enumerate}
\item $\exists i \in \{1,\ldots,k+1\}$ such that $t_i \geq \tfrac 23 - \xi$;
\item Or $\exists$ a partition $S \cup T=\{1, \ldots, k + 1\}$ such that
  \begin{equation} \label{eq:condi3}
  \frac{1}{3} + \xi \leq \sum_{i \in S} t_i \leq \frac{1}{2} - \xi \leq \frac{1}{2} + \xi \leq \sum_{j \in T} t_j \leq \frac{2}{3} - \xi. 
  \end{equation}
\end{enumerate}
If $\exists i \in \{1,\ldots,k+1\}$ such that $t_{i} \geq \tfrac 23 - \xi$, then $i=k+1$ by \eqref{prime}.  This corresponds to the $c$ variable
appearing in $\mathscr{S}_{\ell}(P_1, \ldots, P_{k + 1})$ in \eqref{eq:starter}. 
After applying Proposition \ref{prop:avgtypeI} (average Type-I estimate) we obtain \eqref{eq:goali} uniformly in $|\ell| \leq X^{1/100}$. 
If the second alternative holds, 
then \eqref{eq:goali} follows from Proposition \ref{prop:broad} (broad Type-II estimate). 

\subsubsection*{Assembly}
Summing over all dyadic partitions $(P_1,\ldots,P_{k+1})$ we obtain
\begin{align} \label{secondsum}
\sum_{0 \leq k \leq 1 / \varepsilon} & \frac{(-1)^{k}}{k!} \sum_{\substack{ (\varpi_1, \ldots, \varpi_k,c) \in \mathcal{S}(w, z) }} \tilde{g}_{\ell}(\varpi_1 \ldots \varpi_k c) W \Big ( \frac{N(\varpi_1 \ldots \varpi_k c)}{X} \Big ) \nonumber \\
 & = \sum_{0 \leq k \leq 1 / \varepsilon} \frac{(-1)^{k}}{k!} \sum_{\substack{(\varpi_1, \ldots, \varpi_k, c) \in \mathcal{S}(w, z)}} \frac{\mu^2(\varpi_1 \ldots \varpi_k c) \big ( \frac{\varpi_1 \ldots \varpi_k c}{|\varpi_1 \ldots \varpi_k c|} \big )^{\ell}}{N(\varpi_1 \ldots \varpi_k c)^{1/6}} W \Big ( \frac{N(\varpi_1 \ldots \varpi_k c)}{X} \Big ) 
\nonumber  \\
 & + O \Big (\frac{\xi X^{5/6}}{\varepsilon^4 \log X} \Big ) +O \Big( \frac{X^{5/6}}{\varepsilon^{7/2} (\log X)^{3/2}} \Big) +
 O \Big( \frac{X^{3/4+\xi/2}}{\varepsilon^{7/2} (\log X)^{3/2}} \Big)+O_{A, \xi, \varepsilon} \Big ( \frac{X^{5/6}}{\log^{A-\varepsilon^{-1}} X} \Big ),
\end{align}
uniformly in $|\ell| \leq X^{1/100}$. We now drop the third error term in \eqref{secondsum} because it is majorised by the second one.
Combining \eqref{secondsum} and \eqref{firsttermfinal} in \eqref{eq:bgpf1}--\eqref{eq:bgpf2}, and then applying Lemma \ref{le:comb} (in the reverse direction, and to the symbol $\mu^2(\cdot)
(\cdot \cdot)^{\ell} $) gives
\begin{align} \label{finalconc}
  \sum_{\varpi \equiv 1 \pmod{3}} & \tilde{g}_{\ell}(\varpi) W \Big ( \frac{N(\varpi)}{X} \Big ) = \frac{(2\pi)^{2/3}}{3 \Gamma(\tfrac 23)} \sum_{\varpi \equiv 1 \pmod{3}} \frac{\big ( \frac{\varpi}{|\varpi|} \big )^{\ell}}{N(\varpi)^{1/6}} W \Big ( \frac{N(\varpi)}{X} \Big ) \nonumber \\ & + O \Big (\frac{\xi X^{5/6}}{\varepsilon^4 \log X} \Big )+
  O \Big( \frac{X^{5/6}}{\varepsilon^{7/2}(\log X)^{3/2}} \Big) + O_{A, \xi, \varepsilon} \Big ( \frac{X^{5/6}}{\log^{A-\varepsilon^{-1}} X} \Big ).
   + O_{\xi,\varepsilon}(X^{5/6-\delta_1(\xi,\varepsilon)}),
\end{align}
uniformly in $|\ell | \leq X^{1/100}$.
After choosing $\xi = \varepsilon^{1000}$ and $A=\varepsilon^{-1000}$  (say), 
the error terms in \eqref{finalconc} are $O \big( \frac{\varepsilon X^{5/6}}{\log X} \big)$ as $X \rightarrow \infty$.  
We conclude by noticing that
$$
\sum_{\varpi \equiv 1 \pmod{3}} \frac{1}{N(\varpi)^{1/6}} W \Big ( \frac{N(\varpi)}{X} \Big ) \sim \int_{0}^{\infty} W(x) x^{-1/6} dx \cdot \frac{X^{5/6}}{\log X} \quad \text{as}
\quad X \rightarrow \infty,
$$
and for $\ell \neq 0$,
$$
\sum_{\varpi \equiv 1 \pmod{3}} \frac{\big ( \frac{\varpi}{|\varpi|} \big )^{\ell}}{N(\varpi)^{1/6}} W \Big ( \frac{N(\varpi)}{X} \Big ) = o \Big ( \frac{X^{5/6}}{\log X} \Big ) \quad 
\text{as} \quad X \rightarrow \infty,
$$
uniformly in $|\ell| \leq X^{1/100}$. This proves Theorem \ref{thm:main} and Theorem \ref{thm:main3}.
\end{proof}

\section{Proof of Theorem \ref{thm:cor}} \label{mainsec2}
\begin{proof}[Proof of Theorem \ref{thm:cor}]
We expand $f$ in a Fourier series
$$
f(x) = \sum_{k \in \mathbb{Z}} \widehat{f}(k) e(k x).
$$
For $p \equiv 1 \pmod{3}$, 
$$
f(\theta_{p}) =  \sum_{k \in \mathbb{Z}} \widehat{f}(k) e(k \theta_{p}) = \sum_{k \in \mathbb{Z}} \widehat{f}(k) \tilde{g}(\varpi)^{k},
$$
where $\varpi$ is a prime in $\mathbb{Z}[\omega]$ such that
$p = \varpi \overline{\varpi}$. Therefore
$$
\sum_{p \equiv 1 \pmod{3}} f(\theta_p) W \Big ( \frac{p}{X} \Big ) 
$$
is equal to
$$
\widehat{f}(0) \sum_{p \equiv 1 \pmod{3}} W \Big ( \frac{p}{X} \Big ) + \sum_{0 < |k| \leq X^{1/100}} \widehat{f}(k) \Big ( \sum_{\varpi \equiv 1 \pmod{3}} \tilde{g}(\varpi)^{k} W \Big ( \frac{N(\varpi)}{X} \Big ) + O(\sqrt{X}) \Big ) + O_{A}(X^{-A})
$$
for any given $A > 10$. We now appeal to Theorem \ref{thm:main} and Theorem \ref{thm:main3} to see that the sum over $k \neq 0$  is equal to
$$
(\widehat{f}(1) + \widehat{f}(-1)) \cdot \frac{(2\pi)^{2/3}}{3 \Gamma(\tfrac 23)} \cdot \int_{0}^{\infty} W(x) x^{-1/6} dx \cdot \frac{X^{5/6}}{\log X} + o \Big ( \frac{X^{5/6}}{\log X} \Big ),
$$
as claimed.
\end{proof}

\providecommand{\bysame}{\leavevmode\hbox to3em{\hrulefill}\thinspace}
\providecommand{\MR}{\relax\ifhmode\unskip\space\fi MR }
\providecommand{\MRhref}[2]{  \href{http://www.ams.org/mathscinet-getitem?mr=#1}{#2}
}
\providecommand{\href}[2]{#2}

\appendix
\section{Appendix} \label{appendix}
This table completes the computation in \cite[Table~III]{Pat1} where the values $k_j(E)$ were computed for all $1 \leq j \leq 27$. We supplement \cite[Table~III]{Pat1} by also computing $k_j(T)$ and $k_j(P)$ for all $1 \leq j \leq 27$. We do not require these computations in any of our proofs.

\begin{center}
\begin{tabular}{|l|l|l|l|l|}  
 \hline
 $j$  & $d_j(\mu)$ & $k_j(E)$ & $k_j(P)$ & $k_j(T)$ \\ 
 \hline
 $1$ & $\tau(\mu)$                                                           	& $1$        &   $4$  &  $1$ \\ 
 \hline
 $2$ & $\tau(\mu) \check{e}(\omega \mu)$                                   	& $19$      &   $5$  &   $2$ \\ 
 \hline
 $3$ & $\tau(\mu) \check{e}(-\omega \mu)$                                  	&  $10$      &   $6$ &    $3$ \\ 
 \hline
 $4$ & $\tau(\omega \mu)$                                               	&   $7$       &   $7$ &    $6$ \\
 \hline
 $5$ & $\tau(\omega \mu) \check{e}(-\mu)$                      	&   $23$     &    $9$ &   $4$ \\
 \hline
 $6$ & $\tau(\omega \mu) \check{e}(\mu)$                        	 &  $13$     &   $8$  &   $5$ \\
 \hline
 $7$ & $\tau(\omega^2 \mu)$                                            	 &   $4$      &    $1$ &   $9$ \\
 \hline
 $8$ & $\tau(\omega^2 \mu) \check{e}(-\mu)$                   	 &  $14$     &    $3$  &  $7$ \\
 \hline
  $9$ & $\tau(\omega^2 \mu) \check{e}(\mu)$                    	  &  $22$    &    $2$  &  $8$ \\
  \hline
 $10$ & $\omega \tau_2(\omega \mu) \check{e}(\mu)$     	  &   $3$     &    $14$ & $11$ \\
 \hline
 $11$  & $\omega \tau_2(\omega^2 \mu)$    	  &   $12$   &    $17$ & $12$ \\
 \hline
 $12$  & $\omega \tau_2(\omega \mu)$                            	  &    $11$   &    $11$ & $10$ \\
 \hline
 $13$  & $\omega \tau_2(\mu) \check{e}(\omega^2 \mu)$ 	  &     $6$    &    $10$ & $14$ \\
 \hline
 $14$ & $\omega \tau_2(\mu) \check{e}(-\omega \mu)$        &    $8$    &     $13$ & $15$ \\
 \hline
 $15$ & $\tau_2(\omega^2 \mu) \check{e}(-\omega^2 \mu)$ &   $24$   &     $16$ & $13$ \\
 \hline
 $16$ & $\tau_2(\mu) \check{e}(-\mu)$                                   &   $25$   &     $18$ & $17$ \\
 \hline
 $17$  & $\tau_2(\mu)$                                                            &   $17$   &     $12$ & $18$ \\
 \hline
 $18$  & $\tau_2(\mu) \check{e}(\mu)$                                    &   $27$   &     $15$ & $16$ \\
 \hline
 $19$  & $\omega^2 \tau_1(\omega^2 \mu) \check{e}(\mu)$    &   $2$     &     $22$ & $20$ \\
 \hline
 $20$  & $\omega^2 \tau_1(\omega \mu)$                               &   $21$   &     $21$ & $21$ \\
 \hline
 $21$  & $\omega^2 \tau_1(\omega^2 \mu)$                            &   $20$   &     $26$ & $19$ \\
 \hline
 $22$  & $\omega^2 \tau_1(\mu) \check{e}(\omega \mu)$        &   $9$     &     $23$ & $23$ \\
 \hline
 $23$  & $\omega^2 \tau_1(\omega \mu) \check{e}(\omega^2 \mu)$ & $5$ &  $19$ & $24$ \\
 \hline
 $24$  & $\omega^2 \tau_1(\omega^2 \mu) \check{e}(\omega^2 \mu)$ & $15$ & $27$ & $22$ \\
 \hline
 $25$  & $\tau_1(\mu) \check{e}(-\mu)$                                      &         $16$    &  $24$ & $26$ \\
 \hline
 $26$  & $\tau_1(\mu)$	                                                          &         $26$    &   $20$ & $27$ \\
 \hline
 $27$  & $\tau_1(\mu) \check{e}(\mu)$                                       &         $18$     &   $25$ & $25$ \\
 \hline
\end{tabular}
\end{center}

\par

We note that, 
\begin{align*}
  \langle \{ (j, k_j(E)) : 1 \leq j \leq 27 \} \rangle & \simeq C_{2}^{12} \\
  \langle \{ (j, k_j(P)) : 1 \leq j \leq 27 \} \rangle & \simeq C_{3}^{9}\\
  \langle \{ (j, k_j(T)) : 1 \leq j \leq 27 \} \rangle & \simeq C_{3}^8.
\end{align*}
These isomorphisms are easily seen from the table by following the cycle structure. 
The exponents $12$ in $C_2^{12}$ be explained by noticing that the forms $j = 1, 17, 26$ are invariant under $E$ and
all the other elements are of order two, giving us $\frac{27 - 3}{2} = 12$ generators.
Likewise the exponent $8$ in $C_3^8$ can be explained by noticing that the forms with $j = 1,2,3$ are invariant
and there are $\frac{27 - 3}{3} = 8$ remaining generators all of order $3$. Finally the exponent $9$ appears in the
case of $k_j(P)$ because no forms is left invariant by $P$ and $P$ is of order three.

\end{document}